\documentclass[11pt]{article}
\usepackage{amsmath, amssymb, amscd, amsthm, amsfonts}
\usepackage{graphicx}
\usepackage{float}    
\usepackage{hyperref}
\usepackage{algorithm}  
\usepackage{algpseudocode} 
\usepackage{makecell}
\usepackage{amsmath}
\usepackage{enumitem}
\usepackage{graphicx}    % For including images
\usepackage{subfigure}
\usepackage{subcaption}  % For subfigures
\usepackage{caption}     % For customizing captions (optional)
\usepackage{float}       % For controlling figure placement (optional)

\setcellgapes[t]{0pt}
\makegapedcells 
\usepackage{enumerate}
\usepackage{amsfonts,amsmath,amsthm,amssymb}
\newtheorem{corollary}{Corollary}

\usepackage[textsize=tiny]{todonotes} % To-do notes

\usepackage[most]{tcolorbox}
\tcbuselibrary{theorems}

\usepackage[most]{tcolorbox}
\usepackage{framed}
\usepackage{xcolor}
\usepackage{bm}
\colorlet{shadecolor}{orange!15}

\tcbset{
    mymathstyle/.style={
        colback=cyan!10,    % 背景颜色
        colframe=cyan!50,   % 边框颜色
        sharp corners,      % 去掉圆角
        boxrule=0.5pt,      % 边框粗细
        left=2pt,           % 左侧内边距
        right=2pt,          % 右侧内边距
        top=2pt,            % 上边内边距
        bottom=2pt          % 下边内边距
    }
}

\usepackage{mdframed}  % For creating framed environments

\newmdenv[
  backgroundcolor=yellow!20,
  linecolor=red,
  linewidth=2pt,
  roundcorner=10pt,
  shadow=true,
  nobreak=true
]{warningbox}

\usepackage{empheq}
\newcommand{\highlight}[1]{%
    \colorbox{cyan!10}{$\displaystyle #1$}%
}

\usepackage{xcolor}

\definecolor{defscol}{HTML}{ecd8d7} %For definitions
\definecolor{asumscol}{HTML}{ecd8d7} %For Assumptions

\definecolor{rmkscol}{HTML}{313160} %For remarks
\definecolor{exmscol}{HTML}{e04b52} %For examples

\definecolor{lemscol}{HTML}{2c3943} %For Lemmes
\definecolor{thmscol}{HTML}{595765} %For Theorems
\definecolor{prpscol}{HTML}{9c98b1} %For proposition
\definecolor{corscol}{HTML}{dfd9fd} %For corrolaries

\definecolor{clmscol}{HTML}{165c58} %For claims
\definecolor{facscol}{HTML}{28a8a1} %For facts

\newtcbtheorem[number within=section]{mydefinition}{Definition}
{
    enhanced,
    frame hidden,
    titlerule=0mm,
    toptitle=1mm,
    bottomtitle=1mm,
    fonttitle=\bfseries\large,
    coltitle=black,
    colbacktitle=defscol!40!white,
    colback=defscol!20!white,
}{defn}

\newtcbtheorem[use counter from=mydefinition]{myassumption}{Assumption}
{
    enhanced,
    frame hidden,
    titlerule=0mm,
    toptitle=1mm,
    bottomtitle=1mm,
    fonttitle=\bfseries\large,
    coltitle=black,
    colbacktitle=asumscol!40!white,
    colback=asumscol!20!white,
}{asum}

\newtcbtheorem[use counter from=mydefinition]{mytheorem}{Theorem}
{
    enhanced,
    frame hidden,
    titlerule=0mm,
    toptitle=1mm,
    bottomtitle=1mm,
    fonttitle=\bfseries\large,
    coltitle=black,
    colbacktitle=thmscol!40!white,
    colback=thmscol!20!white,
}{thm}

\newenvironment{thmpf}{
	{\noindent{\it \textbf{Proof for Theorem.}}}
	\tcolorbox[blanker,breakable,left=5mm,parbox=false,
    before upper={\parindent15pt},
    after skip=10pt,
	borderline west={1mm}{0pt}{thmscol!40!white}]
}{
    \textcolor{thmscol!40!white}{\hbox{}\nobreak\hfill$\blacksquare$} 
    \endtcolorbox
}

\NewDocumentCommand{\thmp}{m+m+m+m}{
    \begin{mytheorem}[label=#2]{#1}{} 
        #3
    \end{mytheorem}

    \begin{thmpf}
        #4
    \end{thmpf}
}

\newtcbtheorem[use counter from=mydefinition]{mylemma}{Lemma}
{
    enhanced,
    frame hidden,
    titlerule=0mm,
    toptitle=1mm,
    bottomtitle=1mm,
    fonttitle=\bfseries\large,
    coltitle=black,
    colbacktitle=lemscol!40!white,
    colback=lemscol!20!white,
}{lem}

\NewDocumentCommand{\lem}{m+m+m}{
    \begin{mylemma}[label=#2]{#1}{}
        #3
    \end{mylemma}
}

\newenvironment{lempf}{
	{\noindent{\it \textbf{Proof for Lemma}}}
	\tcolorbox[blanker,breakable,left=5mm,parbox=false,
    before upper={\parindent15pt},
    after skip=10pt,
	borderline west={1mm}{0pt}{lemscol!40!white}]
}{
    \textcolor{lemscol!40!white}{\hbox{}\nobreak\hfill$\blacksquare$} 
    \endtcolorbox
}

\NewDocumentCommand{\lemp}{m+m+m+m}{
    \begin{mylemma}[label=#2]{#1}{}
        #3
    \end{mylemma}

    \begin{lempf}
        #4
    \end{lempf}
}

\newtcbtheorem[use counter from=mydefinition]{mycorollary}{Corollary}
{
    enhanced,
    frame hidden,
    titlerule=0mm,
    toptitle=1mm,
    bottomtitle=1mm,
    fonttitle=\bfseries\large,
    coltitle=black,
    colbacktitle=corscol!40!white,
    colback=corscol!20!white,
}{cor}

\NewDocumentCommand{\cor}{+m}{
    \begin{mycorollary}{}{}
        #1
    \end{mycorollary}
}

\newenvironment{corpf}{
	{\noindent{\it \textbf{Proof for Corollary.}}}
	\tcolorbox[blanker,breakable,left=5mm,parbox=false,
    before upper={\parindent15pt},
    after skip=10pt,
	borderline west={1mm}{0pt}{corscol!40!white}]
}{
    \textcolor{corscol!40!white}{\hbox{}\nobreak\hfill$\blacksquare$} 
    \endtcolorbox
}

\newtcbtheorem[use counter from=mydefinition]{myproposition}{Proposition}
{
    enhanced,
    frame hidden,
    titlerule=0mm,
    toptitle=1mm,
    bottomtitle=1mm,
    fonttitle=\bfseries\large,
    coltitle=black,
    colbacktitle=prpscol!30!white,
    colback=prpscol!20!white,
}{prop}

\newenvironment{proppf}{
	{\noindent{\it \textbf{Proof for Proposition.}}}
	\tcolorbox[blanker,breakable,left=5mm,parbox=false,
    before upper={\parindent15pt},
    after skip=10pt,
	borderline west={1mm}{0pt}{prpscol!40!white}]
}{
    \textcolor{prpscol!40!white}{\hbox{}\nobreak\hfill$\blacksquare$} 
    \endtcolorbox
}

\newtcbtheorem[use counter from=mydefinition]{myclaim}{Claim}
{
    enhanced,
    frame hidden,
    titlerule=0mm,
    toptitle=1mm,
    bottomtitle=1mm,
    fonttitle=\bfseries\large,
    coltitle=black,
    colbacktitle=clmscol!40!white,
    colback=clmscol!20!white,
}{clm}

\newenvironment{clmpf}{
	{\noindent{\it \textbf{Proof for Claim.}}}
	\tcolorbox[blanker,breakable,left=5mm,parbox=false,
    before upper={\parindent15pt},
    after skip=10pt,
	borderline west={1mm}{0pt}{clmscol!40!white}]
}{
    \textcolor{clmscol!40!white}{\hbox{}\nobreak\hfill$\blacksquare$} 
    \endtcolorbox
}

\newtcbtheorem[use counter from=mydefinition]{myfact}{Fact}
{
    enhanced,
    frame hidden,
    titlerule=0mm,
    toptitle=1mm,
    bottomtitle=1mm,
    fonttitle=\bfseries\large,
    coltitle=black,
    colbacktitle=facscol!40!white,
    colback=facscol!20!white,
}{fact}

\newenvironment{myexample}{
    \tcolorbox[blanker,breakable,left=5mm,parbox=false,
    before upper={\parindent15pt},
    after skip=10pt,
	borderline west={1mm}{0pt}{clmscol!40!white}]
}{
    \textcolor{clmscol!40!white}{\hbox{}\nobreak\hfill$\blacksquare$} 
    \endtcolorbox
}

\newenvironment{myalg}{
    \tcolorbox[blanker,breakable,left=5mm,parbox=false,
    before upper={\parindent15pt},
    after skip=10pt,
	borderline west={1mm}{0pt}{clmscol!40!white}]
}{
    \textcolor{clmscol!40!white}{\hbox{}\nobreak\hfill$\blacksquare$} 
    \endtcolorbox
}

\NewDocumentCommand{\alg}{m+m}{
    \begin{myalg}
	{\noindent{\it \textbf{Algorithm : #1 }}}\\ 
        #2
    \end{myalg}
}

\newenvironment{remark}{
    \par
    \vspace{5pt}
    \begin{minipage}{\textwidth}
        {\par\noindent{\textbf{Remark.}}}
        \tcolorbox[blanker,breakable,left=5mm,
        before skip=10pt,after skip=10pt,
        borderline west={1mm}{0pt}{rmkscol!20!white}]
}{
        \endtcolorbox
    \end{minipage}
    \vspace{5pt}
}

\NewDocumentCommand{\rmkb}{+m}{
    \begin{remark}
        #1
    \end{remark}
}

\newenvironment{sketch}{
    \par
    \vspace{5pt}
    \begin{minipage}{\textwidth}
        {\par\noindent{\textbf{Proof Sketch.}}}
        \tcolorbox[blanker,breakable,left=5mm,
        before skip=10pt,after skip=10pt,
        borderline west={1mm}{0pt}{rmkscol!20!white}]
}{
        \endtcolorbox
    \end{minipage}
    \vspace{5pt}
}

\NewDocumentCommand{\pfsketch}{+m}{
    \begin{sketch}
        #1
    \end{sketch}
}

\title{What is a Sketch-and-Precondition Derivation for Low-Rank Approximation? Inverse Power Error or Inverse Power Estimation?}
\author{Ruihan Xu \thanks{Department of Mathematics, University of Chicago, email:\texttt{ruihanx@uchicago.edu}} \and Yiping Lu \thanks{Industrial Engineering \& Management Science, Northwestern University, email: \texttt{yiping.lu@northwestern.edu} }}
\date{}

\newtheorem{theorem}{Theorem}
\newtheorem{lemma}[theorem]{Lemma}

\newcommand{\yplu}[1]{{\color{orange}{ [\textbf{Yiping:}} #1]}}
\newcommand{\rh}[1]{{\color{red}{ [\textbf{Ruihan:}} #1]}}
\begin{document}

\maketitle

\begin{abstract}
Randomized sketching accelerates large-scale numerical linear algebra by reducing computational complexity. While the traditional sketch-and-solve approach reduces the problem size directly through sketching, the sketch-and-precondition method leverages sketching to construct a computation-friendly preconditioner. This preconditioner improves the convergence speed of iterative solvers applied to the original problem, maintaining accuracy in the full space. Furthermore, the convergence rate of the solver improves at least linearly with the sketch size. Despite its potential, developing a sketch-and-precondition framework for randomized algorithms in low-rank matrix approximation remains an open challenge. We introduce the \emph{Error-Powered Sketched Inverse Iteration} (\textbf{\texttt{EPSI}}) Method via running sketched Newton iteration for the Lagrange form as a sketch-and-precondition variant for randomized low-rank approximation. Our method achieves theoretical guarantees, including a convergence rate that improves at least linearly with the sketch size.
\end{abstract}

\begin{center}
\begin{tikzpicture}
    % node options:
    %   fill=orange: background color
    %   fill opacity=0.5: semi-transparency
    %   text opacity=1: text remains fully opaque
    %   rounded corners: round the corners of the box
    %   inner sep=10pt: padding inside the box
    \node[
        fill=orange,
        fill opacity=0.2,
        text opacity=1,
        rounded corners,
        inner sep=10pt
    ]{
        \parbox{0.8\linewidth}{
            \centering
            \emph{"In ending this book with the subject of preconditioners, 
            we find ourselves at the philosophical center of the scientific 
            computing of the future."} 
            
            \vspace{6pt}
            
            --- \textbf{L. N. Trefethen and D. Bau III}, \texttt{Numerical Linear Algebra} \cite{trefethen2022numerical}
        }
    };
\end{tikzpicture}
\end{center}

\tableofcontents

%\begin{warningbox}
%\textbf{Warning (First Version):} \\We are the first version of this preprint, and we use the most naïve approach to prove the theorem. We believe the proof can be further simplified, and we are actively working on it.
%\end{warningbox}

\section{Introduction}\label{section-introduction}

Randomized Numerical Linear Algebra (RNLA) \cite{halko2011finding,martinsson2020randomized,woodruff2014sketching,mahoney2011randomized,drineas2016randnla,murray2023randomized,avron2010blendenpik} is a rapidly advancing area of matrix computations that has significantly impacted low-rank approximations, iterative methods, projections and etc. This field highlights the power of randomized algorithms as highly effective tools for constructing approximate matrix computations. These algorithms stand out for their simplicity, efficiency, and their ability to yield surprisingly accurate results. Effective randomized least-squares solvers for minimizing $\|Ax-b\|$ ($A \in \mathbb{R}^{m \times n}$, $x \in \mathbb{R}^n$, $b \in \mathbb{R}^m$) often involve constructing a preconditioner using a sketched matrix. Notable methods include \texttt{sketch-and-precondition} approaches \cite{rokhlin2008fast, avron2010blendenpik,spielman2014nearly,meng2014lsrn,lacotte2020effective}, \texttt{iterative sketching} techniques \cite{pilanci2016iterative, ozaslan2019iterative} (with their preconditioning effect formally analyzed in \cite{xu2024randomized}), and \texttt{sketch-and-projection} frameworks \cite{gower2016sketch,gower2021adaptive} (with their preconditioning effect formally analyzed in \cite{derezinski2024sharp}). These approaches often achieve convergence rates that scale at least linearly with the sketch size, and can exhibit even faster convergence when the data matrix has favorable spectral decay properties \cite{derezinski2024sharp}.

In this paper, we mainly focus on developing a \texttt{Sketch-and-Precondition} derivation of randomized low-rank matrix approximation algorithms, which play a central role in data analysis and scientific computing. One of the most significant techniques in this area is \texttt{Randomized SVD} \cite{liberty2007randomized, halko2011algorithm, witten2015randomized, martinsson2011randomized, swartworth2023optimal}, which reduces the problem size by projecting the original matrix onto a lower-dimensional subspace using random sketching methods. The accuracy of the resulting low-rank approximation can be significantly improved through several refinement steps employing subspace iteration techniques \cite{rokhlin2010randomized, musco2015randomized, gu2015subspace,feng2024algorithm}. Such an algorithm operates as a \texttt{sketch-and-solve} approach \cite{sarlos2006improved,clarkson2017low}, where the matrix is first reduced in size using a random sketch and the solution is directly computed in the lower-dimensional space. Unlike \texttt{sketch-and-precondition} methods, it does not enhance the convergence speed of the subspace iteration refinement steps, which could otherwise scale linearly with the sketch size. A natural and important research question is 
\begin{shaded}
    \emph{What is the sketch-and-precondition derivation of randomized low-rank approximation algorithms?}
\end{shaded}
Our paper presents a simple and efficient method, the \emph{Error-Powered Sketched Inverse Iteration} (\texttt{EPSI}) method for finding the top eigenvector for a symmetric matrix $A\in\mathbb{R}^{n\times n}$, which \textbf{applies sketched inverse iteration to the sketching error but not estimated eigenvector}. To establish a sketch-and-precondition framework for developing randomized low-rank approximation algorithms, we reinterpret inverse power iteration for low-rank approximation through the perspective of Newton's method for nonlinear programming \cite{garber2015fast, tapia2018inverse}.  Specifically, we aim to run Newton sketch method \cite{roosta2016sub,roosta2016sub2,bollapragada2019exact,berahas2020investigation} for minimizing the Lagrange form $
F(u,\lambda)=\frac{1}{2}u^\top A u-\frac{\lambda}{2}(u^\top u-1).$ Thus if one has an (easy to compute inverse) approximate $\hat A$ to $A$, we can perform the following iteration
$$
u_{k+1}=u_k-\underbrace{(\hat A-\lambda I)^{-1}}_{\text{approximated }(\nabla^2 F)^{-1}}\underbrace{(A-\lambda I)u_k}_{\nabla F}=(\hat A-\lambda I)^{-1}(A-\hat A)u_k
$$

\begin{comment}

$$
\min_{u,\lambda} \frac{1}{2}u^\top A u-\frac{\lambda}{2}(u^\top u-1)+\frac{c}{8}(u^\top u - 1)^2 \quad \text{for} \quad c\not =0
$$

Set target function
$$F(u)=\hat{A}u-u(\lambda_R-\frac{c}{2}(u^\top u-1)),\ \lambda_R=(u^\top u)^{-1}(u^\top Au)$$

$u_{i+1} = u_i-F'(u_i)^{-1}F(u_i) =u_i - (\hat{A}-\lambda I)^{-1}(A-\lambda I)u_i = (\hat{A}-\lambda I)^{-1}(\hat{A}-A)u_i$
Then from Newton method 
\begin{align*}
u_{i+1} &= u_i - (F'(u))^{-1}F(u)\\
&= u_i - ((I-2u_iu_i^\top)(A-\lambda_RI)+cu_iu_i^\top)^{-1}(A-(\lambda_R-\frac{c}{2}(u_i^\top u_i-1))u_i)
\end{align*}
\end{comment}

The most closely related work to ours is \cite{jin2015robust,allen2016lazysvd}, which employs an approximate solver (e.g., gradient descent) to compute \highlight{\text{the shifted inverse power iteration } (A - \lambda I)^{-1}u}. The shifted inverse power method starting with an initial guess of eigenvalue $\lambda_0$ and eigenvector \( u_0 \), the method iteratively solves the linear system \((A - \lambda_k I) u_{k+1} = u_k\) for \( u_{k+1} \), where \( I \) is the identity matrix.  The method leverages the matrix \((A - \lambda I)^{-1}\), which magnifies the components of eigenvectors associated with eigenvalues close to the shift \( \lambda \), thereby accelerating convergence for those eigenvalues. Recently, \cite{tapia2018inverse} also reinterpreted shifted inverse iteration as Newton’s method.   However, as the accuracy of the estimated eigenvectors improves, the approximate solver must achieve increasingly higher precision, leading to additional computational cost. In contrast to na\"{i}vely approximate the shifted inverse power method, our paper proposes a slightly \highlight{\text{ modified iteration } (\hat{A} - \lambda I)^{-1}(A - \hat{A})u}. In our approach, \textbf{regardless of the choice of \(\hat{A}\) (the approximate solver), the true eigenvector remains the fixed point of the iteration.} A better choice of \(\hat{A}\) only enhances the preconditioning effect but does not dictate the accuracy level needed to achieve convergence. This distinction mirrors the difference between \texttt{sketch-and-solve} and \texttt{sketch-and-precondition} approaches. Our method avoids the need for increasingly precise solvers to improve accuracy, thereby reducing computational overhead.

\begin{table}[h]
\begin{center}
\begin{tabular}{c|c|c|c}
\hline
\hline
                                  & \makecell{\textbf{\small Overparameterized}\\ [-0.6em] \textbf{ \small Least Square }} & \makecell{\textbf{\small Low Rank}\\[-0.6em]\textbf{\small Approximation}} \\ \hline\hline
\textbf{\small Sketch-and-Solve}                  &      \cite{sarlos2006improved,drineas2011faster}         &     \cite{liberty2007randomized,halko2011algorithm}    &    \makecell[l]{{\scriptsize  approximate in geometry-preserving,}\\[-0.6em]{\scriptsize  lower-dimensional space.}}             \\ \hline
\makecell{\textbf{\small Randomized}\\ [-0.6em] \textbf{\small Initialization}}  &               & \cite{gu2015subspace,musco2015randomized}       & \makecell[l]{{\scriptsize Initialize an (iterative) algorithm}\\ [-0.6em]{\scriptsize at a random point to make}\\ [-0.6em]{\scriptsize a favorable trajectory}}                  \\ \hline
\textbf{\small Sketch-and-Precondition}           &       \makecell{\cite{rokhlin2008fast, avron2010blendenpik}\\[-0.6em]
\cite{meng2014lsrn,derezinski2024sharp}}        &   \underline{\textbf{\color{red}This work}}      &     \makecell[l]{{\scriptsize boosting solver convergence while} \\[-0.6em]{\scriptsize  preserving the original problem's accuracy}}           \\ \hline
\end{tabular}
\end{center}
\vspace{-0.2in}
\caption{Different Themes in Randomized Matrix Computation \cite{kireeva2024randomized}. Our work fills the gap that constructing a randomized low-rank approximation algorithm that boost the convergence via perconditioning while preserving the accuracy of the original problem.}
\vspace{-0.3in}
\end{table}

\subsection{Contribution}
\begin{itemize}
\item We propose \emph{Error-Powered Sketched Inverse Iteration} (\texttt{EPSI}), which differs from previous methods such as approximated inverse power iteration \cite{jin2015robust,allen2016lazysvd}, applies sketched
inverse iteration to the sketching error but not the  the estimated eigenvector. In this approach, the true eigenvector remains a fixed point of the Error-Powered iteration regardless of the quality of the randomized embedding. Our Error-Powered iteration ensures that the embedding's accuracy level affects convergence only as a preconditioner, rather than determining it. Our algorithm framework is quite flexible

\item We extend the idea of \texttt{EPSI} to compute the first $k$ singular vectors, preserving the property that the sketched approximate solver serves only as a preconditioner. As a result, our method achieves a convergence rate that depends solely on $\frac{\lambda_k}{\lambda_k-\lambda_{k+1}}$, rather than on $\max\{\frac{\lambda_1}{\lambda_1-\lambda_2},\cdots,\frac{\lambda_k}{\lambda_k-\lambda_{k+1}}\}$, aligns with the results in the literature \cite{li2015convergence,musco2015randomized,allen2016lazysvd}. As discussed in Section~\ref{section:ksvd}, the approach from \cite{allen2016lazysvd} cannot be directly extended to our setting. Hence, we introduce a novel \emph{orthogonalization step} that eliminates dependence on intermediate spectral gaps.

\item Unlike the Newton Sketch method \cite{pilanci2017newton,gower2019rsn}, which requires solving a sketched optimization subproblem at each iteration, we leverage the low-rank structure of the Nyström approximation and employ the Woodbury identity (see Lemma \ref{lemma:woodbury}) to enable fast computation of the inverse in each iteration.

\item Compared to existing works on precondition-type method \cite{crouzeix1994davidson}, our method provides theoretical guarantees, including a convergence rate that improves at least linearly with the sketch size, similar to the improvements achieved by sketch-and-precondition and sketch-and-project techniques for least-squares problems. As far as the authors are known, this is the first result that can be achieved for eigenvalue computation.
\end{itemize}
\subsection{Preliminary}

\paragraph{Power and Inverse Power Methods} Two foundational iterative approaches for approximating eigenvalues are the \texttt{Power Method} and its precondition version \texttt{Shifted Inverse Power Method}. Consider a symmetric matrix $A \in \mathbb{R}^{n \times n}$ with eigenvalues $\lambda_1, \lambda_2, \dots, \lambda_n$ ordered so that $|\lambda_1| > |\lambda_2| \geq \cdots \geq |\lambda_n|$, and associated eigenvectors $v_1, v_2, \dots, v_n$. The Power Method iteratively applies $A$ to a vector to isolate the dominant eigenvalue $\lambda_1$. Starting from an initial nonzero vector $x^{(0)}$, one forms \(
x^{(k+1)} = \frac{A x^{(k)}}{\| A x^{(k)} \|}, \quad k = 0, 1, 2, \dots.\) As $k \to \infty$, $x^{(k)}$ converges to $v_1$ and the rate of convergence is roughly geometric with a factor determined by the ratio $|\lambda_2/\lambda_1|$. If $\lambda_2$ is close in magnitude to $\lambda_1$, the convergence can be slow. To design a precondition algorithm, Shifted Inverse Power Methods can use a shift $\mu$ and consider $(A - \mu I)^{-1}$. Applying the Power Method to this inverse-shifted matrix yields \(x^{(k+1)} = \frac{(A - \mu I)^{-1} x^{(k)}}{\|(A - \mu I)^{-1} x^{(k)}\|}.\) If $\mu$ is chosen close to a particular eigenvalue $\lambda_j$, the iteration effectively “magnifies” the influence of $\lambda_j - \mu$ in the inverse spectrum. As a result, the Inverse Power Method can converge much faster to the eigenvector associated with $\lambda_j$ compared to the standard Power Method. By iteratively adjusting the shift \(\mu\) to approximate \(\lambda_1\), one can achieve a convergence rate governed by a factor substantially smaller than \(|\lambda_2/\lambda_1|\), thereby significantly accelerating the overall convergence. As shown in \cite{tapia2018inverse}, this process can be interpreted as a Newton iteration followed by a normalization step, providing a key motivation for the approach presented in this work.

 \paragraph{Nystr\"{o}m Approximation} The Nystr\"{o}m approximation \cite{bach2013sharp,alaoui2015fast} obtains a low-rank approximation of the original matrix simply by drawing the test matrix $\Omega$ at random. Draw
a standard (normal) test matrix \(
\Omega \in \mathbb{R}^{n \times \ell}\), where \(\ell\) is the sketch size, and compute the sketch \(Y = A \,\Omega\). The Nystr\"{o}m approximation constructs low rank approximation directly from
the test matrix \(\Omega\) and the sketch \(Y\) via
\begin{equation}
    \begin{aligned}
    \widehat{A}_{\mathrm{nys}} \;=\; A\langle \Omega \rangle 
\;=\; Y\,(\Omega^{\mathsf{T}} Y)^{\dagger}\,Y^{\mathsf{T}}.
    \end{aligned}
\end{equation}

The randomness in the construction ensures that $ \widehat{A}_{\mathrm{nys}}$ is a
good approximation to the original matrix $A$ with high probability \cite[Section 14]{martinsson2020randomized}. Recently \cite{frangella2023randomized} uses Nystr\"{o}m approximation for solving (regularized) least square problems and we borrow the idea for the eigenvlaue problems. In this paper, we use the \texttt{NysSI} \cite{tropp2023randomized} methods to construct the pre-condition in \texttt{EPSI}.

\begin{algorithm}[ht]
\caption{Randomized Nyström Approximation \cite{szlam2014implementation,tropp2017fixed,frangella2023randomized}}
\label{alg:randnys}
\textbf{Input:} Positive-semidefinite matrix $A \in S_{n}^+(\mathbb{R})$, rank $\ell$\\
\textbf{Output:} Nyström approximation in factored form 
$\widehat{A}_{\mathrm{nys}} = U\,\widehat{\Lambda}\,U^{\mathsf{T}}$

\begin{algorithmic}[1]
\State Sample $\Omega = \mathrm{randn}(n,\ell)$  and compute QR Decomposition $\Omega = \mathrm{qr}(\Omega,0)$ 
    \Comment{Thin QR decomposition for the Gaussian test matrix}
\State $Y = A\,\Omega$ 
    \Comment{$\ell$ matrix-vector multiplications with $A$}
\State Compute $Y_{\nu} = Y + \nu\,\Omega$  where $\nu = \mathrm{eps}\bigl(\|Y\|_{\mathrm{fro}}\bigr)$ 
    \Comment{Shift for stability}
\State Compute  $C = \mathrm{chol}\bigl(\Omega^{\mathsf{T}}\,Y_{\nu}\bigr)$ and $B = Y_{\nu}\,/\,C$
\State $[\,U,\Sigma,\sim\,] = \mathrm{svd}\bigl(B,0\bigr)$ 
    \Comment{Thin SVD}
\State $\widehat{\Lambda} = \max\{\,0,\ \Sigma^{2} - \nu\,I\}$ 
    \Comment{Remove shift, compute eigenvalues}
\end{algorithmic}
\end{algorithm}

\begin{comment}
    
\
\paragraph{Sketching and Randomized Dimension Rediction}
The following construction was proposed by Kane \& Nelson \cite{kane2014sparser}. Consider a sparse random matrix of the form
\[
\Phi = [\varphi_1 \ \cdots \ \varphi_n] \in \mathbb{R}^{s \times n} \quad \text{where } \varphi_i \in \mathbb{R}^s \text{ are i.i.d. sparse vectors.} \tag{5.8}
\]
More precisely, each column $\varphi_i$ contains exactly $\zeta$ nonzero entries, equally likely to be $\pm 1 / \sqrt{\zeta}$, in uniformly random positions. It is not hard to check that the sparse map \eqref{5.8} preserves the squared norm of a vector on average, as in \eqref{5.5}. We can apply this matrix to a vector in $O(\zeta n)$ operations. The storage cost is at most $\zeta n$ parameters. If $\zeta \ll s$, then we obtain a significant computational benefit.

Cohen \cite{cohen2016nearly} showed that the sparse map serves as a subspace embedding with constant distortion, $\epsilon = \text{const}$, when the embedding dimension $s$ and the sparsity $\zeta$ satisfy
\[
s \geq \text{Const} \cdot \log d \quad \text{and} \quad \zeta \geq \text{Const} \cdot \log d.
\]

\end{comment}

\paragraph{Preconditioning and (Jacobi-)Davidson’s Method for Eigenvalue Computation}  Preconditioned iterative methods for the linear system
\(
A x - b = 0
\)
are often mathematically equivalent to standard iterative methods applied to the preconditioned system
\(
T(A x - b) = 0.
\)
For example, applying the classical Richardson iteration to the preconditioned system leads to the update
\(
x_{n+1} = x_n \;-\; \tau_n \, T\bigl(A x_n - b\bigr),
\)
where \(\tau_n\) is a suitably chosen scalar.

Turning to eigenvalue problems involving a real symmetric positive-definite matrix \(A\) \cite{knyazev1998preconditioned,knyazev2003geometric,knyazev2009gradient,argentati2017convergence}, one can compute an eigenvector by solving the homogeneous system
\(
(A - \lambda_* I)\,x = 0,
\)
or, equivalently,
\(
T\bigl(A - \lambda_* I\bigr)\,x = 0,
\)
where \(I\) is the identity matrix. The corresponding Richardson iteration step becomes
\(
x_{n+1} = x_n \;-\; \tau_n \, T\bigl(A - \lambda_* I\bigr)\,x_n.
\) Previous works~\cite{knyazev1998preconditioned,knyazev2003geometric,knyazev2009gradient,argentati2017convergence}
construct the preconditioner \(T\) by approximating \(A^{-1}\). This approach achieves an improved convergence rate of 
\(
\gamma + (1-\gamma)\,\bm{\texttt{gap}},
\)
where \(\gamma\) measures the approximation quality of \(T\) and \(\bm{\texttt{gap}}\) denotes the eigenvalue gap. Instead of preconditioning a first-order method, our work follows \cite{tapia2018inverse} by revisiting inverse power iteration as a second-order method. In this view, the algorithm naturally attains a \((1-\gamma)\,\bm{\texttt{gap}}\) convergence rate, strictly better than previous works \cite{knyazev1998preconditioned,knyazev2003geometric,knyazev2009gradient,argentati2017convergence}. The improvement is because we interpret preconditioned methods as an approximate Newton method but not pre-conditioning a first-order methods.

Notably, \cite{argentati2017convergence} assumes a preconditioner that satisfies $(1 - \eta)A \preceq \hat{A} \preceq (1 + \eta)A$, while our approach assumes $A-3\eta I\preceq \hat A\preceq A-\eta I$, which involves two main differences 
\begin{itemize}
    \item Our assumption only requires an absolute perturbation of $A$. It makes the implementation of a broader class of powerful approximation methods possible, and one of them is the Nyström approximation used in our algorithm, which, however, cannot satisfy the assumption in \cite{argentati2017convergence}.
    \item The precondition constructed in our algorithm is around $A-2\eta I$. This distinction motivates the introduction of a negative shift in our algorithm. {The negative shift guarantees that the preconditioner is invertible thus makes the algorithm more stable. }
\end{itemize}
Another way to precondition the eigenvalue computation is the (Jacobi-)Davidson’s Method $x_{n+1}=(D-\lambda I)^{-1}(A-\lambda I)x_n$. The Davidson's method is a widely used algorithm for extracting a small number of extremal eigenvalues of large, sparse, real symmetric matrices \cite{1975JCoPh..17...87D,doi:10.1137/S1064827500372973,huang2019inner}. It is particularly efficient when the matrix is nearly diagonal—equivalently, when its eigenvector matrix is close to the identity. This technique finds its primary application in theoretical chemistry, especially in ab initio quantum chemistry calculations, where the system matrices are strongly diagonally dominant. Though a lot of variants of Davidson’s method have been investigated \cite{sleijpen1995jacobi,zhou2006studies,crouzeix1994davidson}, to the best of our knowledge, no comprehensive theoretical analysis of Davidson’s method or related precondition-type iterative algorithms has been reported. In this paper, we establish the first convergence guarantee for this class of precondition-type methods, thereby enabling the incorporation of a broader class of powerful preconditioning techniques, such as the Nyström approximation, into this framework.

 %\yplu{ruihan, please change this paragraph to our assumption is different. Nystrom approximation don't satisfies their assumption but satisfies ours assumption.}

%A little irritating is the fact that the method fails for diagonal matrices: 
\paragraph{Relationship with Inexact Rayleigh Quotient-Type Methods} Rayleigh Quotient iteration serves as a fast solver in finding the extreme eigenvalue of a real symmetric matrix with a cubic convergence. The inexact Rayleigh Quotient-Type Methods solve the linear correction equation in Rayleigh Quotient iteration inexactly, either by using an iterative solver for matrix inversion, or solving an easier approximate linear equation. Though fast enough, RQI and inexact RQI have several difficulties in real practice: the iterative solver can be slow and unstable when the approximated eigenvalue is close to the real one, and it is highly sensitive to initialization, which makes it even harder to implement in solving multiple eigenpairs. As discussed in \cite{tapia2018inverse,simoncini2002inexact}, there exists an equivalence between the Rayleigh quotient iteration and the Newton-Grassman method. Furthermore, the Newton-Grassman method solves exactly the same linear correction equation as Jacobi-Davidson's Method. Compared to existing works, this paper is the first to analyze the convergence rate under the quality of preconditioner (inexact solver), and use Nystrom approximation to construct an inverse-friendly preconditioner with high quality. Our method can be easily generalized to solving multiple eigenpairs as Davidson's method, while preserve a fast convergence as inexact RQI-type method. We present the comparison between these methods in the experiment section.

\paragraph{Notation}  We write $\mathcal{S}_n(\mathbb{R})$ for the linear space of $n\times n$ real symmetric matrices, 
while $\mathcal{S}_n^+(\mathbb{R})$ denotes the convex cone of real positive semidefinite (psd) 
matrices. The symbol $A \preceq B$ denotes the Loewner order on $\mathcal{S}_n(\mathbb{R})$; that is,
$A \preceq B$ if and only if the eigenvalues of $B - A$ are all nonnegative. 
The function $\mathrm{tr}[\cdot]$ returns the trace of a square matrix. 
We write $\lambda_j(A)$ for the $j$th largest eigenvalue of $A$; we may omit $(A)$ when the context is clear.  $\|\cdot\|$ denotes vector $\ell_2$ norm for vectors and operator $\ell_2$ norm for matrices. $\|\cdot\|_F$ denotes Frobenius norm for matrices. We denote $V_{i:j}$ as matrix consists of $i$-th to $j$-th columns of $V$.

\section{Error-Powered Sketched Inverse Iteration}

In this section, we propose a new iteration named as  Error-Powered Sketched Inverse Iteration (\textbf{EPSI}). The inversion required in the Shifted Inverse Iteration \(x^{(k+1)} = \frac{(A - \mu I)^{-1} x^{(k)}}{\|(A - \mu I)^{-1} x^{(k)}\|}\) can be computationally expensive. To mitigate this, prior works such as \cite{jin2015robust,allen2016lazysvd} employ approximate solvers, leading to the iteration \(x^{(k+1)} = \frac{\widehat{(A - \mu I)^{-1}} x^{(k)}}{\|\widehat{(A - \mu I)^{-1}} x^{(k)}\|}\), where \(\widehat{(A - \mu I)^{-1}}\) is an approximation of the inverse. This approximation, while alleviating computational burden, generally yields only an approximate solution. Consequently, achieving higher accuracy necessitates the use of more precise approximation methods or improved solver techniques. Different from \cite{jin2015robust,allen2016lazysvd},  Error-Powered Sketched Inverse Iteration applies sketched
inverse iteration to the sketching error but not estimated eigenvector and leads to the following new iteration methods

\alg{Error-Powered Sketched Inverse Iteration}{
$$
u_{k+1}=\frac{\underbrace{(\hat A-\lambda(u_k) I)^{-1}}_{\text{\tiny Sketched Inverse Iteration}}\underbrace{(\hat A-A)u_k}_{\text{\tiny Sketch Error}}}{\|(\hat A-\lambda(u_k) I)^{-1}(\hat A-A)u_k\|},
$$
where $\lambda(u_k)=\frac{u_k^\top Au_k}{u_k^\top u_k}$ is the Rayleigh quotient estimation of an eigenvalue. 
}

Unlike (approximated) inverse iteration, the true eigenvector remains a fixed point of the Error-Powered Sketched Inverse Iteration, regardless of the approximation quality of $\hat A$. 
\begin{lemma} If $u^\ast$ is an unit eigenvector of the matrix $A$, \emph{i.e.} $Au^\ast=\lambda^\ast u^\ast$, then $u^\ast$ is a fixed point of Error-Powered Sketched Inverse Iteration.
\end{lemma}
\begin{proof}
Substituting \( Au^\ast = \lambda^\ast u^\ast \) into the iteration, we get
\[
u_{k+1} = \frac{(\hat{A} - \lambda(u^\ast) I)^{-1} (\hat{A} - A) u^\ast}{\| (\hat{A} - \lambda^\ast I)^{-1} (\hat{A} - A) u^\ast \|} = \frac{(\hat{A} - \lambda^\ast I)^{-1} (\hat{A} - \lambda^\ast I) u^\ast}{\| u^\ast \|} = u^\ast.
\]
Thus, \( u^\ast \) is a fixed point.
\end{proof}

Furthermore, in Theorem \ref{thm:Convergence}, we showed that $\hat A$ acts as a pre-conditioner: the closer $\hat A$ is to $A$, the faster the method converges. In Section \ref{section:EPSIasNewton}, we reinterpret the Error-Powered Sketched Inverse Iteration as a Newton Sketch method \cite{pilanci2017newton}, i.e., an approximate Newton method that uses an inexact Hessian but an exact gradient. This perspective reveals a linear-quadratic convergence behavior.
\begin{comment}

\subsection{analysis}
\subsubsection{algorithms}
\rh{I'll list all the versions of algorithms here} 

\textbf{version 1:}It is not a rigorous newton 
\begin{align*}
    u_{k+1} &= u_k - (\hat{A}-\lambda_R I)^{-1}(A - \lambda_R I)u_k\\
    &= (\hat{A}-\lambda_R I)^{-1}(\hat{A} -A)u_k
\end{align*}

\textbf{version 2:}Newton type of $F(x) = Ax-\lambda_R x$ where $\lambda_R = \frac{x^\top Ax}{x^\top x}$.
\begin{align*}
    u_{k+1}&= u_k - ((I-2u_ku_k^\top)(\hat{A}-\lambda_R I))^{-1}(A-\lambda_R I)u_k\\
    &= ((I-2u_ku_k^\top)(\hat{A}-\lambda_R I))^{-1}(\hat{A}-A-2u_ku_k^\top (\hat{A}-\lambda_R I))u_k
\end{align*}

\textbf{version 3:}Newton type of $F(x) = Ax-\lambda_R x+\frac{c}{2}(x^\top x-1)$
\begin{align*}
    u_{k+1}&= u_k - ((I-2u_ku_k^\top)(\hat{A}-\lambda_R I)+cu_ku_k^\top )^{-1}(A-\lambda_R I+\frac{c}{2}(u_k^\top u_k-1))u_k
\end{align*}
\end{comment}

\subsection{Convergence Analysis of Error-Powered Sketched Inverse Iteration}

In this section, we present the  Convergence Analysis of Error-Powered Sketched Inverse Iteration

\pfsketch{We first decompose the error $u_{k+1}-u_{\ast}$ into two parts: one reflecting the natural contraction of the EPSI iteration when using the exact eigenvalue, and another capturing the additional inaccuracy introduced by using an approximate eigenvalue as follows:

\begin{align*}
    u_{k+1}-u_* = \underbrace{\frac{(\hat{A}-\lambda_* I)^{-1}(\hat{A}-A)u_k}{\|(\hat{A}-\lambda_R I)^{-1}(\hat{A}-A)u_k\|}-u_*}_{\text{iteration using the true eigenvalue}}+\underbrace{\frac{((\hat{A}-\lambda_R I)^{-1}-(\hat{A}-\lambda_* I)^{-1})(\hat{A}-A)u_k}{\|(\hat{A}-\lambda_R I)^{-1}(\hat{A}-A)u_k\|}}_{\text{error cause by the approximated eigenvalue}}
\end{align*}

In Lemma \ref{power}, we show that applying power method to the preconditioned matrix $\frac{(\hat{A}-\lambda_* I)^{-1}(\hat{A}-A)u_k}{\|(\hat{A}-\lambda_* I)^{-1}(\hat{A}-A)u_k\|}$  is at a linear speed proportion to the quality of the sketched matrix.

Next we'll show that the additional inaccuracy introduced by using an approximate eigenvalue only contribute to the qudractic convergence part. This is because the approximated eigenvalue via the Rayleigh quotient satisfies $\lambda_R - \lambda_* = O(\|u_k-u_*\|^2)$. Thus we can prove that the error  $\|(\hat{A}-\lambda_R)^{-1}-(\hat{A}-\lambda_*)^{-1}\| = O(\|u_k-u_*\|^2)$ only  contribute to the qudractic convergence part. }

Building on the existing literature \cite{knyazev1998preconditioned, knyazev2003geometric, knyazev2009gradient, argentati2017convergence} that applies preconditioning to eigenvalue problems, we have consistently assumed that the preconditioner $\hat{A}^{-1}$ is a symmetric matrix satisfying
\(A-3\eta I \preceq \hat{A} \preceq A-\eta I,\) where $\eta$ measures the quality of the symmetric matrix. Previous works report a convergence rate of $1 - \eta + \eta \cdot \texttt{gap}$. In contrast, Lemma \ref{power} demonstrates that our \texttt{EPSI} method achieves a convergence rate of $\eta \cdot \texttt{gap}$, which is strictly better than the rates established in the existing preconditioning approaches for eigenvalue problems.

Argentati et al.~\cite{argentati2017convergence} introduce a preconditioner \(\hat{A}\) that satisfies the bound \((1 - \eta)A \preceq \hat{A} \preceq (1 + \eta)A\). In contrast, our approach constructs the preconditioner around \(A-2\eta I\), introducing a key distinction that influences the stability and effectiveness of our method. Specifically, we incorporate a negative shift in our preconditioning strategy. The negative shift is guided by the quality of the sketching approximation, and can be achieved based on different preconditioner:
\begin{itemize}
    \item With a preconditioner $\hat A=U_{nys}\hat\Lambda U_{nys}^\top$ produced by Nyström approximation, we impose a negative shift $c$ on $\hat \Lambda$ based on the knowledge of eigendecay (in most case $c\approx \frac{\hat \lambda_k^2}{\hat\lambda_1}$.
    \item With a preconditioner $\hat{A^\top A}=A^\top S^\top SA$ produced by Gaussian sketch when we want to conduct SVD on asymmetric $A\in \mathbb{R}^{m\times n}$, we draw the item in sketching matrix $S$ from $\mathcal{N}(0,c/s)$ with $c\approx 1-\sqrt{\frac{n}{s}}$ since the distortion $\eta=O(\sqrt{\frac{n}{s}})$.
\end{itemize}

This negative shift serves a crucial purpose: it ensures that the preconditioner remains invertible, thereby enhancing the numerical stability of the algorithm. The choice of shift is guided by the quality of the sketching approximation, which we quantify using the expected error of randomized SVD. By aligning the shift with the randomized SVD error, our method maintains robustness while optimizing the trade-off between approximation accuracy and computational efficiency.

\lemp{\texttt{EPSI} with Ground Truth Eigenvalue}{power}{Suppose one achieved an estimate $\hat{A}$ of semi-positive definite matrix $A$ with $A-3\eta I\preceq \hat A\preceq A-\eta I$, and furthermore shares a same column space with $A$. Let $A=V\Lambda V^\top=\lambda_1 v_1 v_1^\top+V_2\Lambda_2 V_2^\top$ be its eigen-decomposition, where $\lambda_1$ is its max eigenvalue and $\Lambda_2 =diag(\lambda_2,\lambda_3,\cdots,\lambda_n)$ satisfying $\lambda_2>\lambda_3>\cdots>\lambda_n$, then the power iteration $u_{k+1} = (\hat{A}-\lambda_1)^{-1}(\hat{A}-A)u_k$ will converge to its first eigenvector $v_1$. If $\frac{\|V_2^\top u_k\|}{\|u_k\|}\leq \epsilon$, then the iteration satisfies
    \begin{align*}
        \frac{\|V_2^\top u_{k+1}\|}{\|u_{k+1}\|}\leq \frac{3\lambda_1}{\lambda_1-\lambda_2}\frac{\eta}{\lambda_1}\frac{1}{1-3\epsilon}\frac{\|V_2^\top u_{k}\|}{\|u_{k}\|}.
    \end{align*}
}{Since $\hat A$ shares a same column space with $A$, which can be expressed as $\hat A = VTV^\top$, $A-3\eta I\preceq \hat A\preceq A-\eta I$ is equivalent to $\Lambda -3\eta I\preceq T\preceq \Lambda - \eta I$. Then the iteration of $u_k$ is
\begin{align*}
    \notag u_{k+1} &= (\lambda_1 I-\hat A)^{-1}(A-\hat A)u_k\\\notag
    &=V(\lambda_1 I-T)^{-1}(\Lambda-T)V^\top u_k,\\
\end{align*}
which indicates that
\begin{align}
\label{epsi1}
    \notag V^\top u_{k+1}&=(\lambda_1 I-T)^{-1}(\Lambda-T)V^\top u_k\\\notag
    (\lambda_1I-\Lambda)V^\top u_{k+1}&=(\lambda_1I-\Lambda)(\lambda_1 I-T)^{-1}(\Lambda-T)V^\top u_k\\
    &=\bigl(-(\lambda_1I-\Lambda)(\lambda_1 I-T)^{-1}(\lambda_1 I-\Lambda)+(\lambda_1 I-\Lambda)\bigr)V^\top u_k.
\end{align}
Note that $0\preceq\lambda_1I-\Lambda+\eta I\preceq \lambda_1I-T\preceq\lambda_1I-\Lambda+3\eta I$, thus we have
\begin{align*}
    D_-\preceq(\lambda_1I-\Lambda)(\lambda_1 I-T)^{-1}(\lambda_1I-\Lambda)\preceq D_+,
\end{align*}
where  
\begin{align*}
    & D_-=(\lambda_1I-\Lambda)(\lambda_1I-\Lambda+3\eta I)^{-1}(\lambda_1I-\Lambda)\\
    &D_+=(\lambda_1I-\Lambda)(\lambda_1I-\Lambda+\eta I)^{-1}(\lambda_1I-\Lambda)
\end{align*}
are both positive diagonal matrix. Then we can bound \eqref{epsi1} by 
\begin{align*}
    &\|\bigl(-(\lambda_1I-\Lambda)(\lambda_1 I-T)^{-1}(\lambda_1 I-\Lambda)+(\lambda_1 I-\Lambda)\bigr)V^\top u_k\|\\
    \leq& \|-D_-+\lambda_1I-\Lambda\|\|V^\top u_k^2\|\leq 3\eta\|V^\top u_k^2\|,
\end{align*}
where the last inequality comes from 
\begin{align*}
    \|-D_-+\lambda_1I-\Lambda\|&=\max_{ \lambda_i<\lambda_1} \lambda_1-\lambda_i-\frac{(\lambda_1-\lambda_i)^2}{\lambda_1-\lambda_i+3\eta}\\
    &= \max_{0\leq x}x-\frac{x^2}{x+3\eta}\\
    &=\max_{0\leq x} \frac{3\eta x}{x+3\eta}<3\eta.
\end{align*}

Note that $(\lambda_1I-\Lambda)V^\top=(\lambda_1I-\Lambda_2)V^\top_2$, thus we have
\begin{align*}
    \|V_2^\top u_{k+1}\|\leq \frac{3\eta}{\lambda_1-\lambda_2}\|V_2^\top u_k\|.
\end{align*}

Now we want to compare the normalized error $\frac{\|V_2^\top u_{k+1}\|}{\|u_{k+1}\|}$ and $\frac{\|V_2^\top u_{k}\|}{\|u_{k}\|}$. Before that, we need to bound $\|u_{k+1}\|$ in terms of $\|u_k\|$ first. Note that $u_{k+1} = u_k + (\lambda_1I-\hat A)^{-1}(A-\lambda_1I)V_2V_2^\top u_k$, where $\|(\lambda_1I-\hat A)^{-1}(A-\lambda_1I)V_2\|$ can be further bounded by
\begin{align*}
    \|(\lambda_1I-\hat A)^{-1}(A-\lambda_1I)V_2\|&=\|(\lambda_1I-T)^{-1}(\Lambda-\lambda_1I)V^\top V_2\|\\
    &\leq \frac{\lambda_1}{\eta}.
\end{align*}
Thus $\|u_{k+1}\|\geq\|u_k\|-\frac{\lambda_1}{\eta}\|V_2^\top u_k\|$. Finally with assumption $\frac{\|V_2^\top u_{k}\|}{\|u_{k}\|}\leq \epsilon$ we have
\begin{align*}
    \frac{\|V_2^\top u_{k+1}\|}{\|u_{k+1}\|}&\leq \frac{\frac{3\eta}{\lambda_1-\lambda_2}\|V_2^\top u_k\|}{\|u_{k}\|-\frac{\lambda_1}{\eta}\|V_2^\top u_k\|}\leq \frac{3\eta}{\lambda_1-\lambda_2}\frac{1}{1-\frac{\lambda_1}{\eta}\epsilon}\frac{\|V_2^\top u_{k}\|}{\|u_{k}\|}.
\end{align*}}

\thmp{Convergence Rate of EPSI}{thm:Convergence}{
    For estimate $\hat{A}$ of semi positive definite matrix $A$ with $A-3\eta I\preceq \hat A\preceq A-\eta I$, and furthermore shares a same column space with $A$. Let $A=V\Lambda V^\top=\lambda_\ast v_\ast v_\ast^\top+V_2\Lambda_2 V_2^\top$ be its eigen-decomposition, where $\lambda_\ast=\lambda_1$ is its max eigenvalue and $\Lambda_2 =diag(\lambda_2,\lambda_3,\cdots,\lambda_n)$ satisfying $\lambda_2>\lambda_3>\cdots>\lambda_n$. Then EPSI yields a (normalized) series $\{u_k\}$ which converges to $u_*$ in a linear-quadratic behavior. Suppose that $u_k$ satisfies $\frac{\|V_2^\top u_k\|}{\|u_k\|}\leq \epsilon$ with $\frac{\lambda_1}{\eta}\epsilon <1 $, then the convergence of $u_{k+1}$ is guaranteed by 
    \begin{align*}
        \|V_2^\top u_{k+1}\|\leq  \frac{3\lambda_1}{\lambda_1-\lambda_2}\frac{\eta}{\lambda_1}\frac{1}{1-\frac{\lambda_1}{\eta}\epsilon}\|V_2^\top u_k\|+\frac{2\lambda_1}{\eta}\|V_2^\top u_k\|^2
    \end{align*} 
}{As stated in the proof sketch, we first decompose $u_{k+1}-u_* $ into two parts 

\begin{align*}
    V_2^\top u_{k+1} &= \frac{V_2^\top (\hat{A}-\lambda_R I)^{-1}(\hat{A}-A)u_k}{\|(\hat{A}-\lambda_R I)^{-1}(\hat{A}-A)u_k\|}\\
    &= \underbrace{\frac{V_2^\top (\hat{A}-\lambda_* I)^{-1}(\hat{A}-A)u_k}{\|(\hat{A}-\lambda_R I)^{-1}(\hat{A}-A)u_k\|}}_{\text{iteration using the true eigenvalue}}+\underbrace{\frac{V_2^\top ((\hat{A}-\lambda_R I)^{-1}-(\hat{A}-\lambda_* I)^{-1})(\hat{A}-A)u_k}{\|(\hat{A}-\lambda_R I)^{-1}(\hat{A}-A)u_k\|}}_{\text{error cause by the approximated eigenvalue}}
\end{align*}

By lemma \ref{power} we have
\begin{align*}
     \frac{\|V_2^\top (\hat{A}-\lambda_* I)^{-1}(\hat{A}-A)u_k\|}{\|(\hat{A}-\lambda_* I)^{-1} (\hat{A}-A)u_k\|}\leq \frac{3\eta}{\lambda_1-\lambda_2}\frac{1}{1-\frac{\lambda_1}{\eta}\epsilon}\|V_2^\top u_k\|\\
\end{align*}

On the other hand, for the error made by the estimated eigenvalue $V_2^\top \bigl((\hat{A}-\lambda_R I)^{-1}-(\hat{A}-\lambda_* I)^{-1}\bigr)(\hat{A}-A)u_k$, we have
\begin{align*}
    &\|V_2^\top ((\hat{A}-\lambda_R I)^{-1}-(\hat{A}-\lambda_* I)^{-1})(\hat{A}-A)u_k\|\\
    \leq &\| (I-(\hat{A}-\lambda_* I)^{-1}(\hat{A}-\lambda_R I))(\hat{A}-\lambda_R I)^{-1}(\hat A-A)u_k\|\\
    = & \|(\hat A-\lambda_\ast I)^{-1}(\lambda_R-\lambda_\ast)(\hat A-\lambda_R I)^{-1}(\hat A-A)u_k\|\\
    \leq & \frac{1}{\eta}(\lambda_\ast-\lambda_R)\|(\hat A-\lambda_R I)^{-1}(\hat A-A)u_k\|\\
\end{align*}
Since$|\lambda_*-\lambda_R(u_k)|\leq \lambda_*-\lambda_*\|v_*^\top u_k\|^2=\lambda_*(1-(1-\|V_2^\top u_k\|^2))=\lambda_\ast\|V_2^\top u_k\|^2$. Thus 
\begin{align*}
    \| ((\hat{A}-\lambda_R I)^{-1}-(\hat{A}-\lambda_* I)^{-1})(\hat{A}-A)u_k\|\leq \frac{\lambda_\ast|V_2^\top u_k\|^2}{\eta}\|(\hat A-\lambda_R I)^{-1}(\hat A-A)u_k\|
\end{align*}

Combining the linear term and quadratic term one gets the convergence rate of EPSI
\begin{align*}
    \|V_2^\top u_{k+1}\|&=\|\frac{V_2^\top (\hat{A}-\lambda_* I)^{-1}(\hat{A}-A)u_k}{\|(\hat{A}-\lambda_R I)^{-1}(\hat{A}-A)u_k\|}+\frac{V_2^\top ((\hat{A}-\lambda_R I)^{-1}-(\hat{A}-\lambda_* I)^{-1})(\hat{A}-A)u_k}{\|(\hat{A}-\lambda_R I)^{-1}(\hat{A}-A)u_k\|}\|\\
    &\leq \frac{3\eta}{\lambda_1-\lambda_2}\frac{1}{1-3\epsilon}\|V_2^\top u_k\|+2\|\frac{ ((\hat{A}-\lambda_R I)^{-1}-(\hat{A}-\lambda_* I)^{-1})(\hat{A}-A)u_k}{\|(\hat{A}-\lambda_R I)^{-1}(\hat{A}-A)u_k\|}\|\\
    &\leq \frac{3\eta}{\lambda_1-\lambda_2}\frac{1}{1-3\epsilon}\|V_2^\top u_k\|+\frac{2\lambda_\ast}{\eta}\|V_2^\top u_k\|^2,
\end{align*}
where the second inequality comes from 
\begin{align*}
    &\|\frac{V_2^\top (\hat{A}-\lambda_* I)^{-1}(\hat{A}-A)u_k}{\|(\hat{A}-\lambda_R I)^{-1}(\hat{A}-A)u_k\|}\|\\
    =&\|\frac{V_2^\top (\hat{A}-\lambda_* I)^{-1}(\hat{A}-A)u_k}{\|(\hat{A}-\lambda_\ast I)^{-1}(\hat{A}-A)u_k\|}+\frac{V_2^\top (\hat{A}-\lambda_* I)^{-1}(\hat{A}-A)u_k}{\|(\hat{A}-\lambda_R I)^{-1}(\hat{A}-A)u_k\|}-\frac{V_2^\top (\hat{A}-\lambda_* I)^{-1}(\hat{A}-A)u_k}{\|(\hat{A}-\lambda_\ast I)^{-1}(\hat{A}-A)u_k\|}\|\\
    \leq&\|\frac{V_2^\top (\hat{A}-\lambda_* I)^{-1}(\hat{A}-A)u_k}{\|(\hat{A}-\lambda_\ast I)^{-1}(\hat{A}-A)u_k\|}\|\\
    +&\|\frac{V_2^\top (\hat{A}-\lambda_* I)^{-1}(\hat{A}-A)u_k(\|(\hat{A}-\lambda_\ast I)^{-1}(\hat{A}-A)u_k\|-\|(\hat{A}-\lambda_R I)^{-1}(\hat{A}-A)u_k\|)}{\|(\hat{A}-\lambda_\ast I)^{-1}(\hat{A}-A)u_k\|\|(\hat{A}-\lambda_R I)^{-1}(\hat{A}-A)u_k\|}\|\\
    \leq&\|\frac{V_2^\top (\hat{A}-\lambda_* I)^{-1}(\hat{A}-A)u_k}{\|(\hat{A}-\lambda_\ast I)^{-1}(\hat{A}-A)u_k\|}\|+\frac{\|(\hat{A}-\lambda_\ast I)^{-1}(\hat{A}-A)u_k\|-\|(\hat{A}-\lambda_R I)^{-1}(\hat{A}-A)u_k\|}{\|(\hat{A}-\lambda_R I)^{-1}(\hat{A}-A)u_k\|}\\
    \leq&\|\frac{V_2^\top (\hat{A}-\lambda_* I)^{-1}(\hat{A}-A)u_k}{\|(\hat{A}-\lambda_\ast I)^{-1}(\hat{A}-A)u_k\|}\|+\frac{\|((\hat{A}-\lambda_\ast I)^{-1}-(\hat{A}-\lambda_R I)^{-1})(\hat{A}-A)u_k\|}{\|(\hat{A}-\lambda_R I)^{-1}(\hat{A}-A)u_k\|}.\\
\end{align*}}
%\rh{here the coefficient of quadratic term is $1/\eta$, but it doesn't matter right?}
%\yplu{yes}}
%\yplu{I checked, I guess the assumption here can be change to subspace embedding property\todoyp{Can \cite[Prop 5.4,5.6]{kireeva2024randomized} help to verify this}}

\rmkb{Theorem~\ref{thm:Convergence} shows that EPSI achieves linear--quadratic convergence, where the linear rate is on the order of \(\frac{\eta}{\lambda_1 - \lambda_2}.\) As the sketch quality \(\eta\) improves, the convergence rate also improves.  This result illustrates that EPSI mirrors the advantages of the Sketch-and-Precondition framework,  namely that a higher-quality sketched approximation leads to faster convergence. 

Note that with a fixed $\eta$, the algorithm's behavior within the $\eta$-neighborhood of the true solution is dominated by the linear term. Moreover, if $\eta$ is small, then by Davis–Kahan perturbation theory, the initial eigenvector estimate (obtained via low-rank approximation) will lie within this $\eta$-neighborhood. In other words, \textbf{a smaller $\eta$ results in a reduced convergence region and a closer initial guess to the true solution, thereby ensuring that the bound holds in practice.}
}

\begin{comment}

\cor{For a random standard Gaussian matrix $\Phi\in\mathbb{R}^{s\times n}$ with i.i.d $\mathcal{N}(0,\frac{1}{s})$ entries, it holds
\[
\mathbb{P}\{
  1 - \sqrt{\frac{n}{s}} - t\leq
  \lambda_{\min}(\Phi U)\leq
  \lambda_{\max}(\Phi U)\leq
  1 + \sqrt{\frac{n}{s}} + t
\}\leq
1 - e^{\frac{st^2}{2}}.
\]
As $s,n\to \infty$ and $\frac{n}{s}$ fixed, the singular value converge to $\sqrt{\frac{n}{s}}$ almost surely.
\cite[Thm II.I3]{davidson2001local} }

\cor{Given target matrix $A\in \mathbb{R}^{m\times n}$ and distortion $\eta>0$, one can always construct a random matrix $S$ such that with high probability it holds $(1-3\eta)I\preceq T\preceq (1-\eta)I$, by rescaling a standard Gaussian random matrix. The distortion $\eta_S$ of a given standard Gaussian random matrix $\hat{S}\in \mathbb{R}^{m\times n}$ can be approximated as $\eta_{S} = 2\sqrt{\frac{n}{s}}$. Suppose that $(1-\eta)I\preceq T\preceq (1+\eta)I$ holds with probability $p$. Let $\eta'=\sqrt{\frac{1-\eta}{1+\eta}}$, then random matrix $S=\eta' \hat{S}$ has i.i.d $\mathcal{N}(0,\eta'/s)$ entries, and with same probability $p$ we have 
\begin{align*}
    (1-3\eta)I\preceq \frac{(1-\eta)^2}{1+\eta}I\preceq T\preceq (1-\eta)I
\end{align*}}
\end{comment}
\subsection{EPSI for Computing $k$-SVD}
\label{section:ksvd}

In this section, we extend our algorithm to computing the first $k$ singular vectors of a matrix $D$.  The singular value decomposition (SVD) of a rank-$r$ matrix $D \in \mathbb{R}^{n \times m}$ corresponds to decomposing $D = V \Sigma U^T$, where $V \in \mathbb{R}^{n \times r}$ and $U \in \mathbb{R}^{m \times r}$ are two column-orthonormal matrices, and  $\Sigma = \mathrm{diag}(\sigma_1, \ldots, \sigma_r) \in \mathbb{R}^{r \times r} $is a non-negative diagonal matrix with $\sigma_1 \ge \sigma_2 \ge \cdots \ge \sigma_r \ge 0$. The columns of $V$ (resp.\ $U$) are called the left (resp.\ right) singular vectors of $D$ and the diagonal entries of $\Sigma$ are called the singular values of $D$. SVD is one of the most fundamental tools used in machine learning, computer vision, statistics, and operations research, and is essentially equivalent to principal component analysis (PCA) up to column averaging.  Since $\lambda_k(DD^\top)=\lambda_k(D^\top D)=(\sigma_k(D))^2$, solving k-svd is equivalent to solving first k eigendecomposition of $A= DD^\top \in \mathbb{R}^{n\times n}$. We denote $\lambda_1\geq\lambda_2\geq\cdots\geq\lambda_d\geq0$ as the eigenvalues of $M$, where $\lambda_{r+1}=\lambda_{r+2}=\cdots=\lambda_{d}=0$ because of rank $r$ of $A$. We will mainly work on $A$ instead of $D$ for simplicity.

The complexity of computing the top-$k$ singular vectors should depend on the relative gap 
\(
\frac{\lambda_k}{\lambda_k - \lambda_{k+1}}
\)
\cite{li2015convergence,musco2015randomized}.
However, if one naively performs top singular vector computation (1-SVD) repeatedly $k$ times, the running time would depend on all the intermediate gaps $\max\{\frac{\lambda_1}{\lambda_1-\lambda_2},\cdots,\frac{\lambda_k}{\lambda_k-\lambda_{k+1}}\}$. A breakthrough work~\cite{allen2016lazysvd} showed that we can tolerate having the computation of the $s$-th leading eigenvalue ''approximately'' lie in the span of the top $k$ singular vectors, at the cost of a multiplicative error
\(
\frac{\lambda_k}{\lambda_k - \lambda_{k+1}}.
\) Building on a similar idea from~\cite{allen2016lazysvd}, we propose \textbf{\texttt{Lazy-EPSI}}, a (randomized) preconditioned algorithm for computing the first $k$ singular vectors whose complexity depends only on the relative gap 
\(
\frac{\lambda_k}{\lambda_k - \lambda_{k+1}},
\)
and not on the intermediate gaps. Note that \emph{EPSI \textbf{cannot} serve as the approximate 1-SVD algorithm for LazySVD~\cite{allen2016lazysvd}} because LazySVD requires an anisotropic convergence speed for different singular vectors (depending on their singular values). This requirement contradicts the philosophy of preconditioning, which enforces isotropic convergence in every direction. Consequently, \emph{a new proof technique} is needed for our derivation.

\subsubsection{\texttt{Lazy-EPSI}}

In this section, we propose \textbf{\texttt{Lazy-EPSI}}, inspired by {Lazy-SVD}~\cite{allen2016lazysvd}. 
Our goal is to construct a rank-$k$ SVD algorithm by performing 1-SVD $k$ times in sequence. 
Crucially, the complexity of our method depends only on the ratio
\(
\frac{\lambda_k}{\lambda_k - \lambda_{k+1}},
\)
rather than on any intermediate gaps. Although the convergence of the $i$-th singular vector often depends on the intermediate gap 
\(
\frac{\lambda_i}{\lambda_{i} - \lambda_{i+1}},
\)
if the estimate has sufficiently small overlap with the previously computed $k-1$ singular vectors, 
its convergence may instead be governed by the larger gap 
\(
\frac{\lambda_i}{\lambda_i - \lambda_{k+1}}.
\) Therefore, we introduce a orthogonalization step that reduces the component of the estimated $i-$th singular vector in the directions of the first $k-1$ singular vectors to the order of the current estimation error. This step is crucial for ensuring the linear-quadratic convergence guarantee.
 The complete procedure is detailed in Algorithm~\ref{alg:iterative_matrix_approximation}.

\begin{algorithm}[]
\caption{\texttt{Lazy-EPSI} for Computing the First $k$ Singular Vectors}
\label{alg:iterative_matrix_approximation}
\begin{algorithmic}[1]
\Require $A$, the input matrix; $k \in \mathbb{Z}^+$, the number of components; $q_{\text{max}} \in \mathbb{Z}^+$, the maximum number of iterations.

\For{$q = 1$ to $q_{\text{max}}$}
    \State $U \gets [\,]$ \Comment{\footnotesize Initialize $U$ as an empty matrix.}
    \State \textbf{EPSI Iteration: Update first $k$ eigenspace estimation $U$}
    \For{$i = 1$ to $k$}
        \State $\hat{\lambda}_i \gets \frac{(u_q^i)^\top A u_q^i}{(u_q^i)^\top u_q^i}$ 
        \Comment{\footnotesize Compute rayleigh quotient estimation $\hat{\lambda}_i$ for the $i$th eigenvector.}
        \State{ \color{cyan!80} $\hat{u}_{q+1}^i \gets \big((I - UU^\top)\hat{A}(I - UU^\top) - \hat{\lambda}_i I\big)^{-1}((I - UU^\top)\hat{A}(I - UU^\top) - A)u_q^i$}
        \Comment{\footnotesize Update $\hat{u}_{q+1}^i$.}
        \State $U \gets orth([U, \hat{u}_{q+1}^i])$ 
        \Comment{\footnotesize Append gram-schmidt orthogonalized $\hat{u}_{q+1}^i$ to $U$.}
    \EndFor
    \State \textbf{Orthogonalization step:  Use Estimated Eigenvector as Rangefinder}
    \State $\Pi \gets U U^\dagger$ 
    \Comment{\footnotesize Compute the projection matrix $\Pi$.}
    \State $A_U \gets \Pi A \Pi$ 
    \Comment{\footnotesize Compute the projected matrix $A_U$.}
    \State $[u_{q+1}^1, u_{q+1}^2, \dots, u_{q+1}^k] \gets \text{SVD}(A_U)$ 
    \Comment{\footnotesize Perform SVD to update $u_{q+1}^i$.}
\EndFor
\State \Return $U$ \Comment{\footnotesize Return the updated matrix $U$.}
\end{algorithmic}
\end{algorithm}

\paragraph{Fast Computation of {\color{cyan!80}Step 6} of Algorithm \ref{alg:iterative_matrix_approximation} via Nystr\"{o}m Approximation} Suppose one has a rank-$l$ Nystr\"{o}m approximation $\hat A_{\text{nys}}=U_{\text{nys}}\hat{\Lambda}U^\top_{\text{nys}}$. We aim to demonstrate that the inversion $\big((I - UU^\top)\hat{A}(I - UU^\top) - \hat{\lambda}_i I\big)^{-1}$ in the {\color{cyan!80}Step 6} of \texttt{Lazy-EPSI} admits a closed-form solution, enabling fast computation. Specifically, the computational cost of this step is $\mathcal{O}(nl^2) + \mathcal{O}(l^3)$, where $n$ is the dimension of the matrix and $l$ is the rank of the Nystr\"{o}m approximation. This efficiency arises because the large-scale matrix operations scale linearly with $n$, and the inversion is confined to an $l \times l$ matrix, which remains computationally manageable for small $l$.

\begin{lemma}\label{lemma:woodbury} If we denote \(M = (I - UU^\top)\,\widehat{A}\,(I - UU^\top) \;-\; \widehat{\lambda}_i\,I,\)  where \(W = (I - UU^\top)\,U_{\text{nys}}\), then we have
\[
M^{-1} 
\;=\;
-\frac{1}{\widehat{\lambda}_i}\,I
\;+\;
\frac{1}{\widehat{\lambda}_i}\,
W\,
\bigl(-\,\widehat{\lambda}_i\,\widehat{\Lambda}^{-1} + W^\top W\bigr)^{-1}
W^\top,
\]
\end{lemma}
\begin{proof} Recall the Woodbury identity \((I + U\,C\,V^\top)^{-1}
\;=\; I \;-\; U\bigl(C^{-1} + V^\top U\bigr)^{-1}V^\top\), we have

\[
M^{-1}
\;=\;
-\,\frac{1}{\widehat{\lambda}_i}\,(I - \tfrac{1}{\widehat{\lambda}_i}\;W\,\widehat{\Lambda}\,W^\top)^{-1}
= 
-\,\frac{1}{\widehat{\lambda}_i}\left(I - W \,\Bigl(-\,\widehat{\lambda}_i\,\widehat{\Lambda}^{-1} + W^\top W\Bigr)^{-1} W^\top\right).
\]
\end{proof}

All large-scale multiplications involve $(I - UU^\top)$ and $U_{\text{nys}}$, which can be computed efficiently with a computational complexity of $\mathcal{O}(nl)$. The expensive matrix inversion is restricted to the small matrix $\bigl(-\,\widehat{\lambda}_i\,\widehat{\Lambda}^{-1} + W^\top W\bigr)$, which is of size $l \times l$ and incurs a computational cost of $\mathcal{O}(l^3)$. Since $l \ll n$, the overall computational cost for Step 6 is dominated by $\mathcal{O}(nl)$, ensuring that the computation remains scalable and efficient even for large $n$. Thus the computational cost of the \texttt{EPSI} iteration step (with the Sherman-Morrison-Woodbury implementation of inverse iteration) is $O(k(\underbrace{nl^2+l^3}_{\text{Inverse Iteration}}+\underbrace{mn}_{\text{compute }Au_t}))$. The orthogonalization step is the same computational cost as the Subspace Iteration algorithm which cost a computational $O(mnk)$. This leads to the final computational cost at $O(k(nl^2+l^3+mn))$. Our iteration is implemented solely using matrix-vector multiplication, allowing it to leverage matrix sparsity for efficient computation. In numerical experiments, this approach enabled the application of our algorithm to large-scale matrices.

We invoke Corollary 2.3 in \cite{frangella2023randomized} to show that Nyström approximation yields a approximation $\hat A$ satisfies $A-\eta I\preceq\hat A\preceq A+\eta I$ for some $\eta$.
\begin{lemma}
    For $p \geq 2$ and $l = 2p - 1$, we have the bound
    \begin{align*}
        E\| A - \hat A\| \leq \biggl(3 +\frac{4e^2}{p}srp(A)\biggr)\lambda_ p.
    \end{align*}
    The p-stable rank, $srp(A) = \lambda_p^{-1}\sum_{j=p}^n\lambda_j$ reflects decay in the tail eigenvalues.
\end{lemma}

\rmkb{To achieve a negative shift on $\hat A$ such that $A-3\eta I\preceq \hat A\preceq 1-\eta I$, one can directly subtract computed diagonal matrix $\hat \Lambda$ by $c \approx 2\eta$ based on knowledge of eigendecay. }
\subsubsection{Convergence Rate of \texttt{Lazy-EPSI}}

In this section, we present the linear-quadratic convergence rate of 
\texttt{Lazy-EPSI}. Specifically, the linear convergence 
factor is dependent on 
\(\frac{\lambda_k}{\lambda_k - \lambda_{k+1}},\) which aligns with the intuition from subspace iteration. Moreover this rate has a linear dependence on \(\eta\). We provide a sketch of the proof here.

\pfsketch{We begin by decomposing the error of \texttt{Lazy-EPSI} into two parts. First, we show that if the estimates of the former eigenvectors are sufficiently accurate, then the error is reduced by a factor of \(\frac{\eta}{\lambda_i - \lambda_{k+1}},\)
reflecting a linear convergence rate. Second, we bound the additional error arising from inexact projection by measuring how much the approximate singular vectors overlap with the previously computed subspace, \emph{i.e.} \(V_1^\top u_q^i\). Finally, in Lemma \ref{correction}, we prove that the orthogonalization step further refines this projection (quadratically), thereby accelerating the overall convergence to a quadratic rate.}

\lem{Local Convergence Rate for \texttt{Lazy-EPSI}}{lazyepsi}{
    Suppose that PSD matrix $A\in \mathbb{R}^{n\times n}$ has exact eigendecomposition $A=V\Lambda V^\top =V_1\Lambda_1V_1^\top +V_2\Lambda_2V_2^\top$, where $V_1$ has size $n\times k$ and the diagonal of $\Lambda_1,\Lambda_2$ is in descending order. The approximation $\hat{A}$ of $A$ satisfies $A-3\eta I\preceq \hat{A}\preceq A-\eta I$. Suppose that $U$, which is the $i-1$ eigenspace estimation, satisfies $\|V_{i:n}^\top U\|\leq \epsilon$ for small constant $\epsilon<\frac{\eta}{26(\lambda_1-\eta)}$. Then $\hat{u}_{q+1}^i=((I-UU^\top)\hat{A}(I-UU^\top)-\lambda_iI)^{-1}((I-UU^\top)\hat{A}(I-UU^\top)-A)u_{q}^i$ satisfies:
    \begin{align*}
        \frac{\|V_2^\top u^i_{q+1}\|}{\|u^i_{q+1}\|}\leq \frac{1}{1-\epsilon_0}( \underbrace{\frac{3.5\lambda_k}{\lambda_i-\lambda_{k+1}}\frac{\eta}{\lambda_k}\frac{\|V_2^\top u_q^i\|}{\|u_q^i\|}}_{\text{linear convergence}}+\underbrace{\frac{c_1(\lambda_i-\lambda_n)}{(\lambda_i-\lambda_{k+1})^2}\frac{\|(\lambda_i I-\Lambda_1)V_1^\top u_q^i\|}{\|u_q^i\|}}_{\text{error caused by imperfect projection}})
    \end{align*}
    where $\epsilon_0 = max\{\frac{4(\lambda_i-\lambda_{k+1})}{\eta}\|V_2^\top u_q^i\|,\frac{68\lambda_1}{\eta}\|V_2^\top U_q\|^2\}$ with $U_q=[u_q^1,u_q^2,\cdots,u_q^k]$, and $c_1$ is a small constant that depends on $\sqrt{k}$. 
}

\rmkb{Note that we assume $\|V_{i:n}^\top U\|\leq \epsilon$, which also implies another assumption that there are no identical eigenvalues in the first k eigenvalues. The assumption guarantees the local property of the initial guess, for the convenience of theoretical analysis, and will not impact the convergence rate, as shown in our result. Furthermore, the experiment shows that the algorithm still works even if there are identical eigenvalues. 

The dependence of $\epsilon$ on $\eta$ aligns with the quality of the initial guess, which lies in a $\eta$-neighborhood of the true solution.}

\paragraph{Orthogonalization step improves $\|(\lambda_iI-\Lambda_1)V_1^\top u^i_{q}\|$ Quadratically.} In this section, we show that 
the Orthogonalization step enables to improve the $\|(\lambda_iI-\Lambda_1)V_1^\top u^i_{q}\|$ Quadratically.
\begin{lemma}
\label{correction}
    Suppose that PSD matrix $A$ has SVD $A=V_1\Lambda_1 V_1^\top+V_2 \Lambda_2 V_2^\top $, where $\Lambda_1$ has leading $k$ eigenvalues $\lambda_1>\lambda_2>\cdots>\lambda_k$. Then for orthogonal matrix $U\in\mathbb{R}^{n\times k}$, the first $k$ eigenvector $[u^1,u^2,\cdots,u^k]$ of $UU^\top AUU^\top$, which is the projection of $A$ onto $U$, satisfies
    \begin{align*}
        \|(\lambda_iI-\Lambda_1)V_1^\top u^i\|&\leq 17\lambda_1\|V_2^\top U\|^2.\\
    \end{align*}
    The result indicates that after the orthogonalization step, the error introduced by imperfect projection can be quadratically dependent on $\|V_2^\top u^j\|$. 
\end{lemma}

Based on Lemma \ref{correction}, we present the convergence proof of \texttt{Lazy-EPSI}. The key idea is that the error introduced by imperfect projection is mitigated by the orthogonalization step and can be bounded as a quadratic term in the convergence rate (via Lemma \ref{correction}). As a result, the method ultimately achieves a linear convergence rate that depends only on $\frac{\lambda_k}{\lambda_k - \lambda_{k+1}}$.

\thmp{Convergence Rate of \texttt{Lazy-EPSI}}{thm:Convergencelazy}{Suppose that PSD matrix $A\in \mathbb{R}^{n\times n}$ has exact SVD $A=V\Sigma V^\top =V_1\Sigma_1V_1^\top +V_2\Sigma_2V_2^\top$, where $V_1$ has size $n\times k$ and $diag(\Sigma_1,\Sigma_2)$ is in descending order. {The approximation $\hat{A}$ of $A$ satisfies $A-3\eta I\preceq \hat{A}\preceq A-\eta I$ with some distortion factor $\eta<\frac{\lambda_{k}-\lambda_{k+1}}{3.5\sqrt{2}}$.} Suppose that the initial guess $U_0=[u_0^1,u_0^2,\cdots,u_0^k]$ satisfies the assumption of lemma \ref{lazyepsi}. Denote $U_q$ is the orthonormal basis of $q$-th iteration $[u_q^1,u_q^2,\cdots,u_q^k]$, then the convergence of Lazy-EPSI is guaranteed by:
\begin{align*}
    \|V_2^\top  U_{q+1}\|_F&\leq c_0(\frac{3.5\sqrt{2}\lambda_k}{\lambda_k-\lambda_{k+1}}\frac{\eta}{\lambda_k}\|V_2^\top U_q\|_F+c_{2}\|V_2^\top U_q\|_F^2),\\
\end{align*}
where $c_0=1+\frac{c\lambda_1 k\|V_2^\top U_q\|}{\eta}$ and $c_{2}=\frac{17\sqrt{2k}c_1\lambda_1(\lambda_1-\lambda_n)}{(\lambda_k-\lambda_{k+1})^2}+\frac{36\lambda_1}{\eta}$.
}{ In the following proof, we change the sequence of two part in our algorithm, which is equivalent to the original algorithm but easier to illustrate. In one single iteration of Lazy-EPSI, we update $u_q$ by the first correction part and the second EPSI part, which can be expressed as 
    \begin{align*}
        \hat{u}_{q+1}^i&= \text{eigenvector}_i(\Pi A\Pi),\\
        u_{q+1}^i &= \frac{((I-\hat U_{i-1}\hat U_{i-1}^\top)\hat A(I-\hat U_{i-1}\hat U_{i-1}^\top)-\hat\lambda_i I )^{-1}((I-\hat U_{i-1}\hat U_{i-1}^\top)\hat A(I-\hat U_{i-1}\hat U_{i-1}^\top)- A)\hat{u}_{q+1}^i}{\|((I-\hat U_{i-1}\hat U_{i-1}^\top)\hat A(I-\hat U_{i-1}\hat U_{i-1}^\top)-\hat\lambda_i I )^{-1}((I-\hat U_{i-1}\hat U_{i-1}^\top)\hat A(I-\hat U_{i-1}\hat U_{i-1}^\top)- A)\hat{u}_{q+1}^i\|},\\
    \end{align*}
    where $\hat \lambda_i$ is rayleigh estimation of $\hat{u}_{q+1}^i$, $\Pi=U_qU_q^\top $ for $U_q$ the orthonormal basis of $span\{u_q^1,u_q^2,\cdots,u_q^k\}$, and $\hat U_{i-1}=orth([ u_{q+1}^1, u_{q+1}^2,\cdots, u_{q+1}^{i-1}])$. Denote $\hat U^{all}_{q+1}=[\hat u_{q+1}^1,\hat u_{q+1}^2,\cdots,\hat u_{q+1}^i]$. 

    Note that $\hat U_{q+1}^{all}$ is only a rotation of $U_q$, and thus $\|V_2^\top  \hat U_{q+1}^{all}\|=\|V_2^\top U_q\|$. 
    
    We first focus on the second part in updating $\hat u_{q+1}^i$ to $u_{q+1}^i$. Denote $\hat \Pi_{i}=I-\hat U_{i-1}\hat U_{i-1}^\top$, then the updating has the form
    \begin{align}
    \label{decompose2}\scriptsize
        cu_{q+1}^i &= (\hat\Pi_i\hat A\hat\Pi_i-\hat\lambda_i I )^{-1}(\hat\Pi_i\hat A\hat\Pi_i- A)\hat{u}_{q+1}^i\notag\\
        &= (\hat\Pi_i\hat A\hat\Pi_i-\lambda_i I )^{-1}(\hat\Pi_i\hat A\hat\Pi_i- A)\hat{u}_{q+1}^i\notag\\
        &+((\hat\Pi_i\hat A\hat\Pi_i-\hat\lambda_i I )^{-1}-(\hat\Pi_i\hat A\hat\Pi_i-\lambda_i I )^{-1})(\hat\Pi_i\hat A\hat\Pi_i- A)\hat{u}_{q+1}^i.
    \end{align}
    where $c=\|(\hat\Pi_i\hat A\hat\Pi_i-\hat\lambda_i I )^{-1}(\hat\Pi_i\hat A\hat\Pi_i- A)\hat{u}_{q+1}^i\|$ is the normalization constant. Similar to analysis in EPSI, the second term in \eqref{decompose2} is a quadratic term of $\|V_2^\top \hat u_{q+1}\|$, since
    \begin{align*}
        &\frac{1}{c}\|((\hat\Pi_i\hat A\hat\Pi_i-\hat\lambda_i I )^{-1}-(\hat\Pi_i\hat A\hat\Pi_i-\lambda_i I )^{-1})(\hat\Pi_i\hat A\hat\Pi_i- A)\hat{u}_{q+1}^i\|\\
        =&\|I-((I-\hat U_{i-1}\hat U_{i-1}^\top)\hat A(I-\hat U_{i-1}\hat U_{i-1}^\top)-\lambda_i I )^{-1}((I-\hat U_{i-1}\hat U_{i-1}^\top)\hat A(I-\hat U_{i-1}\hat U_{i-1}^\top)-\hat\lambda_i I )\|\\
        \leq&\frac{2}{\eta}|\lambda_i-\hat \lambda_i|.
    \end{align*}
    Note that 
    \begin{align*}
        |\lambda_i-\hat \lambda_i|&=|\lambda_i-\hat u_{q+1}^\top V\Lambda V^\top \hat u_{q+1}|\\
        &=|\lambda_i-\hat u_{q+1}^\top V_1\Lambda_1 V_1^\top \hat u_{q+1}-\hat u_{q+1}^\top V_2\Lambda_2 V_2^\top \hat u_{q+1}|\\
        &\leq \|(\lambda_i I-\Lambda_1)V_1^\top \hat u_{q+1}\|+\lambda_{k+1}\|V_2^\top \hat u_{q+1}\|^2\\
        &\leq 18\lambda_1\|V_2^\top \hat U_{q+1}^{all}\|.
    \end{align*}
    Combining with lemma \ref{correction} and lemma \ref{lazyepsi}, we have
\begin{align*}
    c\frac{\|V_2^\top  u^i_{q+1}\|}{\|u_{q+1}\|}&\leq \frac{3.5\eta}{\lambda_i-\lambda_{k+1}}\|V_2^\top \hat u^i_{q+1}\|+\frac{c_1(\lambda_i-\lambda_n)}{(\lambda_i-\lambda_{k+1})^2}\|(\lambda_i I-\Sigma_1)V_1^\top \hat u^i_{q+1}\|+\frac{36\lambda_1}{\eta}\|V_2^\top \hat U_{q+1}^{all}\|^2\\
    &\leq \frac{3.5\eta}{\lambda_i-\lambda_{k+1}}\|V_2^\top \hat u^i_{q+1}\|+(\frac{17c_1\lambda_1(\lambda_i-\lambda_n)}{(\lambda_i-\lambda_{k+1})^2}+\frac{36\lambda_1}{\eta})\|V_2^\top \hat U_{q+1}^{all}\|^2,
\end{align*}
where $c$ is scaling constant $1-\epsilon_0$ in lemma \ref{lazyepsi}.
Let $T=[\frac{u^1_{q+1}}{\|u_{q+1}\|},\frac{ u^2_{q+1}}{\|u_{q+1}\|},\cdots,\frac{u^k_{q+1}}{\|u_{q+1}\|}]$, then by cauchy-schwarz 
\begin{align*}
    c\|V_2^\top  T\|_F&\leq \sqrt{2}(\frac{3.5\eta}{\lambda_k-\lambda_{k+1}}\|V_2^\top \hat U_{q+1}^{all}\|_F+(\frac{17\sqrt{k}c_1\lambda_1(\lambda_1-\lambda_n)}{(\lambda_k-\lambda_{k+1})^2}+\frac{36\lambda_1}{\eta}\|V_2^\top \hat U_{q+1}^{all}\|_F^2)\\
\end{align*}

Finally, we show that the difference between the orthogonal basis $U_{q+1}$ and $T$ in $V_2$ space is small. Note that $T=U_{q+1}R$ by Gram-Schmidt orthogonalization. Since $\langle u_{q+1}^i,u_{q+1}^j\rangle = c\langle\hat u_{q+1}^i+a^i,\hat u_{q+1}^j+a^j\rangle$, where $\langle\hat u_{q+1}^i,\hat u_{q+1}^j\rangle=0$ and $\|a^i\| = \|((I-\hat U_{i-1}\hat U_{i-1}^\top)\hat A(I-\hat U_{i-1}\hat U_{i-1}^\top)-\hat\lambda_i I )^{-1}(\hat \lambda_i I- A)\hat{u}_{q+1}^i\|\leq \frac{\lambda_1\|V_2^\top \hat U_{q+1}^{all}\|}{\eta}$, then $\|R-I\|\leq\frac{c\lambda_1k\|V_2^\top \hat U_{q+1}^{all}\|}{\eta}$ for some constant $c$ and thus
\begin{align*}
    \|V_2^\top U_{q+1}\|_F=\|V_2^\top TR\|_F\leq \|V_2^\top T\|_F\|R\|\leq  \|V_2^\top T\|_F(1+\frac{c\lambda_1k\|V_2^\top \hat U_{q+1}^{all}\|}{\eta})
\end{align*}

Finally, with $\|V_2^\top  \hat U_{q+1}^{all}\|=\|V_2^\top U_q\|$ and set $c_0=(1+\frac{c\lambda_1k\|V_2^\top U_q\|}{\eta})(1-\frac{4(\lambda_1-\lambda_{k+1})\|V_2^\top U_q\|}{\eta})=1+\frac{c\lambda_1 k\|V_2^\top U_q\|}{\eta}$ for some small constant $c$, we have
\begin{align*}
    \|V_2^\top  U_{q+1}\|_F&\leq c_0(\frac{3.5\sqrt{2}\eta}{\lambda_k-\lambda_{k+1}}\|V_2^\top U_q\|_F+(\frac{17\sqrt{2k}c_1\lambda_1(\lambda_1-\lambda_n)}{(\lambda_k-\lambda_{k+1})^2}+\frac{36\lambda_1}{\eta})\|V_2^\top U_q\|_F^2)\\
\end{align*}}

\rmkb{Theorem \ref{thm:Convergencelazy} establishes that \texttt{Lazy-EPSI} exhibits a linear-quadratic convergence pattern. Specifically, the linear term $\frac{3.5\sqrt{2}\lambda_k}{\lambda_k-\lambda_{k+1}}\frac{\eta}{\lambda_k}\|V_2^\top U_q\|_F$ dominates the convergence behavior, while the contributions from the quadratic terms $(\frac{17\sqrt{2k}c_1\lambda_1(\lambda_1-\lambda_n)}{(\lambda_k-\lambda_{k+1})^2}+\frac{36\lambda_1}{\eta})\|V_2^\top U_q\|_F^2$ is relatively minor. Though there is a $\frac{1}{\eta}$ factor in quadratic terms, we remark that the initial guess generated by Nystrom approximation with $\eta$ distortion lies in an $\eta$-neighborhood of true solution, thus justify the validity of the term.

Notably, the linear convergence rate depends solely on $\frac{\lambda_k}{\lambda_k-\lambda_{k+1}}$, rather than on the intermediate rate $\max\{\frac{\lambda_1}{\lambda_1-\lambda_2},\cdots,\frac{\lambda_k}{\lambda_k-\lambda_{k+1}}\}$. 

}

\section{\texttt{EPSI} as Sketch-and-Precondition}
\label{section:EPSIasNewton}

As shown in our theoretical analysis (Theorem \ref{thm:Convergence}), the convergence rate of \texttt{EPSI} (Error-Powered Sketched Inverse Iteration) improves when one employs a (random) embedding with better distortion, mirroring the benefits of sketch-and-precondition algorithms for over-parameterized least squares. To illuminate the preconditioning effect of \texttt{EPSI}, we can interpret \texttt{EPSI} as a form of the Newton Sketch method \cite{xu2016sub,pilanci2017newton}—a randomized approach that accelerates second-order optimization by approximating the Hessian with a projected sketch, thereby reducing computation costs while maintaining rapid convergence. Indeed, Newton Sketch falls under the broader “sketch-and-precondition” paradigm because it constructs a preconditioner from a randomly sketched approximation of the Hessian, which accelerates the solver’s convergence. However, even though \texttt{EPSI} can be regarded as a Newton Sketch method when minimizing the Lagrangian form of the eigenvalue problem, \emph{this conceptual parallel does not allow us to directly transfer theoretical results from Newton Sketch to \texttt{EPSI}, depsite both methods exhibiting similar linear-quadratic convergence rates}. The key distinctions are:
\vspace{-0.1in}
\begin{itemize}[itemindent=0.5cm,leftmargin=0.5cm]
\setlength{\itemsep}{0pt}
\setlength{\parsep}{0pt}
\setlength{\parskip}{0pt}
    \item The local convergence regime and rate of Newton Sketch methods \cite{pilanci2017newton} hinge on having a strictly positive smallest eigenvalue of the Hessian. However, in an eigenvalue problem, the Hessian of the Lagrangian with respect to the eigenvectors at optimality has a zero eigenvalue, rendering the analysis in\cite{pilanci2017newton} inapplicable. A related phenomenon occurs with the power method (which can be viewed as gradient descent for the eigenvalue problem): despite its similarity to gradient descent, the singularity of the Hessian necessitates a specialized analysis beyond classical gradient-based convergence arguments.
    \item Extending $k$-SVD using a Newton sketch requires solving a non-convex problem, which limits its theoretical guarantees. In contrast, our \texttt{Lazy-EPSI} method achieves linear--quadratic convergence that
depends solely on the ratio $\frac{\lambda_{k} - \lambda_{k+1}}{\lambda_{k}}$,
mirroring the behavior of subspace iteration methods.
\end{itemize}
\vspace{-0.1in}
These distinctions make this connection insightful for understanding but insufficient as a basis for rigorous theoretical proofs.

Computationally, our approach leverages the Hessian structure in the eigenvalue problem, which retains the same form \(\lambda I - A\) (with varying \(\lambda\)) throughout the optimization procedure. Because this form remains consistent, a single low-rank approximation of \(A\) suffices for efficient inversion at each step. In contrast, the Newton Sketch method \cite{pilanci2017newton} requires solving a sketched optimization subproblem at every iteration. By utilizing the Nystr\"om approximation and applying the Woodbury identity (see Lemma~\ref{lemma:woodbury}), we achieve fast computation of the inverse in each iteration.

\paragraph{\texttt{EPSI} as Newton Sketch}

In this section, we show that \texttt{EPSI} can be viewed as a Newton Sketch \cite{xu2016sub,pilanci2017newton,bollapragada2019exact,li2020subsampled,ye2021approximate,erdogdu2015convergence}—an inexact second-order optimization method that employs randomized Hessian approximations by projecting high-dimensional data into a lower-dimensional space, thereby reducing computation while preserving rapid convergence.  The connection builds on the work of \cite{tapia2018inverse}, which shows that inverse power methods can be reinterpreted as Newton methods for solving the Lagrangian form of the eigenvalue problem $F(u,\lambda)=\frac{1}{2}u^\top A u-\frac{\lambda}{2}(u^\top u -1)$. Thus, if we have a matrix \(\hat A\) that approximates \(A\) (and whose inverse is easy to compute), we can perform the following iteration:
\[
u_{k+1} 
= u_k 
- \underbrace{(\hat A - \lambda I)^{-1}}_{\text{approximated }(\nabla^2 F)^{-1}}
\underbrace{(A - \lambda I) \, u_k}_{\nabla F}
= (\hat A - \lambda I)^{-1}(A - \hat A)\,u_k.
\]

%\subsection{Further Extensions}

While \texttt{EPSI} shows strong potential as a practical solver, it also opens the door to developing faster algorithms for eigenvector and singular vector computations. Beyond straightforward adaptations of more advanced (inexact) second-order methods—such as Newton-CG \cite{dembo1983truncated,byrd2011use,bollapragada2019exact,royer2020newton} or extend this line of work to broader applications, including canonical component analysis and generalized PCA \cite{allen2017doubly}, as well as streaming $k$-SVD \cite{allen2017first}. Below, we summarize several promising and technically nontrivial directions that we believe will shape future research in this field.

\section{Numerical Experiment}

In this section, we demonstrate the practicality of our \texttt{EPSI} framework on both synthetic random matrices and real-world application matrices\footnote{All experiments were conducted on a 16GB MacBook Pro with an M3 chip using MATLAB R2024a.}. We compare it to the block power method \cite{gu2015subspace, musco2015randomized} to validate our theoretical findings and illustrate the effect of preconditioning. We demonstrate the convergence of eigenspace either by $|\hat\lambda_i-\lambda_i|$, $\|A\hat{u}_i - \hat{\lambda}_i \hat{u}_i\|$ or the proportion of the estimated singular vector outside the top $k$ singular vectors, measured by $\|V_2^\top \hat{u}_i\|$, note that they are in some sense equivalent.

\vspace{-0.1in}
\paragraph{Dense Matrix - Synthetic Data}We generate test matrix $A$ as the following:
\vspace{-0.1in}
\begin{itemize}
\setlength{\itemsep}{0pt}
\setlength{\parsep}{0pt}
\setlength{\parskip}{0pt}    
\item Choose Haar random orthogonal matrices \( U = [U_1 \ U_2] \) in \( \mathbb{R}^{m \times m} \) and \( V \) in \( \mathbb{R}^{n \times n} \), and partition \( U \) so that \( U_1 \in \mathbb{R}^{m \times n} \).
    \item Set \( A := U_1 \Sigma V^T \) where \( \Sigma \) is a diagonal matrix with condition number $\kappa$, which has a exponential decay from $1$ to $\kappa^{-1}$. 
    %\item Form vectors \( w \) in \( \mathbb{R}^n \), \( z \) in \( \mathbb{R}^{m-n} \) with independent standard Gaussian entries.
    %\item Define the solution \( x := \frac{w}{\|w\|} \), residual \( r(x) = \beta \cdot U_2 z / \|U_2 z\| \), and right-hand side \( b := Ax + r(x) \).
\end{itemize}
\vspace{-0.15in}
% In the experiment, we set $m=50000,n=1000$ and the largest entry of $\Sigma$ to 1.000001 and followed by logarithmically equispaced entries between 1 and 1e-1. 
Figure \ref{fig:precondition} illustrates the preconditioning effect of \texttt{Lazy-EPSI}, while Figure \ref{fig:differentsektch} demonstrates that using more sketches accelerates the convergence of \texttt{Lazy-EPSI}.

\begin{figure}[h]
\centering
  \subfigure[Pre-condition Effect of \texttt{Lazy-EPSI}]{\includegraphics[width=0.45\textwidth]{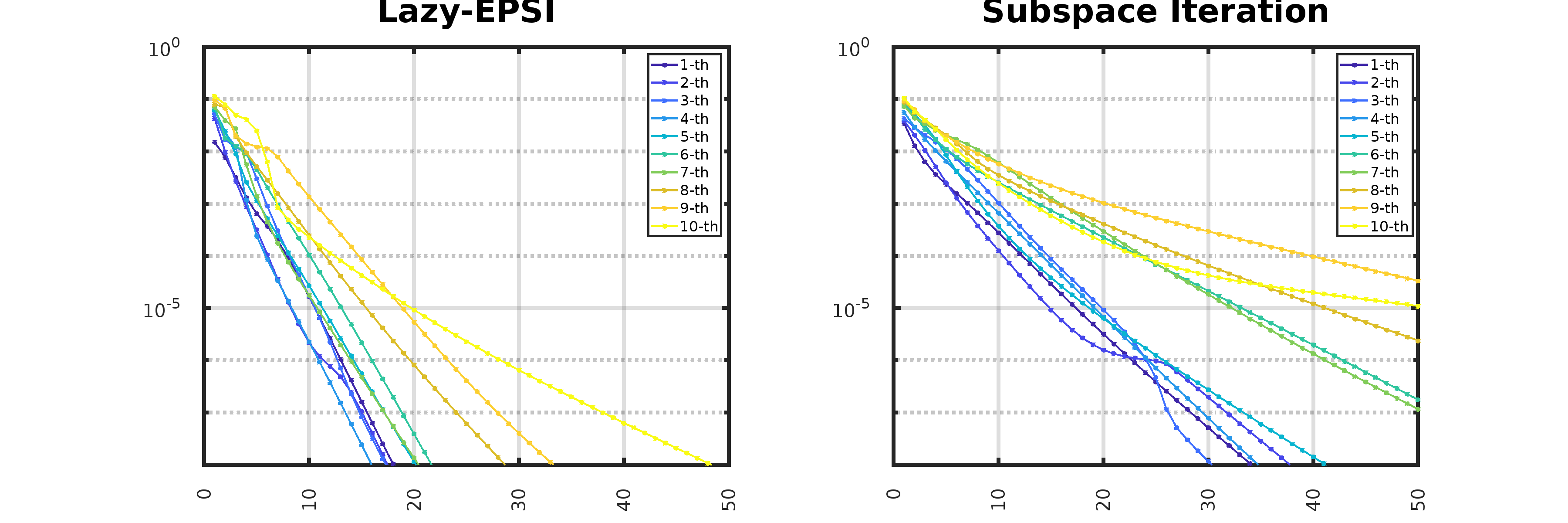}\label{fig:precondition}}
  \subfigure[Better $\hat A$ leads to faster convergence]{\raisebox{0.1in}{\includegraphics[width =0.45\textwidth]{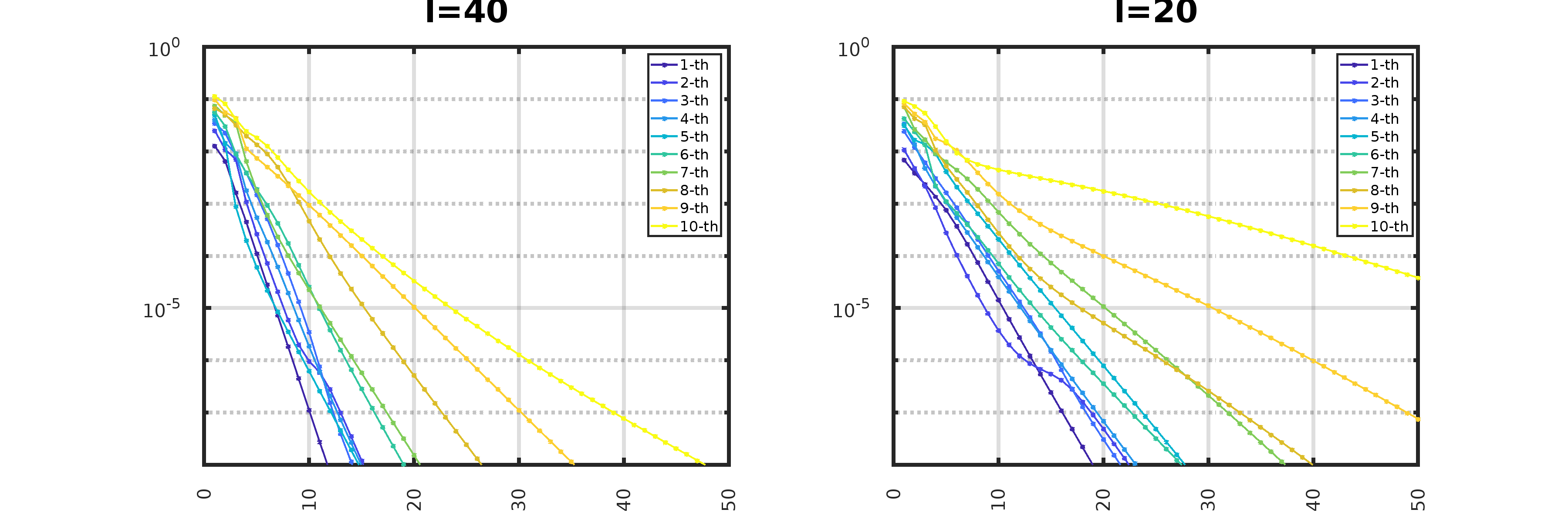}}\label{fig:differentsektch}}
  \vspace{-0.15in}
  \caption{\texttt{Lazy-EPSI} on Synthetic Matrix.}\label{fig:synthetic data}
  \vspace{-0.15in}
\end{figure}

\paragraph{Scalability of \texttt{EPSI} and Selection of the Shift} 
% Argentati et al.~\cite{argentati2017convergence} introduce a preconditioner \(\hat{A}\) that satisfies the bound \((1 - \eta)A \preceq \hat{A} \preceq (1 + \eta)A\). In contrast, as is mentioned in previous text, our approach constructs the preconditioner around \(A-2\eta\), introducing a key distinction that influences the stability and effectiveness of our method. Specifically, we incorporate a negative shift in our preconditioning strategy either by scaling the preconditioner matrix or shift the diagonal matrix of precondition. This negative shift serves a crucial purpose: it ensures that the preconditioner remains invertible, thus enhancing the numerical stability of the algorithm. The choice of shift is guided by the quality of the sketching approximation, which we quantify using the expected error of randomized SVD \cite{halko2011finding}. Thereby maintaining robustness and optimizing the trade-off between approximation accuracy and computational efficiency.

In Figure \ref{fig:scaling}, we demonstrated the scalability of \texttt{Lazy-EPSI} with a sketch size $\ell=10k$. Lines with different color show the convergence of the $i-th$ eigenvalue estimation error $\|\hat \lambda_i-\lambda_i\|$. The first 400 eigenvalue of synthetic matrix has a exponential decay from $1$ to $10^{-3}$ and set all all other eigenvlaues set to be $10^{-3}$. We demonstrate that our convergence rate only depends on the eigenvalue gap but not depend on the size of the matrix. 

\begin{figure}[H]
\centering
\vspace{-0.1in}
  \subfigure[\texttt{$(m,n)=(2000,2000)$}]{\includegraphics[width=0.8\textwidth]{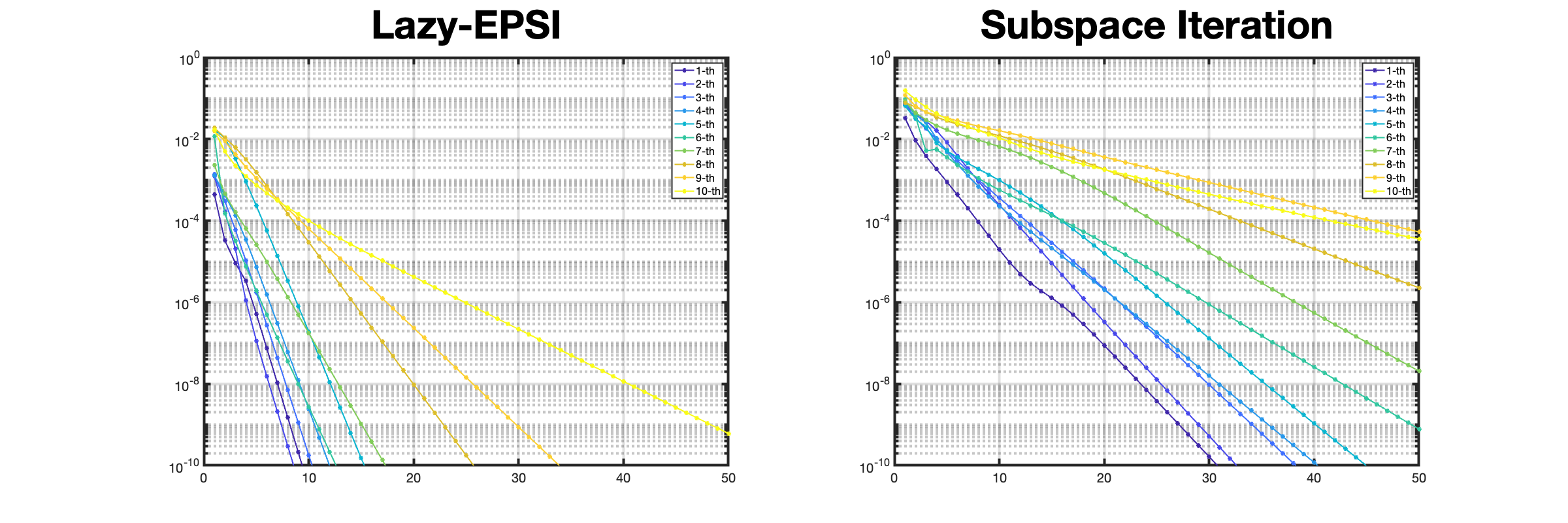}\label{fig:synthetic2000}}\\
  \subfigure[\texttt{$(m,n)=(4000,4000)$}]{\includegraphics[width =0.8\textwidth]{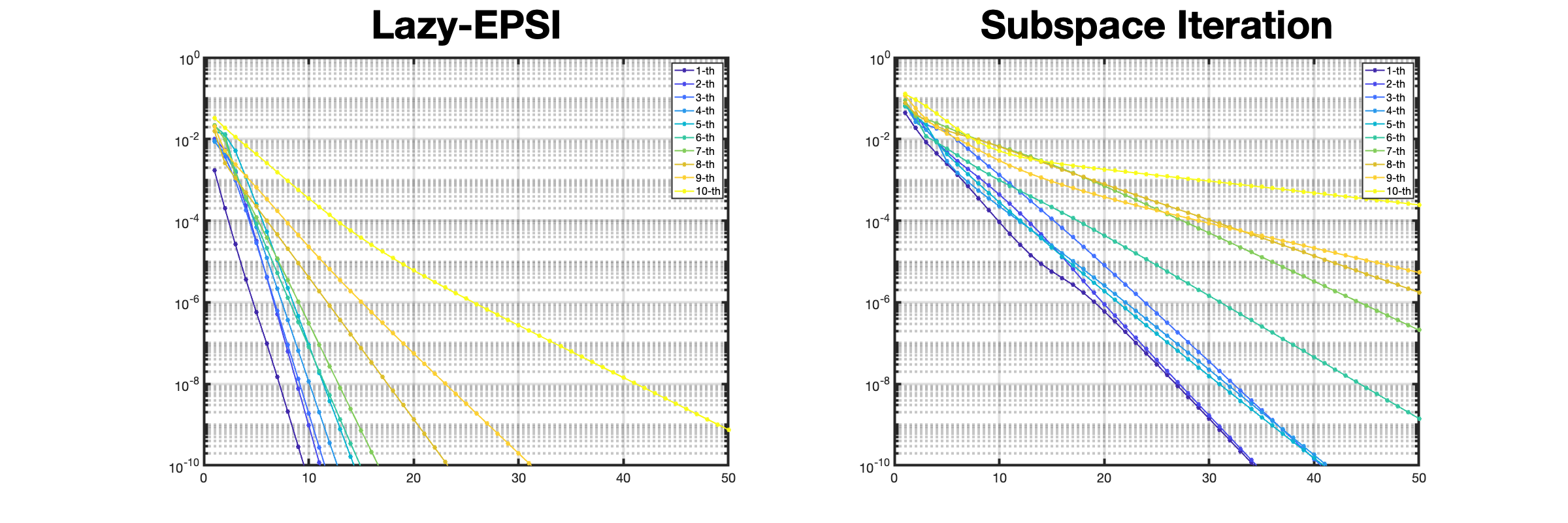}\label{fig:synthetic4000}}\\
  \subfigure[\texttt{$(m,n)=(8000,8000)$}]{\includegraphics[width =0.8\textwidth]{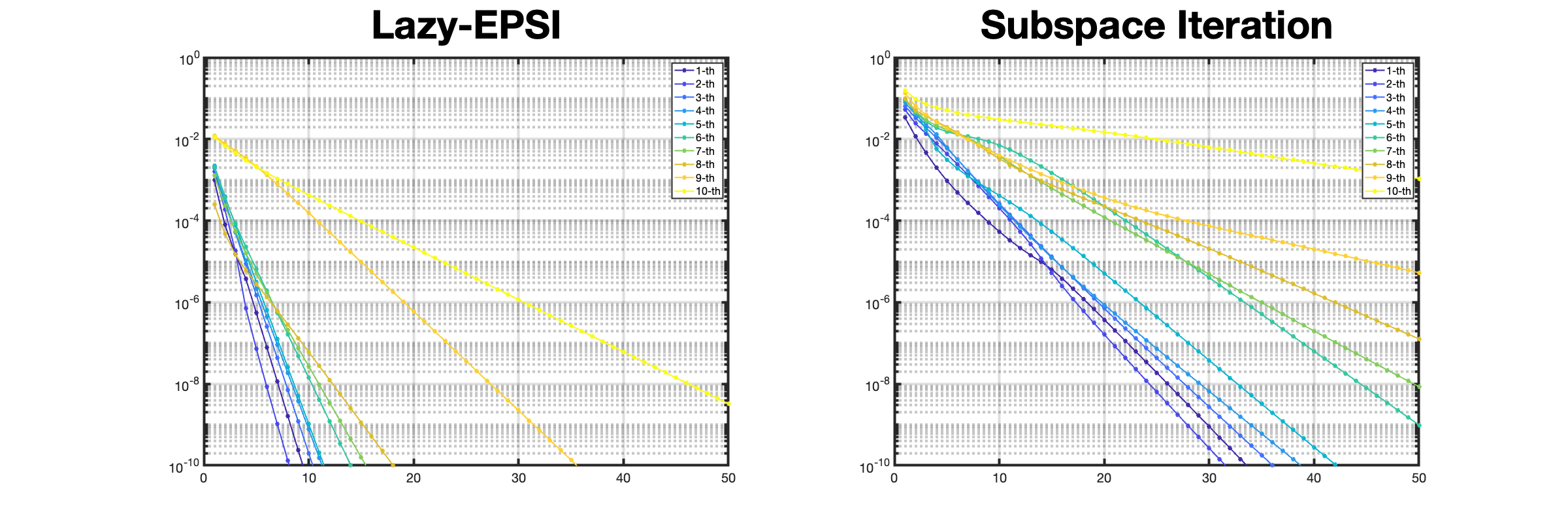}\label{fig:synthetic8000}}
  \vspace{-0.1in}
  \caption{Comparison of \texttt{Lazy-EPSI} and Subspace Iteration on synthetic squared matrix with respect to the error of rayleigh quotient eigenvalue estimation $|\hat\lambda_i-\lambda_i|$.}\label{fig:scaling}
  \vspace{-0.1in}
\end{figure}

\paragraph{runtime comparison between EPSI and Davidson's method} We compare the runtime of EPSI, Davidson's method in \cite{1975JCoPh..17...87D} and inexact RQI in \cite{simoncini2002inexact} with conjugate gradient inner solver and power method as initialization on solving the first eigenpair of a synthetic dense matrix. Note that EPSI and inexact RQI have a substantially fast convergence rate than Davidson's method. However, inexact RQI is hard to implement in practice as it needs a good initialization, otherwise it can easily converge to the other eigenpair, and this explains a longer initialization of inexact RQI in the figure \eqref{fig:scaling_Ux}.

\begin{figure*}[h]
\centering
\vspace{-0.1in}

  \subfigure[\texttt{$n=2000,\kappa=10^{3}$}]{\includegraphics[width=0.32\textwidth]{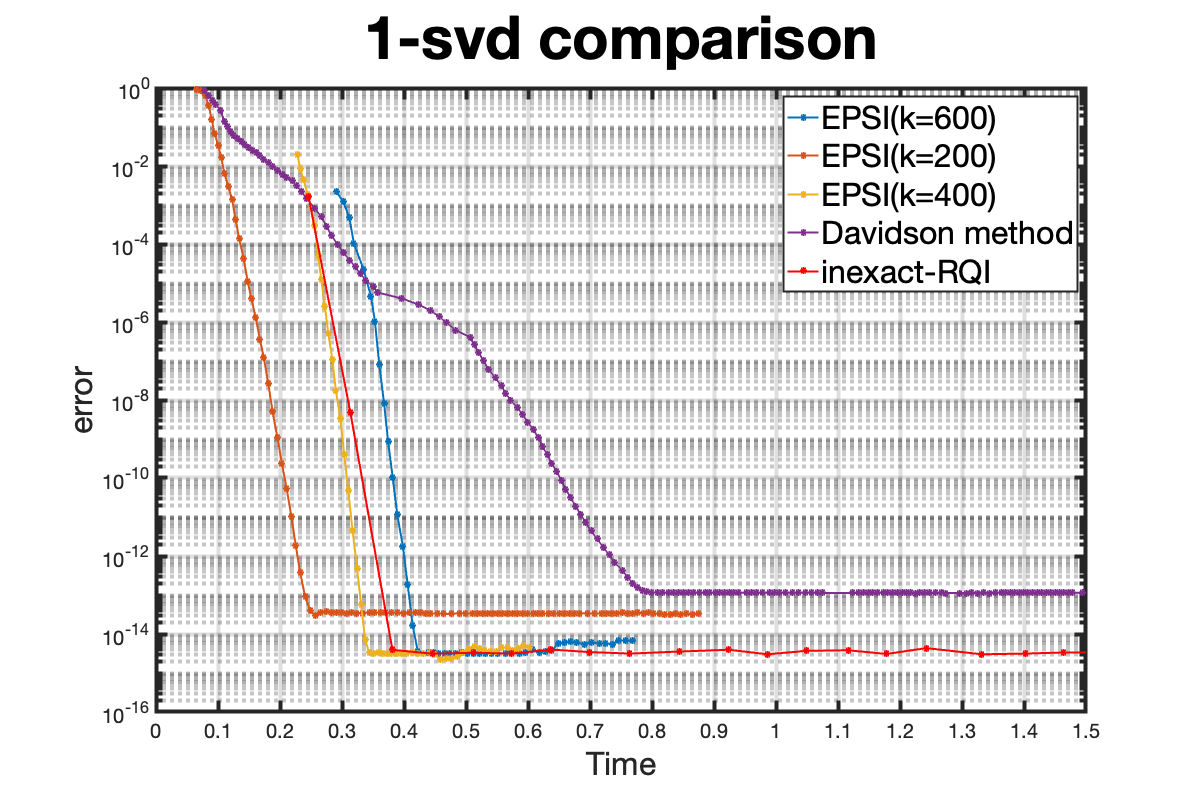}\label{fig:}}
  \subfigure[\texttt{$n=2000,\kappa=10^{6}$}]{\includegraphics[width =0.32\textwidth]{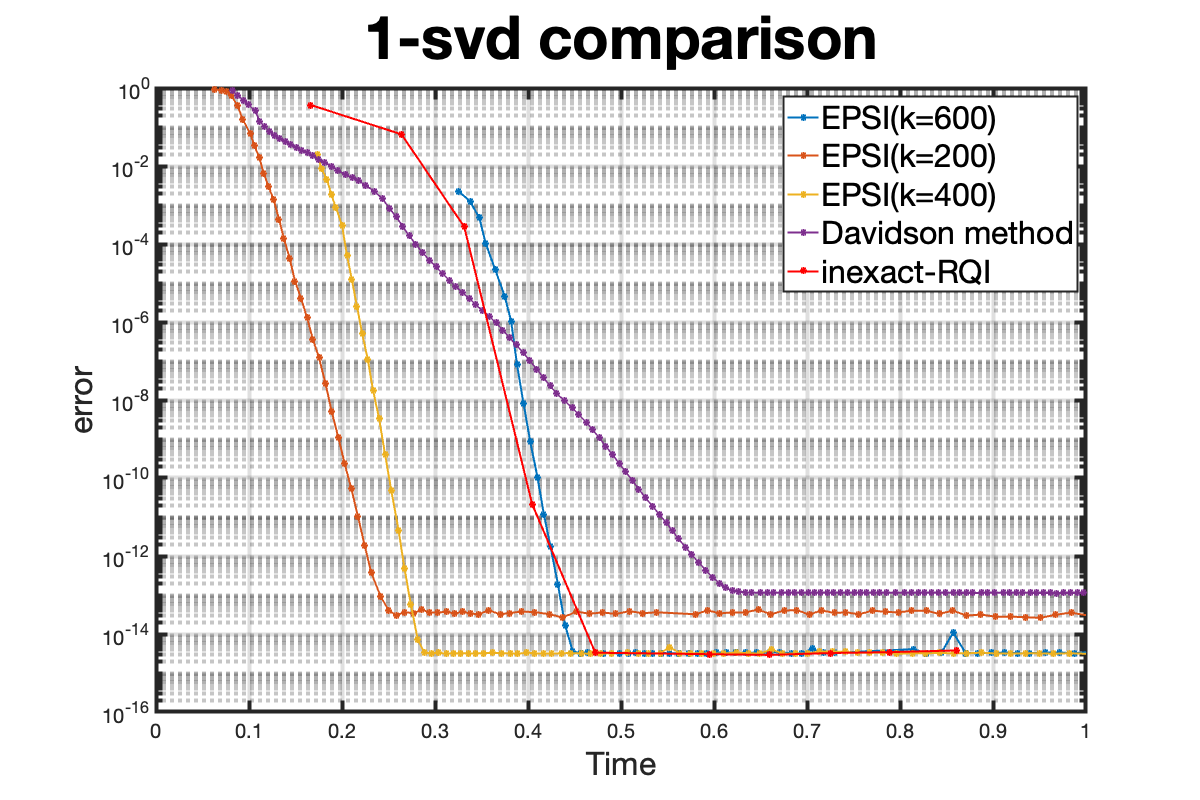}\label{fig:synthetic4000}}
  \subfigure[\texttt{$n=4000,\kappa=10^{6}$}]{\includegraphics[width =0.32\textwidth]{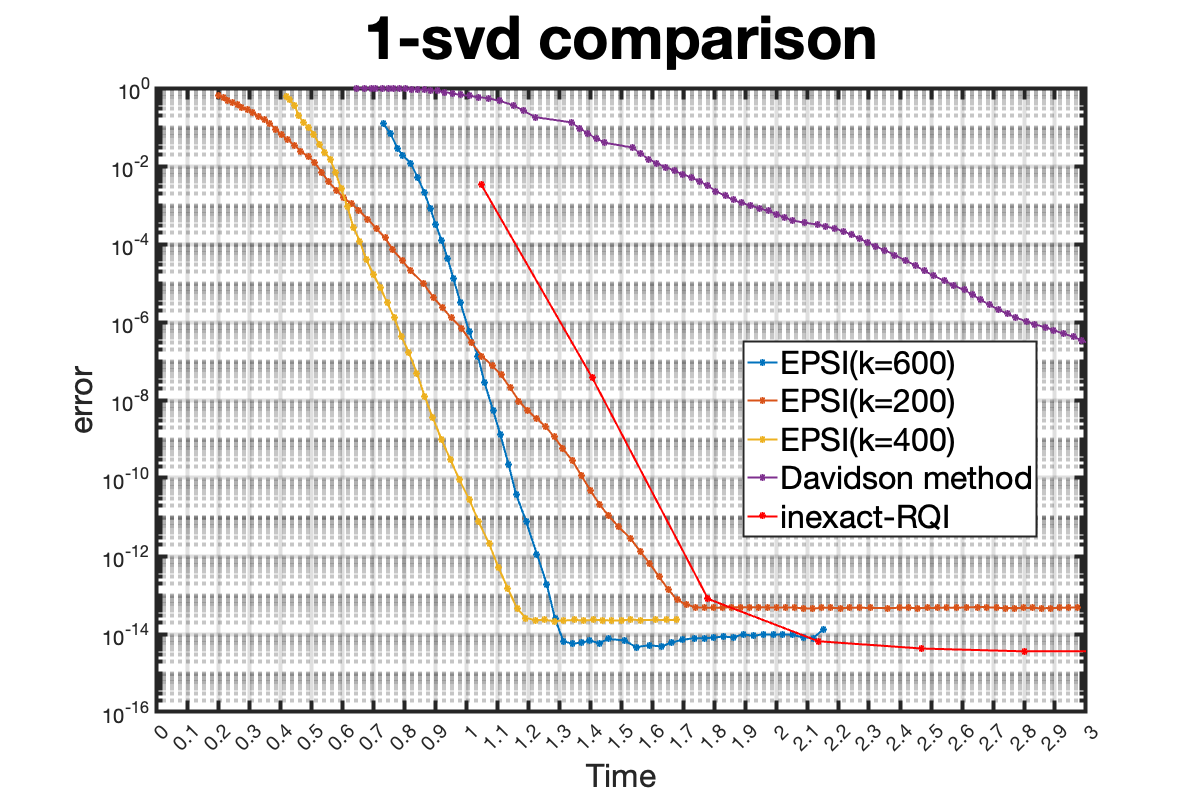}\label{fig:synthetic4000}}
  \vspace{-0.1in}
  \caption{Comparison of EPSI with k-dim Nystrom approximation and Davidson's method on synthetic $2000\times 2000$ matrix with respect to the convergence of component on the last $n-1$ eigenvectors $\|V_2^\top \hat u_i\|$.}\label{fig:scaling_Ux}
  \vspace{-0.1in}
\end{figure*}

\paragraph{runtime comparison between EPSI and Davidson's method on low rank data} In many cases, the data matrix $A\in \mathbb{R}^{n\times n}$ can be decomposed as $A=L_1L_2^\top + E$, where $L_1L_2^\top $ is the underlying low rank ground truth with intrinsic dimension $s\ll n$, and $E$ is white noise. In this case, the eigenvalue of $AA^\top$ will decay really quickly after the first $s$ dimension. We compare EPSI with Davidson's method and inexact RQI in this type of matrix in figure \eqref{fig:low_rank_data}. We use $randn()$ function in matlab to generate $L_1,L_2$ with std $\sigma_1$ and $E$ with std $\sigma_2$. Matrix with different $\sigma_1,\sigma_2,s$ has different first $s$ eigendecay and 'tail eigendecay', which is $\frac{\lambda_k}{\lambda_1}$ for $k$ larger than $s$. The result of experiments support that the performance of EPSI with a Nyström approximation replies on the eigendecay after roughly $\frac{k}{2}$, which makes EPSI powerful on low rank data with a small intrinsic dimension $s$ and low level of noise.

\begin{figure*}[h]
\centering
\vspace{-0.1in}

  \subfigure[\texttt{$s=100,\sigma_1=0.1,\sigma_2=0.05$}]{
  \includegraphics[width=0.26\textwidth]{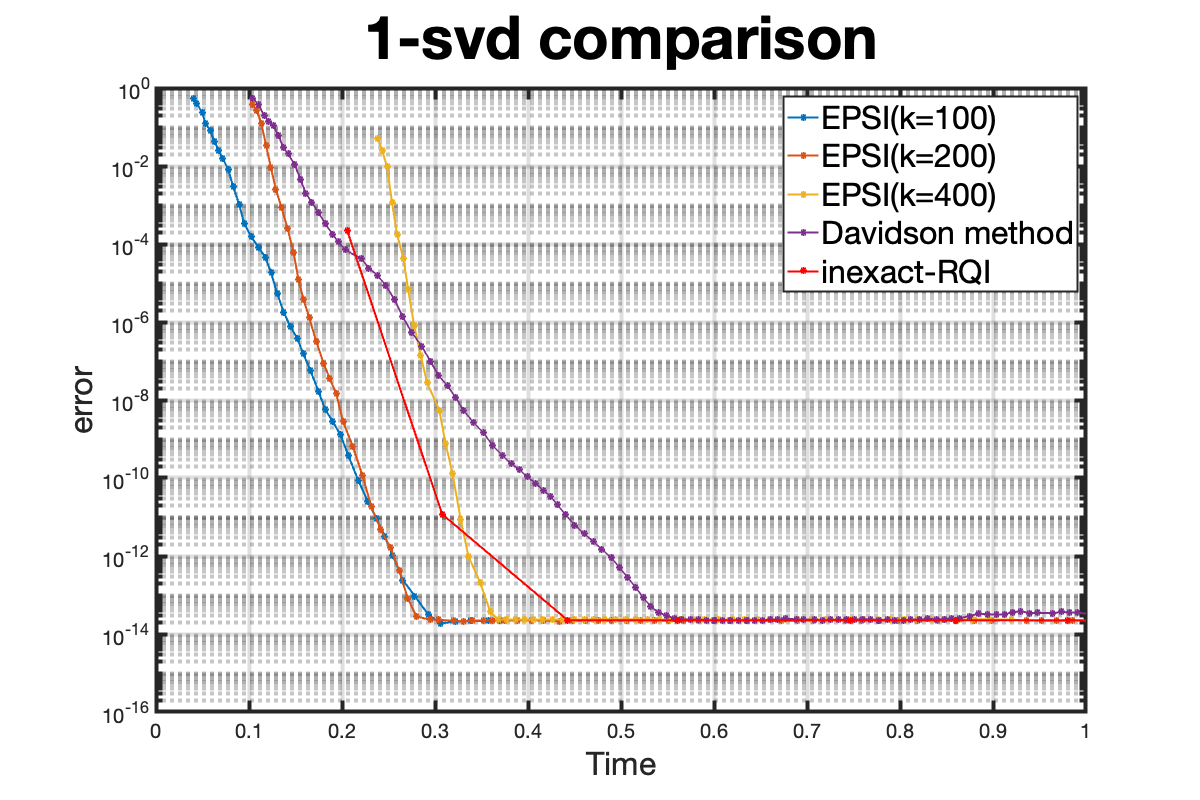}
  \includegraphics[width =0.18\textwidth]{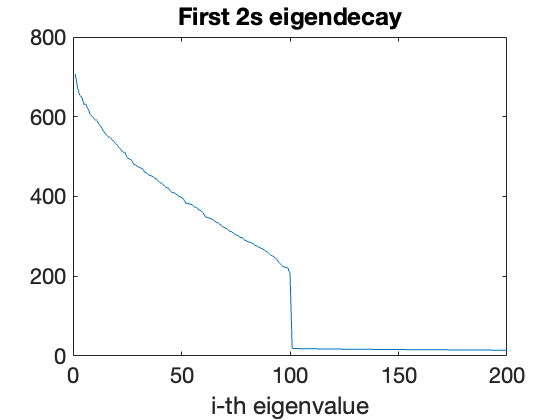}}
  \subfigure[\texttt{$s=50,\sigma_1=0.1,\sigma_2=0.05$}]{
  \includegraphics[width=0.26\textwidth]{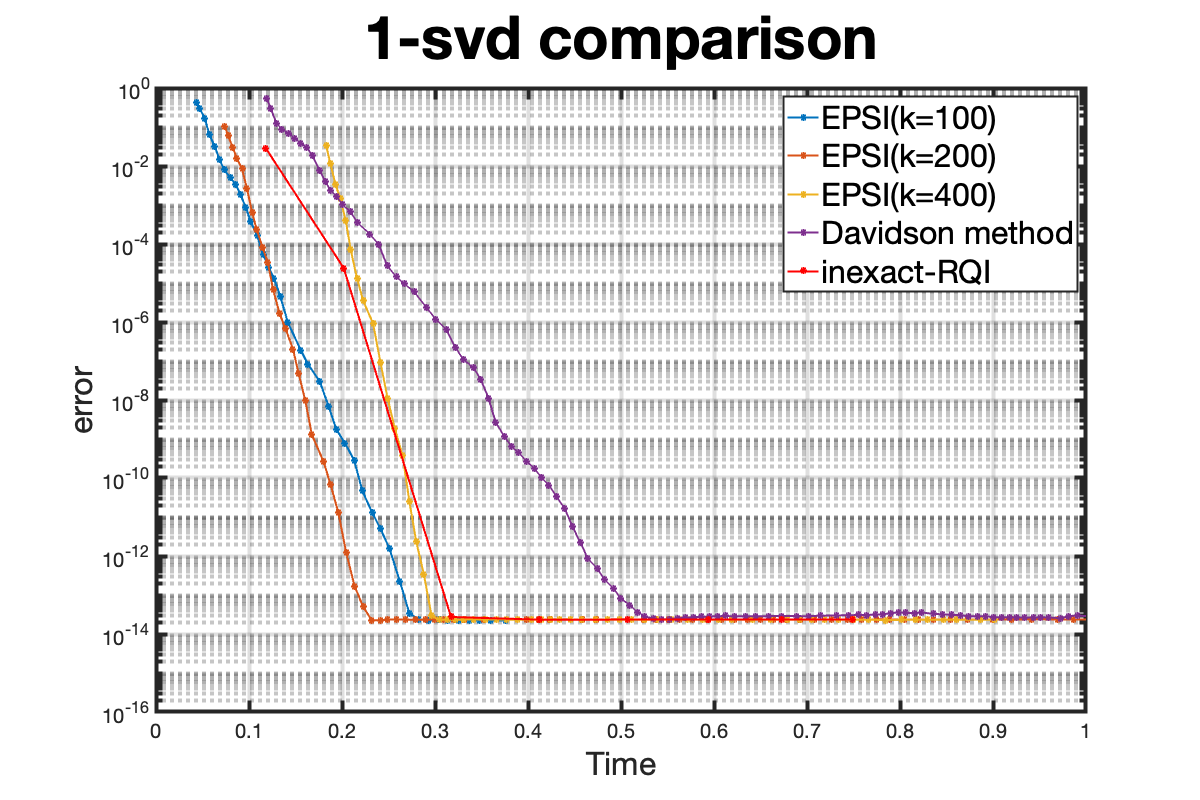}
  \includegraphics[width =0.18\textwidth]{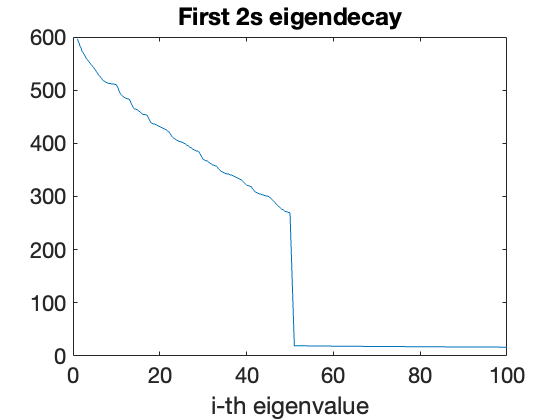}
  }\\
  \subfigure[\texttt{$s=100,\sigma_1=0.1,\sigma_2=0.1$}]{
  \includegraphics[width=0.26\textwidth]{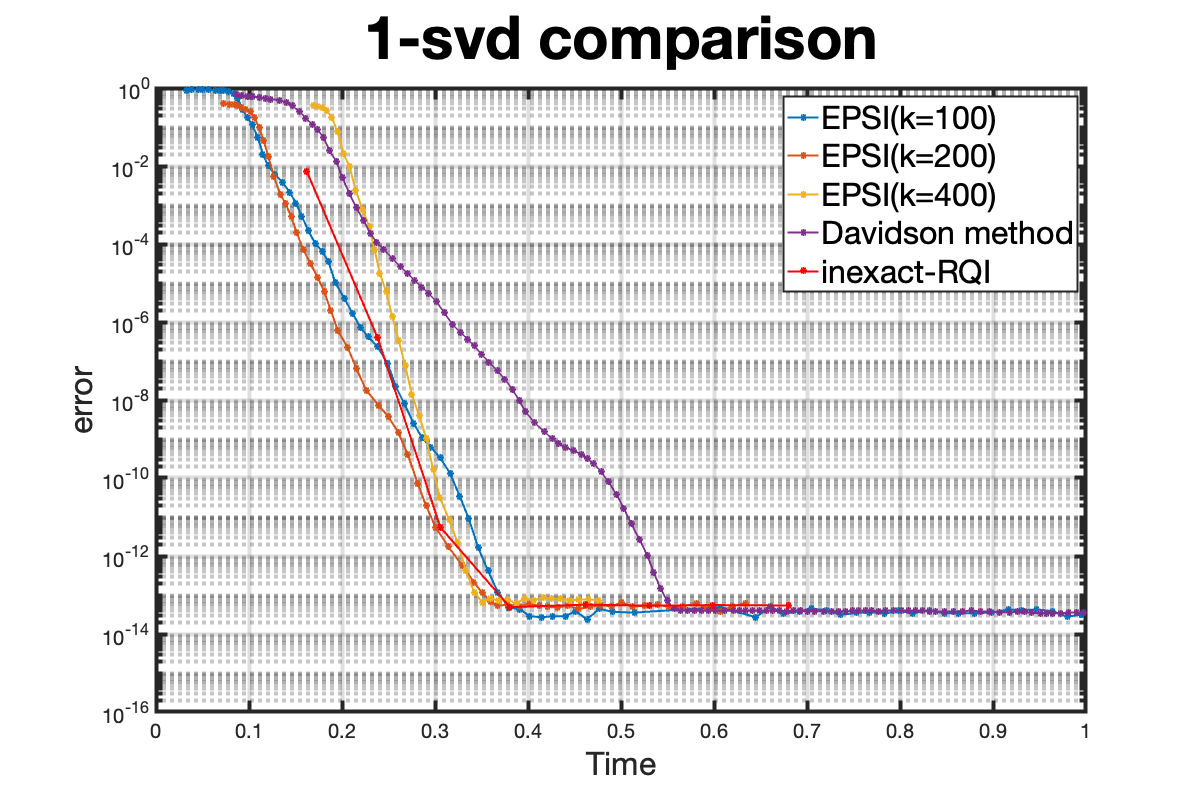}
  \includegraphics[width =0.18\textwidth]{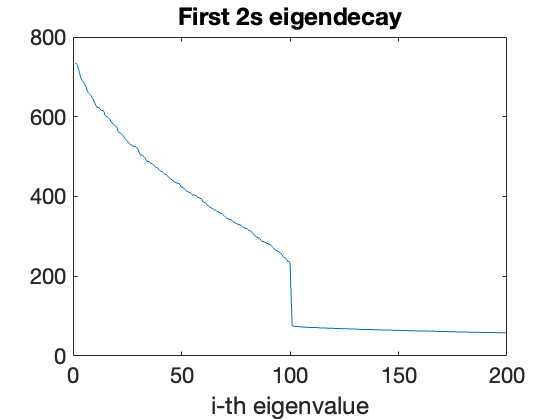}}
  \subfigure[\texttt{$s=200,\sigma_1=0.1,\sigma_2=0.05$}]{
  \includegraphics[width=0.26\textwidth]{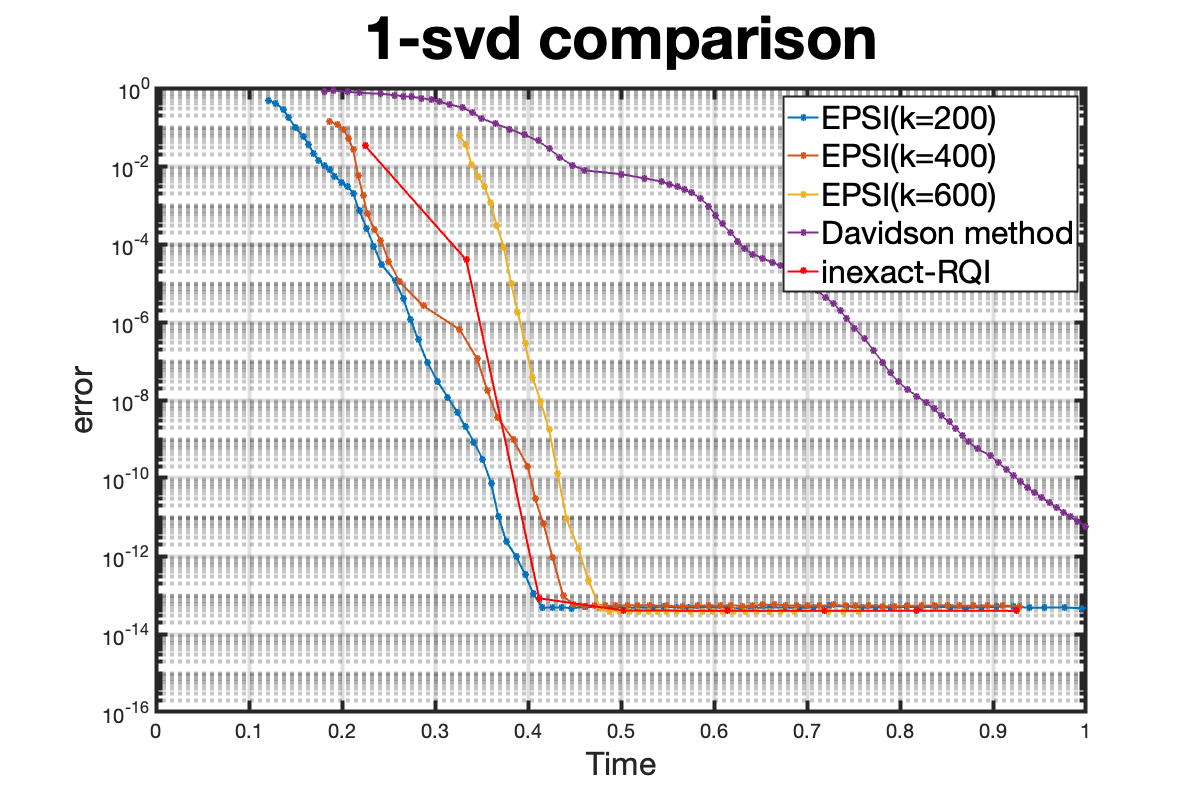}
  \includegraphics[width =0.18\textwidth]{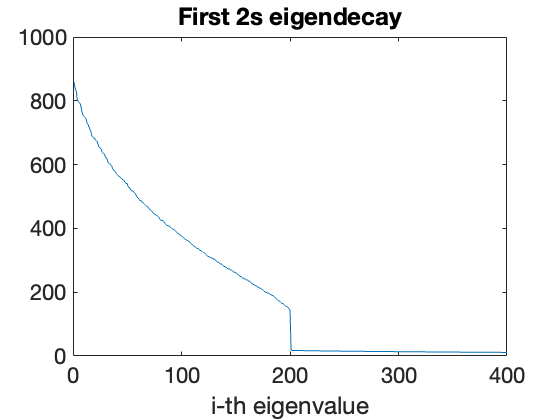}
  }
  \vspace{-0.1in}
  \caption{Comparison of EPSI with k-dim Nystrom approximation and Davidson's method on synthetic $2000\times 2000$ matrix with respect to the convergence of component on last $n-1$ eigenvectors $\|V_2^\top \hat u_i\|$.}\label{fig:low_rank_data}
  \vspace{-0.1in}
\end{figure*}

\paragraph{stability comparison between EPSI and inexact RQI} We have shown that EPSI shares a fast convergence as inexact RQI. However, inexact RQI is highly insensitive to initilization, which makes it less stable than EPSI. We compare the stability of EPSI with $k=200$ and inexact RQI by showing their performance statistics in 50 independent experiment on low rank matrix and dense matrix in figure \eqref{fig:stability}. 

\begin{figure*}[h]
\centering
\vspace{-0.1in}

  \subfigure[\texttt{low rank matrix}]{\includegraphics[width=0.4\textwidth]{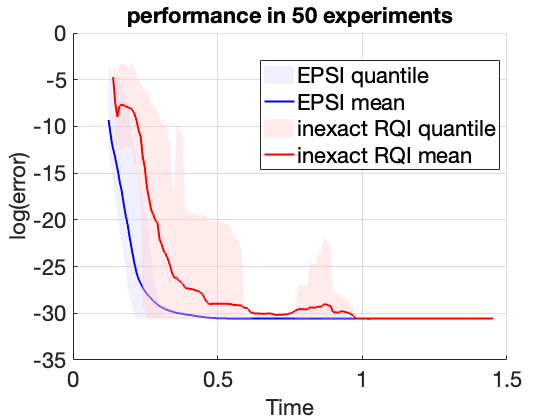}}
  \subfigure[\texttt{dense matrix}]{\includegraphics[width =0.4\textwidth]{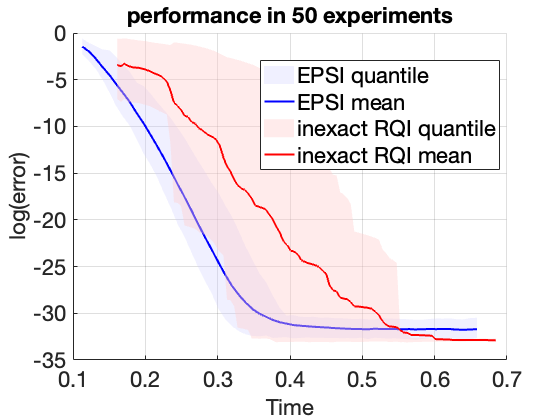}}
  \vspace{-0.1in}
  \caption{Comparison of EPSI and inexact RQI on low rank matrix (left) and dense matrix (right). The shadow area is the 0.1 quantile to 0.9 quantile of the convergence path in 50 independent experiments.  }\label{fig:stability}
  \vspace{-0.1in}
\end{figure*}

% \paragraph{Sparse Matrix - Laplacian Matrix} Following \cite{nakatsukasa2024fast}, we consider a symmetric test matrix $A$ of dimension $10^4$, obtained by discretizing the 2D Laplacian on a 
% $10^2\times 10^2$ grid \footnotemark. The Laplacian matrix is sparse because it is constructed using the five-point finite difference scheme, and its sparsity pattern is visualized in Figure \ref{fig:sparse}. Figure \ref{fig:lapconverge} illustrates the preconditioning effect of the \texttt{Lazy-EPSI} iterations. In all experiments, we set the sketch size to $l=100$.

%% Laplacian 
% \begin{figure}[h]
% \centering
% \vspace{-0.1in}
%   \subfigure[Pre-condition Effect of \texttt{Lazy-EPSI}]{\includegraphics[width=0.72\textwidth]{lazyepsilaplacian.png}\label{fig:lapconverge}}
%   \subfigure[{\scriptsize Laplacian}]{\raisebox{0.1in}{\includegraphics[width =0.2\textwidth]{lap.png}\label{fig:sparse}}}
%   \vspace{-0.1in}
%   \caption{\texttt{Lazy-EPSI} for Laplacian.}\label{fig:lap}
%   \vspace{-0.1in}
% \end{figure}

% \footnotetext{The matrix and eigenvectors were generated using the code available at \url{https://www.mathworks.com/matlabcentral/fileexchange/27279-laplacian-in-1d-2d-or-3d}.}

\paragraph{Sparse Matrix - Real World Data} We evaluate the practicality and preconditioning effect of \texttt{Lazy-EPSI} on real datasets \texttt{1138-bus} (a $1138\times 1138$ symmetric sparse matrix with 4054 non-zeros)\footnote{Our matrix data were obtained from https://sparse.tamu.edu/  } in figure \eqref{fig:realworld}, and in addition we reveal empirical evidence regarding the influence of different matrix structures on the convergence behavior of Davidson’s method. 

\begin{figure*}[h]
\centering
\vspace{-0.1in}
  \subfigure[\texttt{20 truncated svd}]{\includegraphics[width=0.32\textwidth]{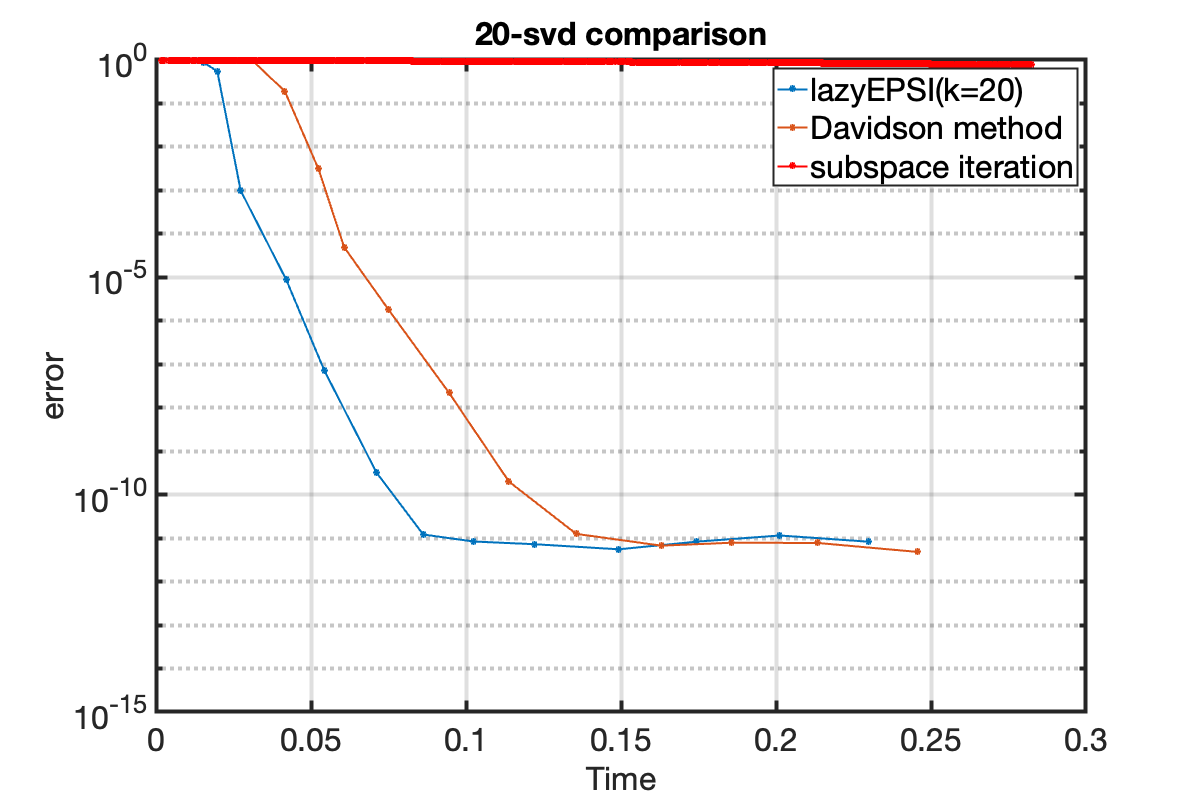}}
  \subfigure[\texttt{50 truncated svd}]{\includegraphics[width =0.32\textwidth]{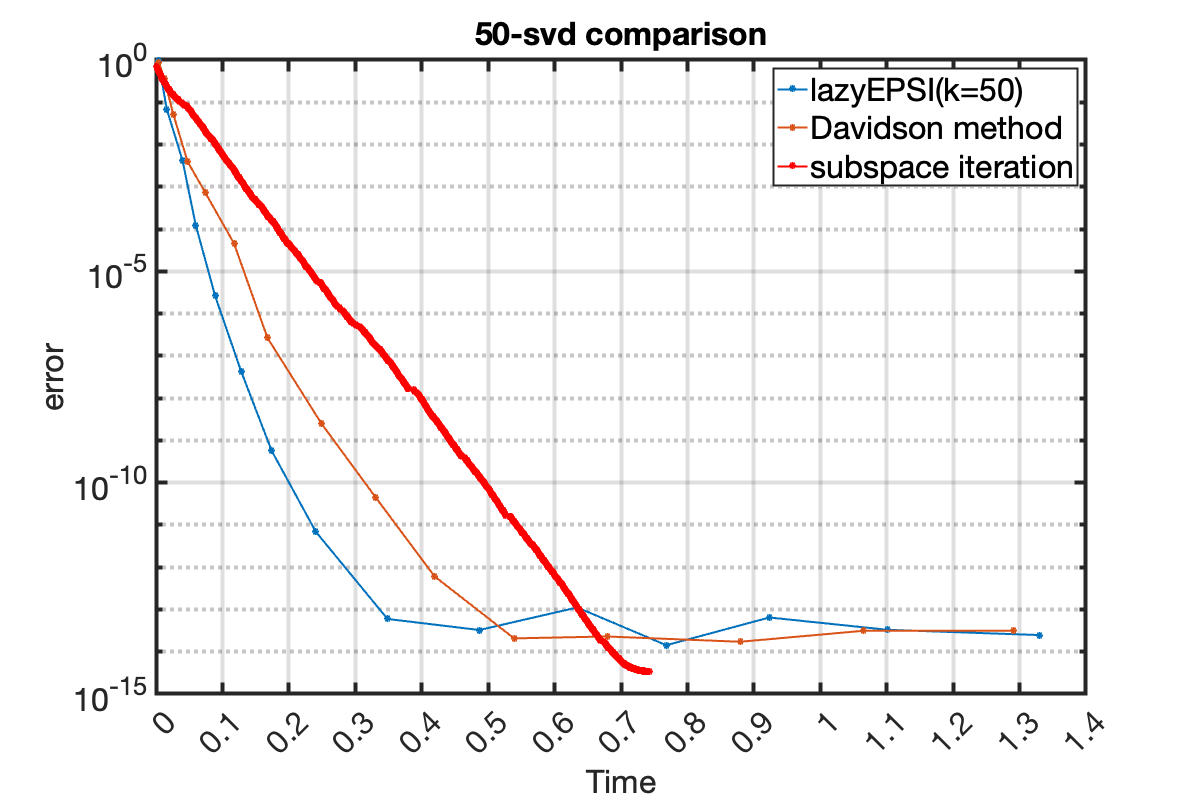}}
  \subfigure[\texttt{100 truncated svd}]{\includegraphics[width =0.32\textwidth]{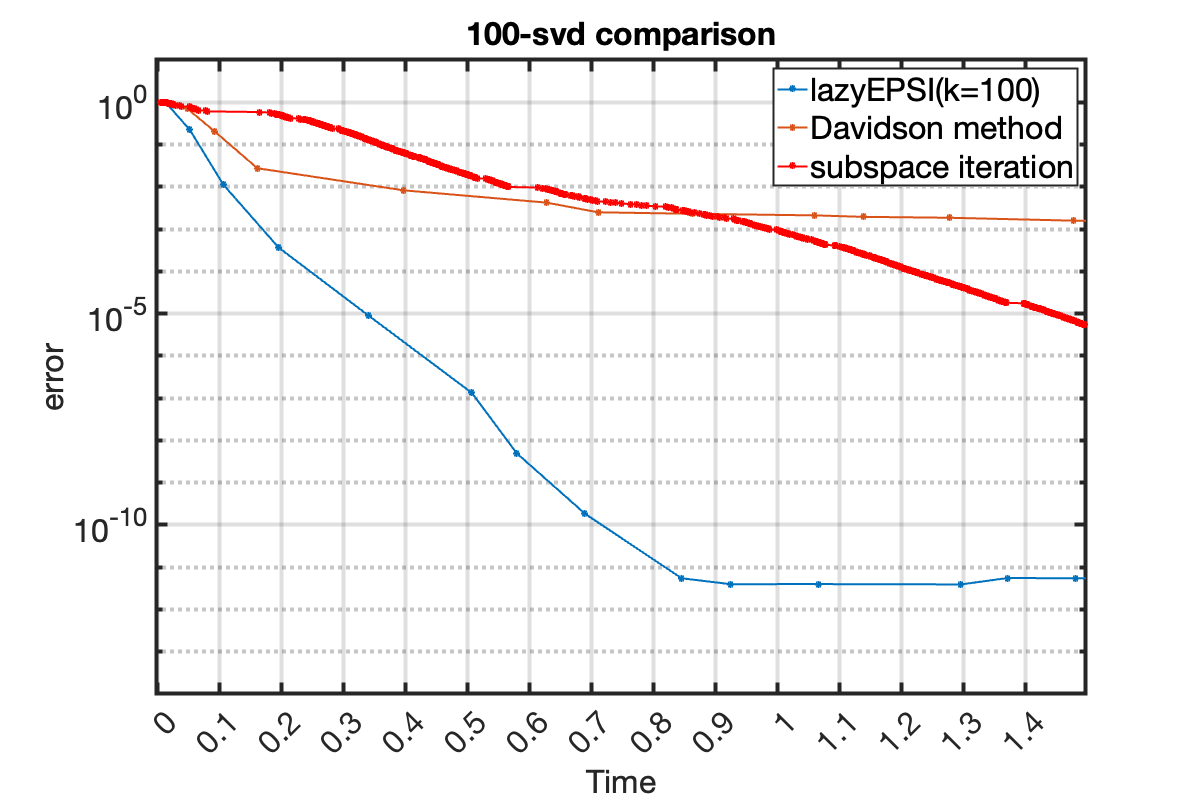}}\\
  \subfigure[\texttt{overview of data matrix}]{\includegraphics[width =0.32\textwidth]{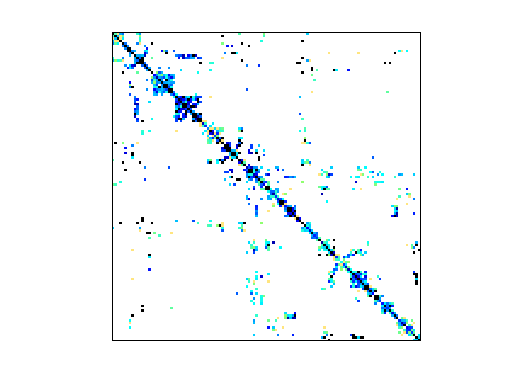}}
  \subfigure[\texttt{eigenvalues of data matrix}]{\includegraphics[width =0.32\textwidth]{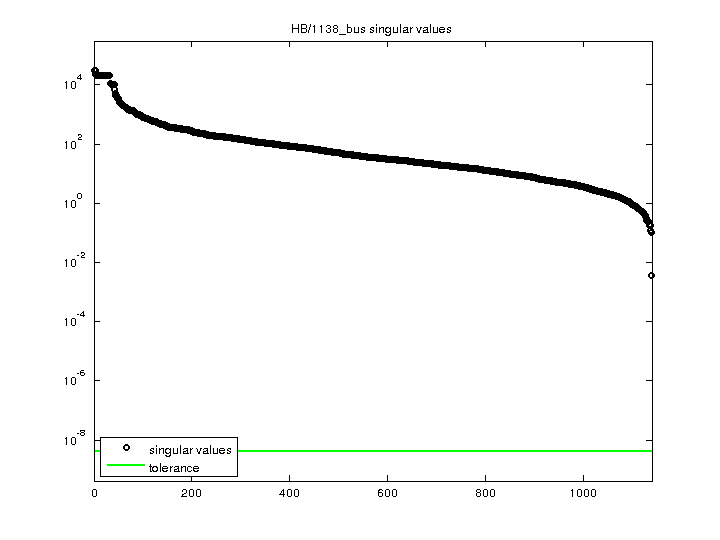}}
  \subfigure[\texttt{first 250 eigenvalues of data matrix}]{\includegraphics[width =0.32\textwidth]{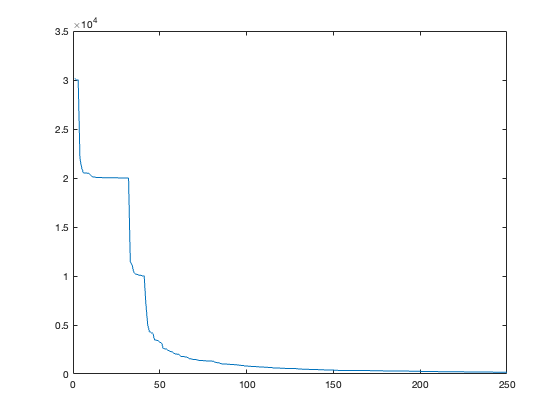}}
  \vspace{-0.1in}
  \caption{Comparison of \texttt{Lazy-EPSI}, Davidson's method and Subspace Iteration on 1138-bus dataset with respect to the error of space distance $\|V_2^\top U_k\|$, and the overview of the structure of data matrix. }\label{fig:realworld}
  \vspace{-0.1in}
\end{figure*}

Due to space constraints, additional experimental results are provided in the appendix.

\section{Conclusion}

In this work, we introduced \emph{Error-Powered Sketched Inverse Iteration} (\texttt{EPSI}), a novel approach to randomized eigenvalue computation that applies sketched inverse iteration to the sketching error rather than the estimated eigenvector. This ensures that the true eigenvector remains a fixed point of the iteration, making the quality of the randomized embedding influence convergence only as a preconditioner rather than determining it. Extending \texttt{EPSI} to compute the first $k$ singular vectors, we establish a convergence rate that depends solely on $\frac{\lambda_k}{\lambda_k-\lambda_{k+1}}$, eliminating dependence on intermediate spectral gaps through a novel \emph{orthogonalization step}. Unlike Newton Sketch, which requires solving a sketched subproblem at each iteration, our method efficiently leverages the \emph{Nyström approximation} and the \emph{Woodbury identity} for computational efficiency. Finally, we provide theoretical guarantees, demonstrating at least linear improvement in convergence with sketch size—marking, to the best of our knowledge, the first such result for eigenvalue computation. These contributions position \texttt{EPSI} as a flexible and theoretically robust framework for scalable spectral computations.

\section*{Acknowledgement}
The authors would like to thank Yifan Chen, Yijun Dong, Jiajin Li, Jorge Nocedal and Lexing Ying for valueable feedback on the work. 

\bibliographystyle{alpha}
\bibliography{references} % see references.bib for bibliography management

\newcommand{\etalchar}[1]{$^{#1}$}
\begin{thebibliography}{MDM{\etalchar{+}}23}

\bibitem[AKN{\etalchar{+}}17]{argentati2017convergence}
Merico~E Argentati, Andrew~V Knyazev, Klaus Neymeyr, Evgueni~E Ovtchinnikov, and Ming Zhou.
\newblock Convergence theory for preconditioned eigenvalue solvers in a nutshell.
\newblock {\em Foundations of Computational Mathematics}, 17:713--727, 2017.

\bibitem[AM15]{alaoui2015fast}
Ahmed Alaoui and Michael~W Mahoney.
\newblock Fast randomized kernel ridge regression with statistical guarantees.
\newblock {\em Advances in neural information processing systems}, 28, 2015.

\bibitem[AMT10]{avron2010blendenpik}
Haim Avron, Petar Maymounkov, and Sivan Toledo.
\newblock Blendenpik: Supercharging lapack's least-squares solver.
\newblock {\em SIAM Journal on Scientific Computing}, 32(3):1217--1236, 2010.

\bibitem[AZL16]{allen2016lazysvd}
Zeyuan Allen-Zhu and Yuanzhi Li.
\newblock Lazysvd: Even faster svd decomposition yet without agonizing pain.
\newblock {\em Advances in neural information processing systems}, 29, 2016.

\bibitem[AZL17a]{allen2017doubly}
Zeyuan Allen-Zhu and Yuanzhi Li.
\newblock Doubly accelerated methods for faster cca and generalized eigendecomposition.
\newblock In {\em International Conference on Machine Learning}, pages 98--106. PMLR, 2017.

\bibitem[AZL17b]{allen2017first}
Zeyuan Allen-Zhu and Yuanzhi Li.
\newblock First efficient convergence for streaming k-pca: a global, gap-free, and near-optimal rate.
\newblock In {\em 2017 IEEE 58th Annual Symposium on Foundations of Computer Science (FOCS)}, pages 487--492. IEEE, 2017.

\bibitem[Bac13]{bach2013sharp}
Francis Bach.
\newblock Sharp analysis of low-rank kernel matrix approximations.
\newblock In {\em Conference on learning theory}, pages 185--209. PMLR, 2013.

\bibitem[BBN19]{bollapragada2019exact}
Raghu Bollapragada, Richard~H Byrd, and Jorge Nocedal.
\newblock Exact and inexact subsampled newton methods for optimization.
\newblock {\em IMA Journal of Numerical Analysis}, 39(2):545--578, 2019.

\bibitem[BBN20]{berahas2020investigation}
Albert~S Berahas, Raghu Bollapragada, and Jorge Nocedal.
\newblock An investigation of newton-sketch and subsampled newton methods.
\newblock {\em Optimization Methods and Software}, 35(4):661--680, 2020.

\bibitem[BCNN11]{byrd2011use}
Richard~H Byrd, Gillian~M Chin, Will Neveitt, and Jorge Nocedal.
\newblock On the use of stochastic hessian information in optimization methods for machine learning.
\newblock {\em SIAM Journal on Optimization}, 21(3):977--995, 2011.

\bibitem[CPS94]{crouzeix1994davidson}
Michel Crouzeix, Bernard Philippe, and Miloud Sadkane.
\newblock The davidson method.
\newblock {\em SIAM Journal on Scientific Computing}, 15(1):62--76, 1994.

\bibitem[CW17]{clarkson2017low}
Kenneth~L Clarkson and David~P Woodruff.
\newblock Low-rank approximation and regression in input sparsity time.
\newblock {\em Journal of the ACM (JACM)}, 63(6):1--45, 2017.

\bibitem[{Dav}75]{1975JCoPh..17...87D}
Ernest~R. {Davidson}.
\newblock {The Iterative Calculation of a Few of the Lowest Eigenvalues and Corresponding Eigenvectors of Large Real-Symmetric Matrices}.
\newblock {\em Journal of Computational Physics}, 17(1):87--94, January 1975.

\bibitem[DM16]{drineas2016randnla}
Petros Drineas and Michael~W Mahoney.
\newblock Randnla: randomized numerical linear algebra.
\newblock {\em Communications of the ACM}, 59(6):80--90, 2016.

\bibitem[DMMS11]{drineas2011faster}
Petros Drineas, Michael~W Mahoney, Shan Muthukrishnan, and Tam{\'a}s Sarl{\'o}s.
\newblock Faster least squares approximation.
\newblock {\em Numerische mathematik}, 117(2):219--249, 2011.

\bibitem[DR24]{derezinski2024sharp}
Micha{\l} Derezi{\'n}ski and Elizaveta Rebrova.
\newblock Sharp analysis of sketch-and-project methods via a connection to randomized singular value decomposition.
\newblock {\em SIAM Journal on Mathematics of Data Science}, 6(1):127--153, 2024.

\bibitem[DS83]{dembo1983truncated}
Ron~S Dembo and Trond Steihaug.
\newblock Truncated-newton algorithms for large-scale unconstrained optimization.
\newblock {\em Mathematical Programming}, 26(2):190--212, 1983.

\bibitem[EM15]{erdogdu2015convergence}
Murat~A Erdogdu and Andrea Montanari.
\newblock Convergence rates of sub-sampled newton methods.
\newblock {\em Advances in Neural Information Processing Systems}, 28, 2015.

\bibitem[FTU23]{frangella2023randomized}
Zachary Frangella, Joel~A Tropp, and Madeleine Udell.
\newblock Randomized nystr{\"o}m preconditioning.
\newblock {\em SIAM Journal on Matrix Analysis and Applications}, 44(2):718--752, 2023.

\bibitem[FYXT24]{feng2024algorithm}
Xu~Feng, Wenjian Yu, Yuyang Xie, and Jie Tang.
\newblock Algorithm xxx: Faster randomized svd with dynamic shifts.
\newblock {\em ACM Transactions on Mathematical Software}, 2024.

\bibitem[GH15]{garber2015fast}
Dan Garber and Elad Hazan.
\newblock Fast and simple pca via convex optimization.
\newblock {\em arXiv preprint arXiv:1509.05647}, 2015.

\bibitem[GKLR19]{gower2019rsn}
Robert Gower, Dmitry Kovalev, Felix Lieder, and Peter Richt{\'a}rik.
\newblock Rsn: randomized subspace newton.
\newblock {\em Advances in Neural Information Processing Systems}, 32, 2019.

\bibitem[GMMN21]{gower2021adaptive}
Robert~M Gower, Denali Molitor, Jacob Moorman, and Deanna Needell.
\newblock On adaptive sketch-and-project for solving linear systems.
\newblock {\em SIAM Journal on Matrix Analysis and Applications}, 42(2):954--989, 2021.

\bibitem[Gow16]{gower2016sketch}
Robert~M Gower.
\newblock Sketch and project: Randomized iterative methods for linear systems and inverting matrices.
\newblock {\em arXiv preprint arXiv:1612.06013}, 2016.

\bibitem[Gu15]{gu2015subspace}
Ming Gu.
\newblock Subspace iteration randomization and singular value problems.
\newblock {\em SIAM Journal on Scientific Computing}, 37(3):A1139--A1173, 2015.

\bibitem[HJ19]{huang2019inner}
Jinzhi Huang and Zhongxiao Jia.
\newblock On inner iterations of jacobi--davidson type methods for large svd computations.
\newblock {\em SIAM Journal on Scientific Computing}, 41(3):A1574--A1603, 2019.

\bibitem[HMST11]{halko2011algorithm}
Nathan Halko, Per-Gunnar Martinsson, Yoel Shkolnisky, and Mark Tygert.
\newblock An algorithm for the principal component analysis of large data sets.
\newblock {\em SIAM Journal on Scientific computing}, 33(5):2580--2594, 2011.

\bibitem[HMT11]{halko2011finding}
Nathan Halko, Per-Gunnar Martinsson, and Joel~A Tropp.
\newblock Finding structure with randomness: Probabilistic algorithms for constructing approximate matrix decompositions.
\newblock {\em SIAM review}, 53(2):217--288, 2011.

\bibitem[Hoc01]{doi:10.1137/S1064827500372973}
Michiel~E. Hochstenbach.
\newblock A jacobi--davidson type svd method.
\newblock {\em SIAM Journal on Scientific Computing}, 23(2):606--628, 2001.

\bibitem[JKM{\etalchar{+}}15]{jin2015robust}
Chi Jin, Sham~M Kakade, Cameron Musco, Praneeth Netrapalli, and Aaron Sidford.
\newblock Robust shift-and-invert preconditioning: Faster and more sample efficient algorithms for eigenvector computation.
\newblock {\em arXiv preprint arXiv:1510.08896}, 2015.

\bibitem[KN03]{knyazev2003geometric}
Andrew~V Knyazev and Klaus Neymeyr.
\newblock A geometric theory for preconditioned inverse iteration iii: A short and sharp convergence estimate for generalized eigenvalue problems.
\newblock {\em Linear Algebra and its Applications}, 358(1-3):95--114, 2003.

\bibitem[KN09]{knyazev2009gradient}
Andrew~V Knyazev and Klaus Neymeyr.
\newblock Gradient flow approach to geometric convergence analysis of preconditioned eigensolvers.
\newblock {\em SIAM journal on matrix analysis and applications}, 31(2):621--628, 2009.

\bibitem[Kny98]{knyazev1998preconditioned}
Andrew~V Knyazev.
\newblock Preconditioned eigensolvers-an oxymoron?
\newblock {\em Electronic Transactions on Numerical Analysis}, 7:104--123, 1998.

\bibitem[KT24]{kireeva2024randomized}
Anastasia Kireeva and Joel~A Tropp.
\newblock Randomized matrix computations: Themes and variations.
\newblock {\em arXiv preprint arXiv:2402.17873}, 2024.

\bibitem[LP20]{lacotte2020effective}
Jonathan Lacotte and Mert Pilanci.
\newblock Effective dimension adaptive sketching methods for faster regularized least-squares optimization.
\newblock {\em Advances in neural information processing systems}, 33:19377--19387, 2020.

\bibitem[LWM{\etalchar{+}}07]{liberty2007randomized}
Edo Liberty, Franco Woolfe, Per-Gunnar Martinsson, Vladimir Rokhlin, and Mark Tygert.
\newblock Randomized algorithms for the low-rank approximation of matrices.
\newblock {\em Proceedings of the National Academy of Sciences}, 104(51):20167--20172, 2007.

\bibitem[LWZ20]{li2020subsampled}
Xiang Li, Shusen Wang, and Zhihua Zhang.
\newblock Do subsampled newton methods work for high-dimensional data?
\newblock In {\em Proceedings of the AAAI Conference on Artificial Intelligence}, volume~34, pages 4723--4730, 2020.

\bibitem[LZ15]{li2015convergence}
Ren-Cang Li and Lei-Hong Zhang.
\newblock Convergence of the block lanczos method for eigenvalue clusters.
\newblock {\em Numerische Mathematik}, 131:83--113, 2015.

\bibitem[M{\etalchar{+}}11]{mahoney2011randomized}
Michael~W Mahoney et~al.
\newblock Randomized algorithms for matrices and data.
\newblock {\em Foundations and Trends{\textregistered} in Machine Learning}, 3(2):123--224, 2011.

\bibitem[MDM{\etalchar{+}}23]{murray2023randomized}
Riley Murray, James Demmel, Michael~W Mahoney, N~Benjamin Erichson, Maksim Melnichenko, Osman~Asif Malik, Laura Grigori, Piotr Luszczek, Micha{\l} Derezi{\'n}ski, Miles~E Lopes, et~al.
\newblock Randomized numerical linear algebra: A perspective on the field with an eye to software.
\newblock {\em arXiv preprint arXiv:2302.11474}, 2023.

\bibitem[MM15]{musco2015randomized}
Cameron Musco and Christopher Musco.
\newblock Randomized block krylov methods for stronger and faster approximate singular value decomposition.
\newblock {\em Advances in neural information processing systems}, 28, 2015.

\bibitem[MRT11]{martinsson2011randomized}
Per-Gunnar Martinsson, Vladimir Rokhlin, and Mark Tygert.
\newblock A randomized algorithm for the decomposition of matrices.
\newblock {\em Applied and Computational Harmonic Analysis}, 30(1):47--68, 2011.

\bibitem[MSM14]{meng2014lsrn}
Xiangrui Meng, Michael~A Saunders, and Michael~W Mahoney.
\newblock Lsrn: A parallel iterative solver for strongly over-or underdetermined systems.
\newblock {\em SIAM Journal on Scientific Computing}, 36(2):C95--C118, 2014.

\bibitem[MT20]{martinsson2020randomized}
Per-Gunnar Martinsson and Joel~A Tropp.
\newblock Randomized numerical linear algebra: Foundations and algorithms.
\newblock {\em Acta Numerica}, 29:403--572, 2020.

\bibitem[OPA19]{ozaslan2019iterative}
Ibrahim~Kurban Ozaslan, Mert Pilanci, and Orhan Arikan.
\newblock Iterative hessian sketch with momentum.
\newblock In {\em ICASSP 2019-2019 IEEE International Conference on Acoustics, Speech and Signal Processing (ICASSP)}, pages 7470--7474. IEEE, 2019.

\bibitem[PW16]{pilanci2016iterative}
Mert Pilanci and Martin~J Wainwright.
\newblock Iterative hessian sketch: Fast and accurate solution approximation for constrained least-squares.
\newblock {\em Journal of Machine Learning Research}, 17(53):1--38, 2016.

\bibitem[PW17]{pilanci2017newton}
Mert Pilanci and Martin~J Wainwright.
\newblock Newton sketch: A near linear-time optimization algorithm with linear-quadratic convergence.
\newblock {\em SIAM Journal on Optimization}, 27(1):205--245, 2017.

\bibitem[RKM16a]{roosta2016sub}
Farbod Roosta-Khorasani and Michael~W Mahoney.
\newblock Sub-sampled newton methods i: globally convergent algorithms.
\newblock {\em arXiv preprint arXiv:1601.04737}, 2016.

\bibitem[RKM16b]{roosta2016sub2}
Farbod Roosta-Khorasani and Michael~W Mahoney.
\newblock Sub-sampled newton methods ii: Local convergence rates.
\newblock {\em arXiv preprint arXiv:1601.04738}, 2016.

\bibitem[ROW20]{royer2020newton}
Cl{\'e}ment~W Royer, Michael O’Neill, and Stephen~J Wright.
\newblock A newton-cg algorithm with complexity guarantees for smooth unconstrained optimization.
\newblock {\em Mathematical Programming}, 180:451--488, 2020.

\bibitem[RST10]{rokhlin2010randomized}
Vladimir Rokhlin, Arthur Szlam, and Mark Tygert.
\newblock A randomized algorithm for principal component analysis.
\newblock {\em SIAM Journal on Matrix Analysis and Applications}, 31(3):1100--1124, 2010.

\bibitem[RT08]{rokhlin2008fast}
Vladimir Rokhlin and Mark Tygert.
\newblock A fast randomized algorithm for overdetermined linear least-squares regression.
\newblock {\em Proceedings of the National Academy of Sciences}, 105(36):13212--13217, 2008.

\bibitem[Sar06]{sarlos2006improved}
Tamas Sarlos.
\newblock Improved approximation algorithms for large matrices via random projections.
\newblock In {\em 2006 47th annual IEEE symposium on foundations of computer science (FOCS'06)}, pages 143--152. IEEE, 2006.

\bibitem[SE02]{simoncini2002inexact}
Valeria Simoncini and Lars Eld{\'e}n.
\newblock Inexact rayleigh quotient-type methods for eigenvalue computations.
\newblock {\em BIT Numerical Mathematics}, 42(1):159--182, 2002.

\bibitem[SKT14]{szlam2014implementation}
Arthur Szlam, Yuval Kluger, and Mark Tygert.
\newblock An implementation of a randomized algorithm for principal component analysis.
\newblock {\em arXiv preprint arXiv:1412.3510}, 2014.

\bibitem[ST14]{spielman2014nearly}
Daniel~A Spielman and Shang-Hua Teng.
\newblock Nearly linear time algorithms for preconditioning and solving symmetric, diagonally dominant linear systems.
\newblock {\em SIAM Journal on Matrix Analysis and Applications}, 35(3):835--885, 2014.

\bibitem[SVdV95]{sleijpen1995jacobi}
Gerard~LG Sleijpen and Henk~A Van~der Vorst.
\newblock The jacobi-davidson method for eigenvalue problems and its relation with accelerated inexact newton schemes.
\newblock {\em Iterative Methods in Linear Algebra, II, SD Margenov and PS Vassilevski, eds., IMACS Ser. Comput. Appl. Math}, 3:377--389, 1995.

\bibitem[SW23]{swartworth2023optimal}
William Swartworth and David~P Woodruff.
\newblock Optimal eigenvalue approximation via sketching.
\newblock In {\em Proceedings of the 55th Annual ACM Symposium on Theory of Computing}, pages 145--155, 2023.

\bibitem[TB22]{trefethen2022numerical}
Lloyd~N Trefethen and David Bau.
\newblock {\em Numerical linear algebra}.
\newblock SIAM, 2022.

\bibitem[TDJS18]{tapia2018inverse}
Richard~A Tapia, John~E Dennis~Jr, and Jan~P Sch\"{a}fermeyer.
\newblock Inverse, shifted inverse, and rayleigh quotient iteration as newton's method.
\newblock {\em Siam Review}, 60(1):3--55, 2018.

\bibitem[TW23]{tropp2023randomized}
Joel~A Tropp and Robert~J Webber.
\newblock Randomized algorithms for low-rank matrix approximation: Design, analysis, and applications.
\newblock {\em arXiv preprint arXiv:2306.12418}, 2023.

\bibitem[TYUC17]{tropp2017fixed}
Joel~A Tropp, Alp Yurtsever, Madeleine Udell, and Volkan Cevher.
\newblock Fixed-rank approximation of a positive-semidefinite matrix from streaming data.
\newblock {\em Advances in Neural Information Processing Systems}, 30, 2017.

\bibitem[W{\etalchar{+}}14]{woodruff2014sketching}
David~P Woodruff et~al.
\newblock Sketching as a tool for numerical linear algebra.
\newblock {\em Foundations and Trends{\textregistered} in Theoretical Computer Science}, 10(1--2):1--157, 2014.

\bibitem[WC15]{witten2015randomized}
Rafi Witten and Emmanuel Candes.
\newblock Randomized algorithms for low-rank matrix factorizations: sharp performance bounds.
\newblock {\em Algorithmica}, 72:264--281, 2015.

\bibitem[XL24]{xu2024randomized}
Ruihan Xu and Yiping Lu.
\newblock Randomized iterative solver as iterative refinement: A simple fix towards backward stability.
\newblock {\em arXiv preprint arXiv:2410.11115}, 2024.

\bibitem[XYR{\etalchar{+}}16]{xu2016sub}
Peng Xu, Jiyan Yang, Fred Roosta, Christopher R{\'e}, and Michael~W Mahoney.
\newblock Sub-sampled newton methods with non-uniform sampling.
\newblock {\em Advances in Neural Information Processing Systems}, 29, 2016.

\bibitem[YLZ21]{ye2021approximate}
Haishan Ye, Luo Luo, and Zhihua Zhang.
\newblock Approximate newton methods.
\newblock {\em Journal of Machine Learning Research}, 22(66):1--41, 2021.

\bibitem[Zho06]{zhou2006studies}
Yunkai Zhou.
\newblock Studies on jacobi--davidson, rayleigh quotient iteration, inverse iteration generalized davidson and newton updates.
\newblock {\em Numerical Linear Algebra with Applications}, 13(8):621--642, 2006.

\end{thebibliography}

\newpage
\appendix

\section{Local convergence of Lazy ESPI}

\begin{proof}
    
Provided that we have previous $i-1$ eigenvector approximation $U\in \mathbb{R}^{n\times i-1}$, and we move update the $i$-th eigenvector. For notation simplicity, we use $u_{q}, u_{q+1}$ instead of $u^i_q,u^i_{q+1}$.

Denote $\hat{T}=\lambda_{i} I-(I-UU^\top)\hat A(I-UU^\top)$. For $(\lambda_i I-A)u_{q+1}=(\lambda_i I-A)u_{q}-(\lambda_{i} I-A)(\lambda_{i} I-(I-UU^\top)\hat A(I-UU^\top))^{-1}(\lambda_{i} I-A)u_q$, we have
\begin{align}
\label{decompose2}
    (\lambda_{i} I-\Lambda^2)V_2^\top u_{q+1} &= (\lambda_{i} I-\Lambda_2)V_2^\top u_q-(\lambda_{i} I-\Lambda_2)V_2^\top \hat{T}^{-1}(\lambda_{i} I-A)u_q\notag\\
    &= (\lambda_{i} I-\Lambda_2)V_2^\top u_q-(\lambda_{i} I-\Lambda_2)V_2^\top \hat{T}^{-1}(V_1(\lambda_{i} I-\Lambda_1)V_1^\top +V_2(\lambda_{i} I-\Lambda_2)V_2^\top )u_q\notag\\
    &= \underbrace{(\lambda_{i} I-\Lambda_2)V_2^\top u_q-(\lambda_{i} I-\Lambda_2)V_2^\top \hat{T}^{-1}V_2(\lambda_{i} I-\Lambda_2)V_2^\top u_q}_{\texttt{term}\ 1}\notag\\
    &- \underbrace{(\lambda_{i} I-\Lambda_2)V_2^\top \hat{T}^{-1}V_1(\lambda_{i} I-\Lambda_1)V_1^\top u_q}_{\texttt{term}\ 2}.
\end{align}
We aim to bound \texttt{term} 1 and \texttt{term} 2 separately in following analysis.

Firstly, we aim to show that \texttt{term} 1 can be bounded as (\ref{eq:term1results}). With assumption $\|V_{i+1:n}^\top U\|\leq \epsilon$, by \cite[Lemma B.2]{allen2016lazysvd} (for completeness we restate as Lemma \ref{tildeT}), we have $\|(I-UU^\top)A(I-UU^\top)-(I-V_{1:i-1}V_{1:i-1}^\top)A(I-V_{1:i-1}V_{1:i-1}^\top)\|\leq 13\lambda_1\epsilon$. 

\begin{lemma}[\cite{allen2016lazysvd} Lemma B.2]
\label{tildeT}
    Let \( M \) be a PSD matrix with eigenvalues \( \lambda_1 \geq \cdots \geq \lambda_d \) and the corresponding eigenvectors \( u_1, \dots, u_d \in \mathbb{R}^d \). For every \( k \geq 1 \), define 
\[
U^\perp = (u_1, \dots, u_k) \in \mathbb{R}^{d \times k} \quad \text{and} \quad U = (u_{k+1}, \dots, u_d) \in \mathbb{R}^{d \times (d-k)}.
\]
For every \( \varepsilon \in (0, \frac{1}{2}) \), let \( V_s \in \mathbb{R}^{d \times s} \) be an orthogonal matrix such that \( \| V_s^\top U \|_2 \leq \varepsilon \). Define \( Q_s \in \mathbb{R}^{d \times s} \) to be an arbitrary orthogonal basis of the column span of \( U^\perp (U^\perp)^\top V_s \). Then we have:
\[
\left\| \left( I - Q_s Q_s^\top \right) M \left( I - Q_s Q_s^\top \right) - \left( I - V_s V_s^\top \right) M \left( I - V_s V_s^\top \right) \right\|_2 \leq 13 \lambda_1 \varepsilon.
\]
\end{lemma}

This indicates
{\footnotesize
\begin{align*}
    (I-V_{1:i-1}V_{1:i-1}^\top)A(I-V_{1:i-1}V_{1:i-1}^\top)-13\lambda_1\epsilon I
    \preceq (I-UU^\top)A(I-UU^\top)
    \preceq (I-V_{1:i-1}V_{1:i-1}^\top)A(I-V_{1:i-1}V_{1:i-1}^\top)+13\lambda_1\epsilon I.
\end{align*}}

With assumption $\epsilon\leq \frac{\eta}{26(\lambda_1-\eta)}$, we can guarantee that $\hat{T}$ is positive definite (thus invertible)  via
\begin{align*}
    &\lambda_{i} I-(I-UU^\top)\hat A(I-UU^\top)
    \succeq \lambda_{i} I-(I-UU^\top)(A-\eta I)(I-UU^\top)\\
    \succeq& \lambda_{i} I-(I-V_{1:i-1}V_{1:i-1}^\top)(A-\eta I)(I-V_{1:i-1}V_{1:i-1}^\top)-13(\lambda_1-\eta)\epsilon I\\
    \succeq& \lambda_{i} I-V_{i:n}(\Lambda_{i:n}-\eta I) V_{i:n}^\top-13(\lambda_1-\eta)\epsilon I\succeq 0,
\end{align*}  
$\Lambda_{i:n}$ in last line denotes a diagonal matrix, whose contains $i$ to $n$ diagonal elements of $\Lambda$.

% which is in a form as
% \begin{align*}
%     \hat{\Sigma}_{i:n} = 
%     \begin{pmatrix}
%         0&0\\
%         0&\Sigma_{i:n,i:n}\\
%     \end{pmatrix}.
% \end{align*}
Similarly we have $\lambda_{i} I-(I-UU^\top)\hat A(I-UU^\top)\preceq \lambda_{i} I-V_{i:n}(\Lambda_{i:n}-3\eta I) V_{i:n}^\top+13(\lambda_1-\eta)\epsilon$. Thus 
\begin{align}
\label{order1}\scriptsize
    \notag&(\lambda_{i} I-(I-UU^\top)\hat A(I-UU^\top))^{-1}\succeq (\lambda_{i} I-V_{i:n}(\Lambda_{i:n}-3\eta I) V_{i:n}^\top+13(\lambda_1-\eta)\epsilon I)^{-1}\\
    &(\lambda_{i} I-(I-UU^\top)\hat A(I-UU^\top))^{-1}\preceq (\lambda_{i} I-V_{i:n}(\Lambda_{i:n}-\eta I) V_{i:n}^\top-13(\lambda_1-\eta)\epsilon I)^{-1},
\end{align}
and thus $\|\hat T^{-1}\|\leq \frac{1}{\eta-13(\lambda_1-\eta)\epsilon}\leq \frac{2}{\eta}$.

By \eqref{order1} we can bound term $(\lambda_{i} I-\Lambda_2)-(\lambda_{i} I-\Lambda_2)V_2^\top \hat{T}^{-1}V_2(\lambda_{i} I-\Lambda_2)$ by 
\begin{align*}
    &(\lambda_{i} I-\Lambda_2)-(\lambda_{i} I-\Lambda_2)V_2^\top (\lambda_{i} I-(I-UU^\top)\hat A(I-UU^\top))^{-1}V_2(\lambda_{i} I-\Lambda_2)\\
    \preceq & (\lambda_{i} I-\Lambda_2)-(\lambda_{i} I-\Lambda_2)V_2^\top(\lambda_{i} I-V_{i:n}(\Lambda_{i:n}-3\eta I) V_{i:n}^\top+13(\lambda_1-\eta)\epsilon I)^{-1}V_2(\lambda_{i} I-\Lambda_2)\\
    =&(\lambda_{i} I-\Lambda_2)-(\lambda_{i} I-\Lambda_2)((\lambda_{i}+3\eta+13(\lambda_1-\eta)\epsilon) I-\Lambda_{2})^{-1}(\lambda_{i} I-\Lambda_2)\\
    &(\lambda_{i} I-\Lambda_2)-(\lambda_{i} I-\Lambda_2)V_2^\top (\lambda_{i} I-(I-UU^\top)\hat A(I-UU^\top))^{-1}V_2(\lambda_{i} I-\Lambda_2)\\
    \succeq &(\lambda_{i} I-\Lambda_2)-(\lambda_{i} I-\Lambda_2)((\lambda_{i}+\eta-13(\lambda_1-\eta)\epsilon) I-\Lambda_{2})^{-1}(\lambda_{i} I-\Lambda_2).
\end{align*}
Note that $\|(\lambda_{i} I-\Lambda_2)-(\lambda_{i} I-\Lambda_2)((\lambda_{i}+3\eta+13(\lambda_1-\eta)\epsilon) I-\Lambda_{2})^{-1}(\lambda_{i} I-\Lambda_2)\|=\max_{\lambda_{j}\leq \lambda_{i}}(\lambda_{i}-\lambda_{j})-\frac{(\lambda_{i}-\lambda_{j})^2}{\lambda_{i}+3\eta+13(\lambda_i-\eta)\epsilon-\lambda_{j}}\leq 3\eta+13(\lambda_i-\eta)\epsilon\leq 3.5\eta$ for some small constant $c$, thus we have
\begin{align}
\label{eq:term1results}
    \|(\lambda_{i} I-\Lambda_2)-(\lambda_{i} I-\Lambda_2)V_2^\top (\lambda_{i} I-(I-UU^\top)\hat A(I-UU^\top))^{-1}V_2(\lambda_{i} I-\Lambda_2)\|\leq  3.5\eta.
\end{align}

 Secondly, we provide upper bound of \texttt{term} 2. We decompose $(\lambda_{i} I-\Lambda_2)V_2^\top \hat{T}^{-1}V_1$ by
\begin{align}
\label{decompose1}
    (\lambda_{i} I-\Lambda_2)V_2^\top \hat{T}^{-1}V_1  = & (\lambda_{i} I-\Lambda_2)V_2^\top \tilde{T}^{-1}V_1+(\lambda_{i} I-\Lambda_2)V_2^\top (\hat{T}^{-1}-\tilde{T}^{-1})V_1\notag\\
    &=\underbrace{(\lambda_{i} I-\Lambda_2)V_2^\top \tilde{T}^{-1}V_1}_{\text{using true singular vectors}}+\underbrace{(\lambda_{i} I-\Lambda_2)V_2^\top \tilde{T}^{-1}(\tilde{T}-\hat{T})\hat{T}^{-1}V_1}_{\text{error made by estimated singular vectors}},
\end{align}
where $\tilde{T}=\lambda_{i} I-(I-V_{1:i-1}V_{1:i-1}^\top)\hat A(I-V_{1:i-1}V_{1:i-1}^\top)$ which similar to the power iteration matrix $\hat T$ but using the true singular vectors $V_{1:i-1}$ rather than the estimated one $U$.  Denote $T = \lambda_{i} I-(I-V_{1:i-1}V_{1:i-1}^\top)(A-2\eta I)(I-V_{1:i-1}V_{1:i-1}^\top)$, then the first term of \eqref{decompose1} becomes
\begin{align*}
    (\lambda_{i} I-\Lambda_2)V_2^\top \tilde{T}^{-1}V_1&=(\lambda_{i} I-\Lambda_2)V_2^\top (\tilde{T}^{-1}-T^{-1})V_1= (\lambda_{i} I-\Lambda_2)V_2^\top \tilde{T}^{-1}(T-\tilde{T})T^{-1}V_1,
\end{align*}
for the eigenvector of $T^{-1}$ is exactly that of $A$, thus $V_2^\top T^{-1}V_1=0$. 

Recall that $(I-V_{1:i-1}V_{1:i-1}^\top)\hat A(I-V_{1:i-1}V_{1:i-1}^\top)$ is a perturbation of $ (I-V_{1:i-1}V_{1:i-1}^\top)(A-2\eta I)(I-V_{1:i-1}V_{1:i-1}^\top)$, so the $k-i+2$ to $n-i+1$ eigenvectors of $\tilde{T}$ is approximately close to $V_2$. We rearrange the sequence of eigenvalue of $\tilde T$ and $T$, such that $\tilde{T}$ has SVD $\tilde{T} = U_T\Lambda_TU_T^\top$, where 
\begin{align*}
\Lambda_T=\Lambda(\tilde T)\approx \Lambda(T)= diag(\lambda_i-\lambda_i+2\eta,\lambda_i-\lambda_{i-1}+2\eta,\cdots,\lambda_i-\lambda_n+2\eta,\underbrace{\lambda_i,\cdots,\lambda_i}_{i-1})  .
\end{align*}
This means the $k-i+2$-th to $n-i+1$-th eigenvector of $T$ is exactly $V_2$ and then $(U_T)_{k-i+2:n-i+1}\approx V_2$.  Consequently $(\lambda_{i} I-\Lambda_2)V_2^\top\tilde{T}^{-1} U_T = (\lambda_{i} I-\Lambda_2)V_2^\top U_T\Lambda_T^{-1}$ can be decomposed as 
\begin{align}
\label{decomT}
    (\lambda_{i} I-\Lambda_2)V_2^\top U_T\Lambda_T^{-1}=(\lambda_{i} I-\Lambda_2)V_2^\top (U_{T,1}U_{T,1}^\top+U_{T,2}U_{T,2}^\top)U_T\Lambda_T^{-1},
\end{align}
 where $U_{T,1}$ is $(U_T)_{k-i+2:n-i+1}$ and $U_{T,2}$ is the rest $k$ columns of $U_T$. 
 
 We first give the bound of $(\lambda_{i} I-\Lambda_2)V_2^\top U_{T,1}U_{T,1}^\top U_T\Lambda_T^{-1}$ in \eqref{decomT}. Note that $\tilde{T}\succeq \lambda_i I-(I-V_{1:i-1}V_{1:i-1}^\top)(A-\eta I)(I-V_{1:i-1}V_{1:i-1}^\top)$ yields $\Lambda_T^{-1}(j,j)\leq \frac{1}{\lambda_i-\lambda_{i+j-1}+\eta}$ for $j\leq n-i+1$. Thus 
\begin{align}
\label{decompT1}
    \|(\lambda_{i} I-\Lambda_2)V_2^\top U_{T,1}U_{T,1}^\top U_T\Lambda_T^{-1}\|_F^2&\leq \sum_{s=1}^{n-k}(\sum_{t=1}^{k-i+1} (\frac{\lambda_{i}-\lambda_{k+s}}{\lambda_{i}-\lambda_{i+t-1}+\eta} v_s^\top u_t)^2+\sum_{t=n-i+2}^{n} (\frac{\lambda_{i}-\lambda_{k+s}}{\lambda_{i}} v_s^\top u_t)^2),
\end{align}
where $v_s$ is s-th eigenvector of $V_2$ and $u_t$ is t-th eigenvector of $U_T$. 

To bound $v_s^\top u_t$, we are using the fact that $(I-V_{1:i-1}V_{1:i-1}^\top)\hat A(I-V_{1:i-1}V_{1:i-1}^\top)$ is a perturbation of $ (I-V_{1:i-1}V_{1:i-1}^\top)(A-2\eta I)(I-V_{1:i-1}V_{1:i-1}^\top)$, whose eigenvectors are exactly $v_s$. By Davis-kahan theorem, when $k+1>i+t-1$ and thus $\lambda_{k+s}< \lambda_{i+t-1}$, for any $l>0$ we have
\begin{align}
\label{DK}
    \|(V_2)_{s:s+l}^\top u_t\|& \leq \frac{\|V_{i:n}V_{i:n}^\top\hat AV_{i:n}V_{i:n}^\top-V_{i:n} (\Lambda_{i:n}-2\eta I)V_{i:n}^\top\|}{gap}\notag\\
    &\leq \frac{\eta }{\lambda_{i+t-1}-3\eta-(\lambda_{k+s}-2\eta)}= \frac{\eta }{\lambda_{i+t-1}-\lambda_{k+s}-\eta},
\end{align}
where the second inequality comes form the fact that 
$-\eta I\preceq V_{i:n}^\top\hat AV_{i:n}-(\Lambda_{i:n}-2\eta I)\preceq \eta I$, and $gap\geq \lambda_{i+t-1}-3\eta-(\lambda_{k+s}-2\eta)$, because $v_s$'s corresponding eigenvalue is $\lambda_{k+s}-2\eta$ and $u_t$'s corresponding eigenvalue satisfies $\lambda_{i+t-1}-3\eta \leq \lambda\leq \lambda_{i+t-1}-\eta$.

Inequality \eqref{DK} indicates $\sum_{s=s_0}^{n-k}(v_s^\top u_t)^2\leq (\frac{\eta }{\lambda_{i+t-1}-\lambda_{k+s_0}-\eta})^2$. Then the problem 
$$\max_{v_s}\sum_{s=1}^{n-k} (\frac{\lambda_{i}-\lambda_{k+s}}{\lambda_{i}-\lambda_{i+t-1}+\eta} v_s^\top u_t)^2,\quad \text{subject to}\ \sum_{s=s_0}^{n-k}(v_s^\top u_t)^2\leq (\frac{\eta }{\lambda_{i+t-1}-\lambda_{k+s_0}-\eta})^2, \text{$\forall 1\leq s_0\leq n-k$}$$ 
can be viewed as a linear optimization $\max_{x_s}\sum_sa_sx_s$ with restriction $\sum_{s=s_0}^{n-k} x_s\leq b_{s_0}$, where $x_s = (v_s^\top u_t)^2$, $a_s=(\frac{\lambda_{i}-\lambda_{k+s}}{\lambda_{i}-\lambda_{i+t-1}+\eta})^2$ and $b_{s_0} = (\frac{\eta }{\lambda_{i+t-1}-\lambda_{k+s_0}-\eta})^2$. Since $a_s$ is ascending and $b_{s_0}$ is descending, the solution of the linear optimization problem is $x_j=b_j-b_{j+1}$ for $j\neq n$ and $x_n=b_n$. Thus we have
\begin{align*}
    &\sum_{t=1}^{k-i}\sum_s(\frac{\lambda_{i}-\lambda_{k+s}}{\lambda_{i}-\lambda_{i+t-1}+\eta} v_s^\top u_t)^2\\
    \leq& \sum_{t=1}^{k-i}(\frac{\eta}{\lambda_{i}-\lambda_{i+t-1}+\eta})^2(\sum_{s=1}^{n-k-1}(\lambda_{i}-\lambda_{k+s})^2(\frac{1}{(\lambda_{i+t-1}-\lambda_{k+s}-\eta)^2}-\frac{1}{(\lambda_{i+t-1}-\lambda_{k+s+1}-\eta)^2})).\\
\end{align*}

Note that the above summation $\sum_{s=1}^{n-k-1}(\lambda_{i}-\lambda_{k+s})^2(\frac{1}{(\lambda_{i+t-1}-\lambda_{k+s}-\eta)^2}-\frac{1}{(\lambda_{i+t-1}-\lambda_{k+s+1}-\eta)^2})$ can be viewed as the discretization of integral $\int_{\lambda_{i+t-1}-\lambda_{n}-\eta}^{\lambda_{i+t-1}-\lambda_{k+1}-\eta}(x+\lambda_{i}-\lambda_{t+i-1}+\eta)^2d\frac{1}{x^2}$, and the integral is actually larger. Thus a further analysis gives
\begin{align}
\label{longterm}
\scriptsize
    \notag&\sum_{t=1}^{k-i}(\frac{\eta}{\lambda_{i}-\lambda_{i+t-1}+\eta})^2(\sum_{s=1}^{n-k-1}(\lambda_{i}-\lambda_{k+s})^2(\frac{1}{(\lambda_{i+t-1}-\lambda_{k+s}-\eta)^2}-\frac{1}{(\lambda_{i+t-1}-\lambda_{k+s+1}-\eta)^2}))\\\notag
    \leq& \sum_{t=1}^{k-i}(\frac{\eta}{\lambda_{i}-\lambda_{i+t-1}+\eta})^2\int_{\lambda_{i+t-1}-\lambda_{n}-\eta}^{\lambda_{i+t-1}-\lambda_{k+1}-\eta}(x+\lambda_{i}-\lambda_{t+i-1}+\eta)^2d\frac{1}{x^2}\\\notag
    \leq& \sum_{t=1}^{k-i}(\frac{\eta}{\lambda_{i}-\lambda_{i+t-1}+\eta})^2(2\log(\frac{\lambda_{i+t-1}-\lambda_n-\eta}{\lambda_{i+t-1}-\lambda_{k+1}-\eta}))\\  \notag  
    + &\sum_{t=1}^{k-i}(\frac{\eta}{\lambda_{i}-\lambda_{i+t-1}+\eta})^2\biggl(4(\frac{\lambda_{i}-\lambda_{i+t-1}+\eta}{\lambda_{i+t-1}-\lambda_{k+1}-\eta}-\frac{\lambda_{i}-\lambda_{i+t-1}+\eta}{\lambda_{i+t-1}-\lambda_{n}-\eta})\\\notag
    +&((\frac{\lambda_{i}-\lambda_{i+t-1}+\eta}{\lambda_{i+t-1}-\lambda_{k+1}-\eta})^2-(\frac{\lambda_{i}-\lambda_{i+t-1}+\eta}{\lambda_{i+t-1}-\lambda_{n}-\eta})^2)\biggr)\\
    \leq& (k-i+1)(2\log(\frac{\lambda_{i}-\lambda_n-\eta}{\lambda_{i}-\lambda_{k+1}-\eta})+4(\frac{\eta}{\lambda_{k}-\lambda_{k+1}-\eta}-\frac{\eta}{\lambda_{k}})+(\frac{\eta}{\lambda_{k}-\lambda_{k+1}-\eta})^2-(\frac{\eta}{\lambda_{k}})^2)\\\notag
    \leq& k(2\log(\frac{\lambda_{i}-\lambda_n-\eta}{\lambda_{i}-\lambda_{k+1}-\eta})+c_{gap,\eta}),
\end{align}
Where $c_{gap,\eta}=c\frac{\eta}{\lambda_k-\lambda_{k+1}}$ for some small constant $c$, which bound the last two term in \eqref{longterm} since $constant \cdot \eta<\lambda_k-\lambda_{k+1}$ for some constant. The third inequality comes from the fact that
\begin{align*}
    \max_{\lambda_{i}\geq \lambda_{i+t-1}\geq\lambda_{k}\geq(1+\eta)\lambda_k}(\frac{\eta}{\lambda_{i}-\lambda_{i+t-1}+\eta})^2\log(\frac{\lambda_{i+t-1}-\lambda_n-\eta}{\lambda_{i+t-1}-\lambda_{k+1}-\eta})\leq \log(\frac{\lambda_{i}-\lambda_n-\eta}{\lambda_{i}-\lambda_{k+1}-\eta}).
\end{align*}

Back to \eqref{decompT1}, with $\sum_{t=n-i+1}^{n-1}\sum_s(\frac{\lambda_{i}-\lambda_{k+s}}{\lambda_{i}} v_s^\top u_t)^2\leq (i-1)\frac{\lambda_{i}-\lambda_{k+s}}{\lambda_{i}}\leq i-1$ we have $\|(\lambda_{i} I-\Lambda_2)V_2^\top U_{T,1}U_{T,1}^\top U_T\Lambda_T\|\leq \sqrt{2k(\log(\frac{(1-\eta)\lambda_{i}}{(1-\eta)\lambda_{i}-\lambda_{k+1}})+c_{gap,\eta})}$ with a little abuse of notation $c_{gap,\eta}$, which differs from previous ones with only $\frac{i-1}{2k}$.

Back to \eqref{decomT} we now have
\begin{align}
\label{pre1}
    \|(\lambda_{i} I-\Lambda_2)V_2^\top U_T\Lambda_T\|&\leq \|(\lambda_{i} I-\Lambda_2)V_2^\top U_{T,1}U_{T,1}^\top U_T\Lambda_T\|+\|(\lambda_{i} I-\Lambda_2)V_2^\top U_{T,2}U_{T,2}^\top U_T\Lambda_T\|\notag\\
    &\leq \sqrt{2k(\log(\frac{\lambda_{i}-\lambda_n-\eta}{\lambda_{i}-\lambda_{k+1}-\eta})+c_{gap,\eta})}+\frac{\lambda_{i}-\lambda_n}{\lambda_{i}-\lambda_{k+1}+\eta}\notag\\
    &\leq c_1\frac{\lambda_{i}-\lambda_n}{\lambda_{i}-\lambda_{k+1}},
\end{align}
where $c_1$ is a small constant that depends on $\sqrt{k}$.

Finally combining \eqref{decompose1} with \eqref{pre1}, and recall that we have proved $\|\hat T^{-1}\|\leq \frac{2}{\eta}$ and $\|\hat T-\tilde T\|\leq 13\lambda_1\epsilon$, we get 
\begin{align*}
    \|(\lambda_{i} I-\Lambda_2)V_2^\top \hat{T}^{-1}V_1\|&\leq  \|(\lambda_{i} I-\Lambda_2)V_2^\top \tilde{T}^{-1}V_1\|
    +\|(\lambda_{i} I-\Lambda_2)V_2^\top (\hat{T}^{-1}-\tilde{T}^{-1})V_1\|\\
    & =\|(\lambda_{i} I-\Lambda_2)V_2^\top \tilde{T}^{-1}V_1\|
    +\|(\lambda_{i} I-\Lambda_2)V_2^\top \tilde{T}^{-1}(\tilde{T}-\hat{T})\hat{T}^{-1}V_1\|\\
    & \leq  \frac{c_1 (\lambda_i-\lambda_n)}{\lambda_{i}-\lambda_{k+1}}+\|(\lambda_{i} I-\Lambda_2)V_2^\top \tilde{T}^{-1}\|\|(\tilde{T}-\hat{T})\hat{T}^{-1}V_1\|\\
    & \leq  \frac{c_1 (\lambda_i-\lambda_n)}{\lambda_{i}-\lambda_{k+1}}+\frac{c_1 (\lambda_i-\lambda_n)}{\lambda_{i}-\lambda_{k+1}}\frac{26\lambda_1\epsilon }{\eta}.
\end{align*}

With assumption $\epsilon\leq \frac{\eta}{26(\lambda_1-\eta)}$ and bound of term 1 in \eqref{eq:term1results}, combining two terms in \eqref{decompose2} gives
\begin{align*}
    \|V_2^\top (\lambda_i I-\Lambda^2)u_{q+1}\|&\leq 3.5\eta\|V_2^\top u_q\|+\frac{c_1(\lambda_i-\lambda_n)}{\lambda_i-\lambda_{k+1}}\|(\lambda_i I-\Lambda_1)V_1^\top u_q\|\\
    \|V_2^\top u_{q+1}\|&\leq \frac{3.5\eta}{\lambda_i-\lambda_{k+1}}\|V_2^\top u_q\|+\frac{c_1(\lambda_i-\lambda_n)}{(\lambda_i-\lambda_{k+1})^2}\|(\lambda_i I-\Lambda_1)V_1^\top u_q\|.
\end{align*}

Similar to proof in 1-svd, we bound $\|u_{q+1}\|$ by
\begin{align*}
    \|u_{q+1}\| &= \|u_q+\hat T^{-1}(A-\lambda_i I)u_q\|\\
    &\geq \|u_q\|-\|\hat T^{-1} \|\| (\Lambda_1-\lambda_iI)V_1^\top u_q+(\Lambda_2-\lambda_i I)V_2^\top u_q\|
\end{align*}
Denote $U_q = [u_q^1,u_q^2,\cdots,u_q^k]$, then by lemma \ref{correction} we have $\|(\Lambda_1-\lambda_iI)V_1^\top u_q\|\leq 17\lambda_1\|V_2^\top U_q\|^2$. And $\|(\Lambda_2-\lambda_i I)V_2^\top u_q\|\leq (\lambda_i-\lambda_{k+1})\|V_2^\top u_q\|$. Thus $\|u_{q+1}\|\geq 1-\frac{2(\lambda_i-\lambda_{k+1})}{\eta}\|V_2^\top u_q\|-\frac{34\lambda_1}{\eta}\|V_2^\top U_q\|^2$ for some constant $c$ suppose that $u_q$ is normalized such that $\|u_q\|=1$. 

Thus if $\epsilon_0 = max\{\frac{4(\lambda_i-\lambda_{k+1})}{\eta}\|V_2^\top u_q\|,\frac{68\lambda_1}{\eta}\|V_2^\top U_q\|^2\}$, we have
\begin{align*}
    \frac{\|V_2^\top u_{q+1}\|}{\|u_{q+1}\|}\leq \frac{1}{1-\epsilon_0}( \frac{3.5\eta}{\lambda_i-\lambda_{k+1}}\frac{\|V_2^\top u_q\|}{\|u_q\|}+\frac{c_1(\lambda_i-\lambda_n)}{(\lambda_i-\lambda_{k+1})^2}\frac{\|(\lambda_i I-\Lambda_1)V_1^\top u_q\|}{\|u_q\|}).
\end{align*}
\end{proof}

 \section{Proof of Lemmas \ref{correction}}

 \begin{proof}
    
Let $B=UU^\top AUU^\top$, then $Bu^i=\hat{\lambda}_iu^i$. We have
    \begin{align*}
        (\hat\lambda_i-\Lambda_1)V_1^\top u^i&=V_1^\top B u^i-\Lambda_1V_1^\top u^i\\
        &=(V_1^\top BV_1-\Lambda_1)V_1^\top u^i+V_1^\top BV_2V_2^\top u^i\\
        &=(V_1^\top BV_1-\Lambda_1)e_i+(V_1^\top BV_1-\Lambda_1)(V_1^\top u^i-e_i)+V_1^\top BV_2V_2^\top u^i\\
        (\lambda_i-\Lambda_1)V_1^\top u^i &= (V_1^\top BV_1-\Lambda_1)e_i+(V_1^\top BV_1-\Lambda_1)(V_1^\top u^i-e_i)+V_1^\top BV_2V_2^\top u^i-(\hat\lambda_i-\lambda_i)V_1^\top u^i.\\
    \end{align*}
    Thus we have 
    {\footnotesize
    \begin{align}
    \label{Verror}
        \|(\lambda_i-\Lambda_1)V_1^\top u^i\|\leq \|(V_1^\top BV_1-\Lambda_1)e_i\|+\|(V_1^\top BV_1-\Lambda_1)(V_1^\top u^i-e_i)\|+\|V_1^\top BV_2V_2^\top u^i\|+\|(\hat\lambda_i-\lambda_i)V_1^\top u^i\|.
    \end{align}}

    We aim to show that $\|V_1^\top BV_1-\Lambda_1\|$ is of quadratic order on $\|V_2^\top U\|$. First we decompose it into two parts
   \begin{align}
   \label{VBV}
       \notag\|V_1^\top B V_1-\Lambda_1\|&=\|V_1^\top UU^\top V_1^\top \Lambda_1V_1UU^\top V_1-\Lambda_1+V_1^\top UU^\top V_2^\top \Lambda_2V_2UU^\top V_1\|\\
       &\leq \|V_1^\top UU^\top V_1^\top \Lambda_1V_1UU^\top V_1-\Lambda_1\|+\|V_1^\top UU^\top V_2^\top \Lambda_2V_2UU^\top V_1\|.
   \end{align}

   Denote $E:=V_1^\top UU^\top V_1-I$. Note that $\sigma_{min}(V_1^\top UU^\top V_1)=\sigma_{min}(V_1^\top U)^2\geq 1-\|V_2^\top U\|^2$ since $\|V_1^\top Ux\|^2+\|V_2^\top Ux\|^2=\|x\|^2$. Thus $-\|V_2^\top U\|^2 I\preceq E\preceq 0 $. Then for the first term in \eqref{VBV} we have
   \begin{align*}
       \|V_1^\top UU^\top V_1^\top \Lambda_1V_1UU^\top V_1-\Lambda_1\|&=\|(I+E)\Lambda_1(I+E)-\Lambda_1\|\\
       &=\|E\Lambda_1+\Lambda_1E+E\Lambda_1E\|\\
       &\leq 3\lambda_1\|E\|\leq3\lambda_1\|V_2^\top U\|^2.
   \end{align*}

   For the second term in \eqref{VBV} we have
   \begin{align*}
       \|V_1^\top UU^\top V_2^\top \Lambda_2V_2UU^\top V_1\|\leq\lambda_{k+1}\|V_2^\top U\|^2.
   \end{align*}
    Thus $\|V_1^\top B V_1-\Lambda_1\|\leq4\lambda_1\|V_2^\top U\|^2$.

    For the last two terms in \eqref{Verror}, note that $\|V_1^\top BV_2V_2^\top u^i\|\leq \lambda_1\|U^\top V_2\|\|V_2u^i\|\leq\|V_2^\top U\|^2$ and $|\lambda_i-\hat \lambda_i|\leq \|V_1^\top BV_1-\Lambda_1\|$ by Weyl's theorem, thus we finally achieve
    \begin{align*}
        \|(\lambda_i-\Lambda_1)V_1^\top u^i\|&\leq\|V_1^\top B V_1-\Lambda_1\|+\|V_1^\top B V_1-\Lambda_1\|\|V_1^\top u^i-e_1\|+\|V_2^\top U\|^2+\|V_1^\top BV_1\|\|V_1^\top u^i\|\\
        &\leq 17\lambda_1\|V_2^\top U\|^2.
    \end{align*}

\end{proof}

\end{document}